\documentclass[11pt,a4paper]{article}

\usepackage[OT1]{fontenc}
\usepackage{lmodern}
\usepackage[english]{babel}
\usepackage[utf8x]{inputenc}
\usepackage[sc]{mathpazo}
\usepackage{amsmath,amssymb,amsfonts,mathrsfs,amstext, bbm}
\usepackage[thmmarks]{ntheorem}
\usepackage{wrapfig}
\usepackage{nicefrac}
\usepackage{tabularx}

\usepackage{varioref}

\usepackage{datetime}

\usepackage{mathtools}

\usepackage{array}

\usepackage{listings}
\usepackage[linkcolor=black,colorlinks=true,citecolor=black,filecolor=black]{hyperref}

\usepackage{algpseudocode}
\usepackage{algorithm}

\usepackage{color}
\usepackage{caption}
\usepackage{subcaption}
\usepackage[noabbrev]{cleveref}
\usepackage{csquotes}
\usepackage{bm}
\usepackage{mathtools}
\usepackage{booktabs}

\numberwithin{equation}{section}

\newtheorem{theorem}{Theorem}[section]

\newtheorem{remark}[theorem]{Remark}

\newtheorem{corollary}[theorem]{Corollary}

\newtheorem{lemma}[theorem]{Lemma}
\newtheorem{proposition}[theorem]{Proposition}

\theoremstyle{nonumberplain}
\theorembodyfont{\normalfont}
\theoremsymbol{\ensuremath{\blacksquare}}
\newtheorem{proof}{Proof}

\newcommand\mathmidscript[1]{\vcenter{\hbox{$\scriptstyle #1$}}}

\newcommand{\pvint}{\ensuremath{\;\mathmidscript{\raisebox{-0.9ex}{\tiny{p.\,v.\;}}}\!\!\!\!\!\!\!\!\int}}

\DeclareMathAlphabet{\mathpzc}{OT1}{pzc}{m}{it}

\newcommand*\Laplace{\mathop{}\!\mathbin\bigtriangleup}

\newcommand{\NORM}[1]{\left\lVert#1\right\rVert} 

\newcommand{\DEF}{\coloneqq}
\newcommand{\FED}{\eqqcolon}
\newcommand{\RR}{\mathbb{R}}
\newcommand{\CC}{\mathbb{C}}
\newcommand{\NN}{\mathbb{N}}

\newcommand{\Om}{\Omega}

\newcommand{\del}{\partial}

\newcommand{\OO}{\mathcal{O}}

\newcommand{\intd}{\mathrm{d}}

\newcommand{\TransT}{\mathrm{T}}

\makeatletter
\def\moverlay{\mathpalette\mov@rlay}
\def\mov@rlay#1#2{\leavevmode\vtop{%
   \baselineskip\z@skip \lineskiplimit-\maxdimen
   \ialign{\hfil$\m@th#1##$\hfil\cr#2\crcr}}}
\newcommand{\charfusion}[3][\mathord]{
    #1{\ifx#1\mathop\vphantom{#2}\fi
        \mathpalette\mov@rlay{#2\cr#3}
      }
    \ifx#1\mathop\expandafter\displaylimits\fi}
\makeatother

\newcommand*\euler{\mathrm{e}}
\newcommand*\imagi{\mathrm{i}\,}

\title{A Steklov-spectral approach for solutions of  Dirichlet and Robin boundary value problems}
\date{}
\author{
Kthim Imeri \thanks{\footnotesize Department of Mathematics, Simon Fraser University, 8888 University Dr, Burnaby, BC V5A 1S6, Canada (kthim\_imeri@sfu.ca).} 
\and Nilima Nigam \thanks{\footnotesize Department of Mathematics, Simon Fraser University, 8888 University Dr, Burnaby, BC V5A 1S6, Canada (nigam@math.sfu.ca).}
}

\begin{document}
	\maketitle

\begin{abstract}
	In this paper we revisit an approach pioneered by Auchmuty \cite{AuchmutyFoundation} to approximate solutions of the Laplace- Robin boundary value problem. We demonstrate the efficacy of this approach on a large class of non-tensorial domains, in contrast with other spectral approaches for such problems.

	We establish a spectral approximation theorem showing an exponential fast numerical evaluation with regards to the number of Steklov eigenfunctions used, for smooth domains and smooth boundary data. A polynomial fast numerical evaluation is observed for either non-smooth domains or non-smooth boundary data. We additionally prove a new result on the regularity of the Steklov eigenfunctions, depending on the regularity of the domain boundary.
	
	We describe three numerical methods  to compute Steklov eigenfunctions.
	
\end{abstract}

\def\keywords2{\vspace{.5em}{\textbf{Mathematics Subject Classification
(MSC 2000).}~\,\relax}}
\def\endkeywords2{\par}
\keywords2{35R30, 35C20.}

\def\keywords{\vspace{.5em}{\textbf{ Keywords.}~\,\relax}}
\def\endkeywords{\par}
\keywords{Steklov eigenvalue problem, boundary integral operators, Dirichlet boundary conditions, Robin boundary conditions, numerical solution to the Laplace problem.}


\section{Introduction}\label{Ch:Introduction}
Spectral methods provide a fast and highly accurate approximation strategy for the solution of partial differential equations (PDE), and rely on the use of an approximation basis consisting of eigenfunctions of the PDE operator under consideration. These eigenfunctions are {\it typically} the Dirichlet or Neumann eigenfunctions of the PDE operator, and provide excellent approximations for homogenous boundary value problems. Non-homogenous Dirichlet and Neumann boundary value problems necessitate the use of additional techniques.

However, the use of spectral methods to solve boundary value problems for the Laplacian with {\it inhomogenous Robin data} has thus far been restricted to tensorial geometries, see e.g. \cite{wang}. Robin boundary value problems provide challenges for approximation via the 'standard' basis functions, since these are not eigenfunctions of the associated solution operator.

An ingenous idea was first proposed by \cite{AuchmutyFoundation}, who demonstrated the theoretical foundations for spectral approximation via {\it Steklov eigenfunctions} of elliptic operators. This idea was then explored in \cite{AuchmutyCho} to approximate the solutions of the Laplacian on a domain $\Om$ with Dirichlet, Neumann and Robin boundary conditions. The Steklov eigenfunctions are denoted in this work as $(s_i)_{i=1}^\infty$, which are harmonic functions, satisfying the Steklov boundary condition, that is
\begin{align*}
	\del_\nu s_i = \lambda s_i \quad\text{on }\del\Om\,,
\end{align*}
where $\lambda\geq 0$ and $\del_\nu$ is the normal derivative on the boundary. In \cite{AuchmutyFoundation} it is shown that there are countable many such functions to countable many $\lambda$, which satisfy the above boundary condition, and the pair $(s_i, \lambda_i)$ is called the $i$-th Steklov eigenfunction and Steklov eigenvalue.

The results in \cite{AuchmutyFoundation} moreover show that the eigenfunctions build a maximal orthonormal set in $L^2(\del\Om)$ for $\Om\subset\RR^n$, $n\geq 2$, and that we can build an isometric isomorphism between $H^{1/2}(\del\Om)$ and a subspace of harmonic function in $H^1(\Om)$. An exceptional behaviour is observed for the Steklov eigenvalue in regards to the regularity of the boundary. For any Lipschitz boundary the eigenvalues ascend linearly, then for smooth boundaries the eigenvalues ascend linearly in pairs with an exponential decrease in the distance within a pair, and for non-smooth boundaries this distance might not decrease exponentially any further. Through \cite{Agranovich2006} we have a general formula for the eigenvalues in any Lipschitz domain, and on that foundation the authors in \cite{CurviLinearDomainEVals} discovered a more accurate formula for curvilinear polygons. 
Similarly, in Proposition \ref{prop:s_i in H^p}  we will show that the regularity of Steklov eigenfunctions depend on the regularity of the boundary. Steklov eigenpairs are also useful in more general partial differential equations, the authors in \cite{AuchmutySchrodinger} considered Schroedinger-type operators and the authors in \cite{oursteklovpaper} consider the Helmholtz equation with mixed boundary boundary conditions. The survey \cite{girouard2017spectral} states more such applications as well as more results on the Steklov problem.

This paper is organized as follows. In Section \ref{sec:Prelim} we first define the Steklov problem in $\RR^2$ and show asymptotics for the eigenvalues, regarding the smoothness of the boundary, and the orthonormality of the eigenfunctions. Subsequently, we consider two higher order layer potentials, that is the single layer potential and Neumann-Poincaré operator, and their mapping properties in $H^p$. Need those for the numerics at the end and the next result, Proposition \ref{prop:s_i in H^p}, which shows higher regularity properties for $s_i$. Thereafter, we examine properties on the Dirichlet to Neumann operator, whose eigenvalues coincide with the Steklov eigenvalue. Then we estimate the norm of the Dirichlet to Neumann operator, in Lemma \ref{lemma:normDN<normg invf}, and with that prove that every smooth function on the boundary with inform $L^2$-bounded derivatives is a finite linear combination of Steklov eigenfunctions, in Proposition \ref{prop:uniformL2boundedFcts}. We conclude Section \ref{sec:Prelim} with the formulation of the series expansion for the solution to the Laplace problem in terms the Steklov eigenpairs, theoretical as well as the numerical approximation.

In Section \ref{sec:main} we show the main results, which concerns the asymptotic approximation of the series expansion in terms if the amount of eigenpairs used, as well as the asymptotics in terms of the numerical approximation of the Steklov eigenfunctions. For the Dirichlet boundary problem this is stated in Theorem \ref{thm:main}, for the Robin boundary problem, the asymptotics are given in Corollary \ref{coro:main}. Thereafter we proof the theorem and the corollary.

In Section \ref{sec:Numerics} we first consider three methods to numerically compute the eigenpairs. The first method is based on the results of \cite{ConfNumImpl}, which provides the eigenvalues through a conformal map. We further develop it to also obtain the eigenfunctions. The second method relies on the weak form of the Steklov problem and uses a particular orthonormal basis to obtain a generalized eigenvalue problem from which we obtain the eigenpair. The third method uses the collocation method elaborated in \cite[Chapter 13]{kress2013linear}. Method three is then used in the upcoming subsection, which is about numerical tests to validate Theorem \ref{thm:main} and Corollary \ref{coro:main}. We apply several different domains and several boundary data with various regularity and visually present the results obtained in MATLAB \cite{MATLAB}.

\section{Preliminaries}\label{sec:Prelim}

Let $\Omega\subset\RR^2$ be a simply connected, bounded and open domain with a Lipschitz boundary. More restrictive boundary conditions may be given later on. Then we define $u$ to be the solution to the Helmholtz equation, that is
\begin{align}\label{PDE:Helmholtz}
	\Laplace u(x) = 0 \quad \text{in } \Om\,, 
\end{align}
with the 3 separate boundary conditions
\begin{align}
	u(y)  &= g(y)  \,,\label{PDE:Helmholtz:Dir}\\
	\del_{\nu_y} u(y)  &= g (y)  \,,\label{PDE:Helmholtz:Neu}\\
	\del_{\nu_y} u(y)+b\, u(y)  &= g (y)  \,,\label{PDE:Helmholtz:Rob}
\end{align}
where $g\in L^2(\del \Om)$, $b>0$, and $\del_{\nu_y}$ denotes the outside normal derivative on $\del \Omega$ at $y$. The first condition is known as the Dirichlet condition, the second condition is known as the Neumann condition and third one is known as the Robin boundary condition.  

We define the Steklov eigenvalues $\lambda_0,\lambda_1,\ldots $ and respective Steklov eigenfunction $s_0, s_1,\ldots$ as the solution to 
\begin{align}\label{PDE:Steklov}
	\left\{ 
	\begin{aligned}
		 \Laplace  \, s_j(x) &= 0 \quad &&\text{in } \Om\,, \\
		 \del_{\nu_y} s_j(y) &= \lambda_j \, s_j(y) \quad &&\text{on } \del \Omega \,,
	\end{aligned}
	\right.
\end{align}
for $j\in\NN_0$. Equation (\ref{PDE:Steklov}) has a non-trivial solution $s_i\in H^1(\Om)$ to a corresponding eigenvalue $\lambda_i$. From \cite[Corollary 4.3]{Daners2009} we have that $s_i\in C^\infty(\Om)\cap C^0(\overline{\Om})$. From \cite[Section 6 and 7]{AuchmutyFoundation} we have that all $\lambda_i$ have finite multiplicity. We also know that $\lambda_0 = 0$ and that $s_0$ is constant. For smooth enough boundaries the authors in \cite{StekSpectrumAnalysisGirouard} show that
\begin{align}\label{equ:lambdaAsymp}
	\lambda_{2i} = \lambda_{2i-1}+\OO(i^{-\infty}) = \frac{2\pi}{|\del\Om|}\,i+\OO(i^{-\infty})\,,
\end{align}
for $i\geq 1$, where $\OO(i^{-\infty})$ decays faster than any power of $i^{-1}$ for $i\rightarrow \infty$, where $|\del\Om|$ denotes the length of the boundary. From \cite{Agranovich2006} we further know that for piecewise $C^1$ boundaries we have that
\begin{align}\label{equ:lambdaAsymp-C1}
	\lambda_{i} = \Big(\frac{\pi}{|\del\Om|}+ o(1) \Big)\,i\,.
\end{align}
For curvilinear domains the authors in \cite{CurviLinearDomainEVals} are able to construct values closer to the eigenvalues than what is given with Equation (\ref{equ:lambdaAsymp-C1}).

We define the inner-product $\langle \cdot\,,\cdot\rangle_\del$ through
\begin{align*}
	\langle v,w\rangle_\del \DEF \int_{\Om}\nabla v\cdot\nabla w\;+\,\int_{\del\Om} v\, w\,,
\end{align*}
for $v,w\in H^1(\Om)$, where we use the trace theorem. Then we have that the norm $\NORM{v}_\del\DEF \sqrt{\langle v,v\rangle_\del}$ is equivalent to the $H^1$ norm, see \cite[Corollary 6.2]{AuchmutyFoundation}. The authors showed this by stating that 
\begin{align*}
	 \int_{\Om}(\nabla v)^2\;+\,\int_{\del\Om} v^2 \geq \alpha_0 \int_\Om v^2\,, 
\end{align*}
for some $\alpha_0>0$, thus $||v_m||^2_{H^1(\Omega)}\leq \frac{1}{\alpha_0}||v||^2_\del+\int_\Om(\nabla v)^2\leq C_H ||v||^2_\del$ . They proved it using a proof by contradiction and applying a minimizing sequence $(v_m)_{m\in\NN}$, with $||v_m||_{L^2(\Omega)}=1$, which has a subsequence converging to a limit functions $\hat{v}$ with $||\hat{v}||_\del=0$. But if  $||\hat{v}||_\del=0$ then $\hat{v}$ is constant, because $\int_{\Om}(\nabla \hat{v})^2 = 0$ and $\int_{\del\Om}\hat{v}^2=0$, thus $||\hat{v}_m||_{L^2(\Omega)}=1$, which is a contradiction.

We then confine Steklov eigenfunctions to satisfy the normalisation
\begin{align*}
	\NORM{s_i}_\del =1\,,
\end{align*}
this determines Steklov eigenfunction to simple Steklov eigenvalues up to a sign. Using the Gram-Schmidt orthogonalisation with the inner-product $\langle \cdot\,,\cdot\rangle_\del$ on Steklov eigenfunctions to eigenvalues with non simple multiplicity, we can set all Steklov eigenfunctions to be orthonormal to each other within the inner-product $\langle \cdot\,,\cdot\rangle_\del$, because 
\begin{align*}
	\int_{\Om}\nabla s_i\cdot\nabla s_j 
		&= \lambda_j\int_{\del\Om} s_i\, s_j = \lambda_i\int_{\del\Om} s_i\, s_j = 0\,,\quad && \text{for } i\neq j\,,
\end{align*}
due to Green's identity. We also have that		
\begin{align*}
	\langle s_i\,, s_i\rangle_\del
		&=(1+\lambda_i)\NORM{s_i}_{L^2(\del\Om)}^2 =1\,,\quad && \text{for } i\in\NN_0\,,\\
	\int_{\Om}\nabla s_i\cdot\nabla s_i 
		&= \lambda_i\int_{\del\Om}(s_i)^2 = \frac{\lambda_i}{1+\lambda_i}\,,\quad && \text{for } i\in\NN_0\,, \\
	s_0
		&\equiv \frac{1}{\sqrt{|\del\Om|}}\,.&
\end{align*}


\subsection{Higher Order Layer Potentials}\label{sec:HighOrderLayer}
Given a Lipschitz boundary $\del\Om$, we define the single layer potential $\mathcal{S}: H^{-1/2}(\del\Om)\rightarrow H^{1/2}(\del\Om)$ and the Neumann-Poincaré operator $\mathcal{K}^\ast: H^{-1/2}(\del\Om)\rightarrow H^{-1/2}(\del\Om)$ through
\begin{align*}
	\mathcal{S}[\phi](\tau) &= \int_{-\pi}^\pi \frac{1}{2\pi} \log{|x(\tau)-x(t)|}\;\phi(t)\,|T(t)|\,\intd t\,,\\
	\mathcal{K}^\ast[\phi](\tau) &= \pvint_{-\pi}^\pi \frac{1}{2\pi} \frac{\nu_{x(\tau)}\cdot(x(\tau)-x(t))}{|x(\tau)-x(t)|^2}\phi(t)\,|T(t)|\,\intd t\,,
\end{align*}
where $x(t)\in\del\Om$ parametrizes the boundary of $\Om$ with $|T(t)|\DEF |\tfrac{\intd}{\intd t}{x}(t)|>0$, for all $t\in (-\pi,\pi]$, and where the p.v. denotes a principle value integral. We refer to \cite{VerchotaSingleLayer, WeiLiNPOp} for details. 

The authors in \cite[Theorem 1.1]{HighRegLayerPotMazya} showed that the operators $-\tfrac{1}{2}\mathrm{I}+\mathcal{K}^\ast: H^{l}(\del\Om)\rightarrow H^{l}(\del\Om)$ and  $\mathcal{S}: H^{l-1}(\del\Om)\rightarrow H^{l}(\del\Om)$ are continuous and invertible excluding an one dimensional subspace, provided $2(l-1)>n-1$, for $l>1$, where $n$ denotes the dimension of space, that is $\RR^n$, with $n\geq 3$, and provided $\del\Om\in H^l_{\text{loc}}$. 
For $\RR^2$, they state that: "The changes required in formulations, $n > 2$, are the same as in the logarithm potential theory for contours. Our proofs, given for $n > 2$, apply to the two dimensional case after minor changes." (\cite[pp. 100]{HighRegLayerPotMazya}). Here we use the logarithm potential theory described in \cite[Section 4.]{VerchotaSingleLayer}.
\cite[Theorem 1.1]{HighRegLayerPotMazya} also provides results on the the operators $\mathcal{S}$ and $\tfrac{1}{2}\mathrm{I}+\mathcal{K}^\ast$ for values $1<l\leq (n-1)/2+1$, but with an additional condition on the boundary. That condition is satisfied with the assumption, that $\del\Om\in H^{l+1/2}_{\text{loc}}$, for all $l > 1$. Together with the result in \cite[Theorem 4.11]{VerchotaSingleLayer}, we then have the following lemma.
\begin{lemma} \label{lemma:HighRegOps}
	Let $f_0$ be the unique function satisfying $(-\tfrac{1}{2}\mathrm{I}+\mathcal{K}^\ast)f_0 = 0$ and $\int_{\del\Om}f_0 = 1$. Let $\langle f_0\rangle$ be the space spanned by $f_0$ and let $\langle 1\rangle$ be the space spanned by constant functions. 
	
	Assume $\del\Om\in H^{p+1/2}_{\text{loc}}$, $p\in\NN$ with $p\geq 1$, then the following statements hold:
	\begin{itemize}
		\item[(1)] $\mathcal{S}$ maps $H^{p-1}(\del\Om)\ominus \langle f_0\rangle$ isomorphically onto $H^{p}(\del\Om)\ominus\langle 1\rangle$, and $\mathcal{S}[f_0]$ is a constant function.
		\item[(2)] $(-\tfrac{1}{2}\mathrm{I}+\mathcal{K}^\ast)$ is an isomorphism in $H^{p}(\del\Om)\ominus\langle 1\rangle$.
	\end{itemize}
\end{lemma}

\subsection{Higher Order Steklov Eigenfunctions}\label{sec:HOSE}

\begin{proposition}\label{prop:s_i in H^p}
	Assume $\del\Om\in H^{p+1/2}_{\text{loc}}$, $p\in\NN$ with $p\geq 2$, then 
	\begin{align*}
		s_i\in H^p(\del\Om)\,,\quad \forall \, i\in\NN\,.
	\end{align*}
	Furthermore, given that $s_i=\mathcal{S}[\phi_i]$, then $\phi_i\in H^p(\del\Om)$.
\end{proposition}

\begin{proof}
	From \cite[Theorem 4.1]{VerchotaSingleLayer} we can readily infer that there exists a $\phi_i\in L^2(\del\Om)$ such that $s_i=\mathcal{S}[\phi_i]$, since $\int_{\del\Om}s_i=0$, for all $i\geq 1$, for Lipschitz domains. Using the boundary condition for $s_i$ and limit results for the single layer potential, especially that $\lim_{x\rightarrow\del\Om} \del_\nu \mathcal{S}[\phi](x) = (-\tfrac{1}{2}\mathrm{I}+\mathcal{K}^\ast)[\phi](x)$, we obtain that
	\begin{align*}
		(-\tfrac{1}{2}\mathrm{I}+\mathcal{K}^\ast)[\phi_i] = \lambda_i\,\mathcal{S}[\phi_i]\,,
	\end{align*}
	that is $-\tfrac{1}{2}\phi_i = \lambda_i\,\mathcal{S}[\phi_i] - \mathcal{K}^\ast[\phi_i]$. We can infer that the right-hand side is in $H^1(\del\Om)$, because of Lemma \ref{lemma:HighRegOps} and because for $C^2$ domains the integration kernel of $\mathcal{K}^\ast$ is smooth, thus $\phi_i\in H^1(\del\Om)$. Reapplying this argument, we see that the regularity of $\phi_i$ is as high as Lemma \ref{lemma:HighRegOps} allows, and that is $\phi_i\in H^p(\del\Om)$ and thus $s_i\in H^p(\del\Om)$.
\end{proof}

\subsection{Dirichlet to Neumann operator}

Let $\Om$ be a Lipschitz domain, then we define the Dirichlet to Neumann map $\mathcal{DN}: H^{\nicefrac{1}{2}}(\del\Om)\rightarrow H^{-\nicefrac{1}{2}}(\del\Om)$ such that $\mathcal{DN}[g] = \del_\nu u$ where $u$ is the solution to the homogeneous Dirichlet problem, that is
\begin{align}\label{PDE:Laplace}
	\left\{ 
	\begin{aligned}
		 \Laplace  \, u(x) &= 0 \quad &&\text{in } \Om\,, \\
		 u(y) &= g(y) \quad &&\text{on } \del \Omega \,.
	\end{aligned}
	\right.
\end{align}
According to \cite{DtNBehrndt}, $\mathcal{DN}$ is well-defined. From \cite[Comment after Proposition 1.8]{taylor2010partial}, for $\del\Om\in C^\infty$ and $p\geq 1$, we have that the solution $u$ to Equation (\ref{PDE:Laplace}) is in $H^{p+1}(\Om)$ in case $g\in H^{p+\tfrac{1}{2}}(\del\Om)$. Then with the trace theorem \cite{traceThmZhonghai} we have that $\nabla u|_{\del\Om}\in H^{p-\tfrac{1}{2}}(\del\Om)\subset C^{p-1}(\del\Om)$ using the Sobolev embedding and that $\del\Om$ is a one dimensional manifold. Hence we can infer that $\mathcal{DN}$ maps a function in $H^{p+\tfrac{1}{2}}(\del\Om)$ to a function in $H^{p-\tfrac{1}{2}}(\del\Om)$.

Considering that we can reformulate the Dirichlet to Neumann operator through
\begin{align*}
	\mathcal{DN} = (-\tfrac{1}{2}\mathrm{I}+\mathcal{K}^\ast)\,\mathcal{S}^{-1}\,,
\end{align*}
where we use the representation $u = \mathcal{S}[\phi]$, and several layer potential results \cite{LPTSA}, Lemma \ref{lemma:HighRegOps}, especially that $\mathcal{S}[f_0]$ is a constant function and $(-\tfrac{1}{2}\mathrm{I}+\mathcal{K}^\ast)[f_0]=0$, we see that $\mathcal{DN}: H^{l}(\del\Om)\rightarrow H^{l-1}(\del\Om)$ is well defined for $\del\Om\in H^{l+1/2}_{\text{loc}}$ with $l \geq 1$ or $l=1/2$.


Then we define
\begin{align*}
	\mathcal{DN}^p[g] =  \mathcal{DN}[\mathcal{DN}[\ldots\mathcal{DN}[g]\ldots]]\,,
\end{align*}
$p$ times, where $p\in\NN$. 

\begin{lemma} \label{lemma:DN estimate}
	Assume $\del\Om\in H^{p+1/2}_{\text{loc}}$, $p\in\NN$ with $p\geq 1$, then for all $g\in H^{p}(\del\Om)$ and all $i\in\NN$ we have that
	\begin{align*}
		\int_{\del\Om}g\,s_i\,\leq\,\NORM{\mathcal{DN}^p[g]}_{L^2(\del\Om)}\frac{1}{\lambda_i^{p}(1+\lambda_i)^{1/2}}\,.
	\end{align*}
\end{lemma}

\begin{proof}
	With Green's identity we readily see that
	\begin{align*}
		\int_{\del\Om}g\,s_i 
			= \int_{\del\Om} g\,\frac{1}{\lambda_i}\del_\nu s_i 
			= \int_{\del\Om}\del_\nu u\,\frac{1}{\lambda_i} s_i 
			= \frac{1}{\lambda_i}\int_{\del\Om}\mathcal{DN}^1[g]\, s_i \,.
	\end{align*}
	Repeat the process $p$ times and then we have that
	\begin{align*}
		\int_{\del\Om}g\,s_i 
			= \frac{1}{\lambda_i^p}\int_{\del\Om}\mathcal{DN}^p[g]\, s_i \,.
	\end{align*}
	Using the Cauchy-Schwarz inequality and that $\NORM{s_i}_{L^2(\del\Om)}=(1+\lambda_i)^{-1/2}$, we can infer Lemma \ref{lemma:DN estimate}.
\end{proof}

\begin{lemma}\label{lemma:normDN<normg invf}
	Let $\del\Om\in H^{p+1/2}_{\text{loc}}$, let $g\in H^p(\del\Om)$ be a function on the boundary and let $p\geq 1$ be an integer. Then
	\begin{align*}
		\NORM{\mathcal{DN}^p[g]}_{L^2(\del\Om)}\leq \NORM{g^{(p)}}_{L^2(\del\Om)}\Big(\max_{z\in\del\mathbb{D}}|f'(z)|^{-1}\Big)^p\,,
	\end{align*}
	where $\mathbb{D}$ denotes the open unit disk, where $f:\mathbb{D}\rightarrow\Om\subset\CC$ is a conformal map, and where $g^{(p)}$ the $p$-th weak derivative of $g$.
\end{lemma}

\begin{proof}
	First, let us assume that $\Om=\mathbb{D}$. Then the solution $u:\Om\rightarrow\RR$ to the Laplace problem with boundary $u\!\mid_{\del\Om}=g$ can be described as
	\begin{align*}
		u(x) = \sum_{n=0}^\infty g_n\,r^n\,\euler^{\imagi n\,\theta}\,,
	\end{align*}
	where $x\in\Om$ and $r,\theta$ denotes its radial and angular part, and where $(g_n)_{n=0}^\infty\subset\CC$. By the chain rule, we have that $\del_\nu=\del_r$ and that the tangential derivative is given through $\tfrac{1}{r}\del_\theta$, which in turn is equal to $\tfrac{1}{r} g'$. Then we can infer that
	\begin{align*}
		\mathcal{DN}^p[g] 
			&=  \sum_{n=p}^\infty g_n\,\frac{n!}{(n-p)!}\, r^{n-p}\,\euler^{\imagi n\, \theta}\,.
	\end{align*}
	Using Parseval's theorem, we can say that 
	\begin{align*}
		\NORM{\mathcal{DN}^p[g]}_{L^2(\del\Om)}^2
			=  \sum_{n=p}^\infty |g_n|^2\,\frac{n!^2}{(n-p)!^2}
			\leq  \sum_{n=0}^\infty |g_n|^2\,(n^{p})^2
			= \NORM{\del_\theta^p u}_{L^2(\del\Om)}^2\,,
	\end{align*}
	and thus 
	\begin{align*}
		\NORM{\mathcal{DN}^p[g]}_{L^2(\del\Om)}
			\leq \NORM{g^{p}}_{L^2(\del\Om)}\,.
	\end{align*}
	
	To generalize, we consider the conformal map $f:\mathbb{D}\rightarrow\Om\subset\CC$. Let $u:\Om\rightarrow \RR$ be the solution the Laplace equation with boundary $u\!\mid_{\del\Om}=g$. We define $w \DEF u\circ f: \mathbb{D}\rightarrow \RR$. From complex analysis we know that $w$ is a solution to the Laplace equation on $\mathbb{D}$.
	
	Let $z\in\del\mathbb{D}$ and $\omega\in\del\Om$, such that $f(z) =\omega$, then $w\!\mid_{\del\mathbb{D}}=g\circ f$. Using the chain rule and that $\nu_z=z$ for $z\in\del\mathbb{D}$, we obtain that  $\nu_\omega = \frac{f'(z)}{|f'(z)|}z=\frac{f'(z)}{|f'(z)|}\nu_z$. Using matrix notation to describe the complex multiplication, we can infer that
	\begin{align*}
		\nu_z \cdot\nabla w(z) 
			&=
			|{f'(z)}|
			\begin{bmatrix}
				\del_1 f_1 & \del_2 f_1\\
				\del_1 f_2 & \del_2 f_2
			\end{bmatrix}^{-1}
			\begin{bmatrix}
				(\nu_\omega)_1\\
				(\nu_\omega)_2
			\end{bmatrix}
			\;\cdot\;
			\begin{bmatrix}
				\del_1 f_1 & \del_2 f_1\\
				\del_1 f_2 & \del_2 f_2
			\end{bmatrix}
			\begin{bmatrix}
				\del_1 u\\
				\del_2 u
			\end{bmatrix}\\
			&= |{f'(z)}|\,\nu_\omega\cdot \nabla u(\omega)\,.
	\end{align*}
	From the unit disk case, we can see that
	\begin{align*}
		\NORM{\mathcal{DN}[g]}_{L^2(\del\Om)}
			&= \NORM{\nu_\omega\cdot \nabla u}_{L^2(\del\Om)}
			 = \NORM{|f'|^{-1}\,\nu_z\cdot \nabla w}_{L^2(\del\mathbb{D})}\\
			&\leq \max_{z\in\del\mathbb{D}}|f'(z)|^{-1}\NORM{\nu_z\cdot \nabla w}_{L^2(\del\mathbb{D})}\\
			&\leq \max_{z\in\del\mathbb{D}}|f'(z)|^{-1}\NORM{g\circ f}_{L^2(\del\mathbb{D})}
	\end{align*}
	Applying the equation $p$-times, we can conclude that
	\begin{align*}
		\NORM{\mathcal{DN}^p[g]}_{L^2(\del\Om)}
			\leq \NORM{g^{(p)}}_{L^2(\del\Om)}\Big(\max_{z\in\del\mathbb{D}}|f'(z)|^{-1}\Big)^p\,.
	\end{align*}
\end{proof}

\begin{proposition}\label{prop:uniformL2boundedFcts}
	Let $\Om\in C^\infty$, and let $g\in C^\infty(\del\Om)$. Assume there exists a conformal map $f:\mathbb{D}\rightarrow\Om$, such that $\min |f'|\geq 1$. If there exists a constant $C_g>0$ such that
	\begin{align*}
		\sup_{p\in\NN}\NORM{g^{(p)}}_{L^2(\del\Om)}\leq C_g\,,
	\end{align*}
	then there exists an integer $\kappa>0$ and real constants $(g_i)_{i=0}^\kappa$ such that
	\begin{align*}
		g = \sum_{i=0}^{\kappa-1} g_i s_i\,,
	\end{align*}
	and $\kappa$ denotes the first Steklov eigenvalue $\lambda_i$, which is bigger than $1$.
\end{proposition}
This proposition states that any smooth function $g$ with uniformly $L^2$-norm bounded derivatives is a finite linear combination of Steklov functions.

\begin{proof}
	Lemma \ref{lemma:DN estimate} states that
	\begin{align*}
		\int_{\del\Om}g\,s_i\,\leq\,\frac{\NORM{\mathcal{DN}^p[g]}_{L^2(\del\Om)}}{\lambda_i^{p}(1+\lambda_i)^{1/2}}\,.
	\end{align*}
	With Lemma \ref{lemma:normDN<normg invf} we then obtain that
	\begin{align*}
		\int_{\del\Om}g\,s_i\,
			\leq\,\frac{C_g\Big(\max_{z\in\del\mathbb{D}}|f'(z)|^{-1}\Big)^p}{\lambda_i^{p}(1+\lambda_i)^{1/2}}
			\leq \frac{C_g \, 1}{\lambda_i^{p}(1+\lambda_i)^{1/2}}\,.
	\end{align*}
	Thus, if $\lambda_i>1$ then $\big|\int_{\del\Om}g\,s_i\big|\leq 0$, if we let $p\rightarrow\infty$. From Section \ref{sec:HEO}, and the fact that $\lambda_i$ increase linearly, there exists a constant $\kappa$, which denotes the first $\lambda_i>1$, and which satisfies the Proposition. 
\end{proof}

For the special case $\Om=\mathbb{D}$, we have the following corollary.
\begin{corollary}
	Given a smooth, $2\pi$-periodic real function $g$, if there exists a constant $C>0$ such that
	\begin{align*}
		\sup_{p\in\NN}\NORM{g^{(p)}}_{L^2(\del\Om)}\leq C\,,
	\end{align*}
	then there exist three constants $a,b,c\in\RR$ such that
	\begin{align*}
		 g(t) = a+b\,\cos(t)+c\,\sin(t) \,.
	\end{align*}
\end{corollary}

\subsection{Theoretical PDE Solution in Steklov Series}\label{sec:HEO}
Let $E_H : L^2(\del\Om)\rightarrow L^2(\Om)$ be the harmonic extension operator 
\begin{align*}
	E_H[g](x) \DEF \sum_{i=0}^\infty 
		(1+\lambda_i)\; s_i(x) \int_{\del\Om}g\, s_i\,.
\end{align*}
We define
\begin{align*}
	g_i \DEF (1+\lambda_i)\int_{\del\Om}g\, s_i\,,
\end{align*}
for $j\in\NN_0$.
Then we define the truncated harmonic extension operator $E_H^{(M)} : L^2(\del\Om)\rightarrow L^2(\Om)$ to be 
\begin{align*}
	E_H^{(M)}[g](x) \DEF \sum_{i=0}^{M} g_i\, s_i(x)\,.
\end{align*}
From \cite[Theorem 6.2]{AuchmutyHsSpaces} we have that for Lipschitz domains that $E_H[g]$ is an isometric isomorphism of $H^{1/2}(\del\Om)$ and the space of solutions to the Laplace equation in $\Om$ equipped with a particular inner product. From \cite{AuchmutySVP} we have that 
\begin{align*}
	\NORM{E_H[g]}_{L^2(\Om)}\leq C\, \NORM{g}_{L^2(\del\Om)}\,,
\end{align*}
for all $g\in L^2(\del\Om)$ for smooth boundaries. More general boundary regularity conditions are given in \cite{AuchmutySVP}.

\subsection{Numerical PDE Solution}

We define $\tilde{u}_{M,N}^{(\mathrm{D})}$ to be the numerical solution to Equation (\ref{PDE:Helmholtz:Dir}) with Dirichlet boundary data $g^{(\mathrm{D})}$ follows. We first approximate the solution $u$ with boundary data $g$ through $E_H^{(M)}[g]$ and then numerically compute the eigenfunctions $s_i$ in $E_H^{(M)}$ through the single layer potential, that is 
\begin{align}\label{equ:s_i=int Gamma phi}
	s_i(x) = \mathcal{S}[\phi_i](x)\,,
\end{align}
for $x\in\Om$, where $\phi\in L^2(\del\Om)$. We approximate $\phi_i$ through numerical means and call the approximation $\tilde{\phi_i}$. We numerically compute $\tilde{\phi_i}$ through the truncation of Fourier representation. To be more precise, we denote $\phi_i(t)\DEF\phi_i(x(t))\in\RR$, for $t\in (-\pi,\pi]$, through the series
\begin{align}\label{equ:phiExpSum}
	\phi_i(t) = \sum_{n=-\infty}^{\infty} c_n^{(i)}\exp(\imagi nt)\,.
\end{align}
Applying the Steklov boundary condition $\del_\nu s_i = \lambda_i s_i$ to the single layer description for $s_i$, we can formulate a condition in terms of $\phi_i$, that is
\begin{align}\label{equ:phiCondition}
	(-\tfrac{1}{2}\mathrm{I}+\mathcal{K}^\ast)[\phi_i] = \lambda_i\,\mathcal{S}[\phi_i]\,.
\end{align}
Through the condition in Equation (\ref{equ:phiCondition}) we can numerically compute the coefficients $c_n^{(i)}$. Then we define the numerical solution to be 
\begin{align*}
	\tilde{\phi_i}(t) = \sum_{n=-N}^{N}c_n^{(i)}\exp(\imagi nt)\,,
\end{align*}
for $N\in\NN$.

We define $\tilde{s_i}=\mathcal{S}[\tilde{\phi_i}]$ and then 
\begin{align*}
	\tilde{u}_{M,N}^{(\mathrm{D})}=\sum_{i=0}^M \tilde{g}_i^{(\mathrm{D})} \tilde{s_i} \quad \text{with}\quad
	\tilde{g}_i^{(\mathrm{D})}=(1+\lambda_i)\int_{\del\Om} g^{(\mathrm{D})}\, \tilde{s_i}\,.
\end{align*}

For the Robin boundary condition we define $\tilde{u}_{M,N}^{(\mathrm{R})}$ to be the numerical solution to Equation (\ref{PDE:Helmholtz:Rob}) with Robin boundary data $g^{(\mathrm{R})}$. Assuming that the solution $u$ to the Helmholtz equation with Robin boundary condition has a trace $u\!\mid_{\del\Om}$ on the boundary, we can formulate $u$ through $E_H[u\!\mid_{\del\Om}]$, that is $u = \sum_{i=1}^\infty \alpha_i \, s_i$, for some coefficients $\alpha_i\in\RR$. Applying this identity to Equation (\ref{PDE:Helmholtz:Rob}) and applying $\int_{\del\Om} s_i$ on both sides, we see that 
\begin{align*}
\alpha_i = \frac{g_i^{(\mathrm{R})}}{\lambda_i+b}\quad \text{with}\quad
	g_i^{(\mathrm{R})}=(1+\lambda_i)\int_{\del\Om} g^{(\mathrm{R})}\, s_i\,.
\end{align*}
Hence we define that
\begin{align*}
	\tilde{u}_{M,N}^{(\mathrm{R})}=\sum_{i=0}^M \frac{\tilde{g}_i^{(\mathrm{R})}}{\lambda_i+b} \, \tilde{s_i} \quad \text{with}\quad
	\tilde{g}_i^{(\mathrm{R})}=(1+\lambda_i)\int_{\del\Om} g^{(\mathrm{R})}\, \tilde{s_i}\,.
\end{align*}

\section{Main Result}\label{sec:main}
\begin{theorem}\label{thm:main}
	Let $\del\Om\in C^p$ with $p\in\NN$, $p\geq 2$ and let $g^{(\mathrm{D})}\in H^q(\del\Om)$ for $q\geq 2$. Let $\rho\DEF \min(p,q)$. Then we have that the exact solution $u$ to the Dirichlet problem with boundary data $g^{(\mathrm{D})}$ is $E_H[g^{(\mathrm{D})}]$. For the numerical approximation $\tilde{u}_{M,N}^{(\mathrm{D})}$ we have that
	\begin{align*}
		\NORM{u-\tilde{u}_{M,N}^{(\mathrm{D})}}_\delta
			&= \tilde{C}_{\rho} \sqrt{\sum_{i=M+1}^\infty \frac{1}{\lambda_i^{2\rho-1}}}
				+  \sum_{i=0}^M C_i\NORM{\phi^{(p)}}_{L^2(\del\Om)}\!\!\!\!\frac{\NORM{g}_{L^2(\del\Om)}\, \NORM{\tilde{\phi_i}}_{L^2(\del\Om)}}{N^{p-1/2}}\\
			&= \OO\Big(\frac{1}{M^{\rho-1}}\Big)+\OO\Big(\frac{M}{N^{p-1/2}}\Big)\,,
	\end{align*}
	where $\tilde{C}_{\rho}>0$ is proportional to $\sup_{n=1,\ldots,q}\NORM{g^{(n)}}_{L^2(\del\Om)}$.
\end{theorem}

\begin{corollary}\label{coro:main}
	Let $\del\Om\in C^p$ with $p\in\NN$, $p\geq 2$ and let $g^{(\mathrm{R})}\in H^q(\del\Om)$ for $q\geq 2$. Let $\rho\DEF \min(p,q)$. Then we have for the exact solution $u$ to the Robin problem with boundary data $g^{(\mathrm{R})}$ that 
	\begin{align}\label{coro:term}
	u = \sum_{i=0}^\infty \frac{1+\lambda_i}{b+\lambda_i}\;{s_i}\int_{\del\Om} g^{(\mathrm{R})}\, {s_i}
	\end{align}
	and for the numerical approximation $\tilde{u}_{M,N}^{(\mathrm{R})}$ we have that
	\begin{align*}
		\NORM{u-\tilde{u}_{M,N}^{(\mathrm{R})}}_\delta		
			&= \OO\Big(\frac{1}{M^{\rho}}\Big)+\OO\Big(\frac{\log(M)}{N^{p-1/2}}\Big)\,.
	\end{align*}
\end{corollary}

We note here that the Robin boundary problem has a factor of $\frac{1}{M}$ lower  error in the numerical approximation than the Dirichlet boundary problem. This follows due to the factor $\frac{1}{b+\lambda_i}$ in the Steklov series expansion, Equation (\ref{coro:term}).

\begin{remark}
	Interpreting the numerical tests in the following section, Section \ref{sec:Numerics}, we hypothesise a stronger result than what is given in Theorem \ref{thm:main}, and that is
	\begin{align}\label{equ:remAboutThm}
		||u^{(\mathrm{D})}-\tilde{u}_{M,\infty}^{(\mathrm{D})}||&_{\del}
			= \OO(M^{-(\min(p+1.5,q)-1/2)})\,,\nonumber\\
		||u^{(\mathrm{D})}-\tilde{u}_{M,\infty}^{(\mathrm{D})}||&_{L^2(\del\Om)}
			= \OO(M^{-\min(p+1.5,q)})\,,
	\end{align}
	and
	\begin{align}\label{equ:remAboutThm:robin}
		||u^{(\mathrm{R})}-\tilde{u}_{M,\infty}^{(\mathrm{R})}||&_{\del}
			= \OO(M^{-(\min(p+1.5,q)+1/2)})\,,\nonumber\\
		||u^{(\mathrm{R})}-\tilde{u}_{M,\infty}^{(\mathrm{R})}||&_{L^2(\del\Om)}
			=  \OO(M^{-(\min(p+1.5,q)+1)})\,,
	\end{align}
	for smooth enough domains.
\end{remark}
To prove Equation (\ref{equ:remAboutThm}) and (\ref{equ:remAboutThm:robin}) we suppose a more accurate estimation of the term  $$\int_{\del\Om}\mathcal{DN}^p[g]\, s_i$$ in the proof of Lemma \ref{lemma:DN estimate} is necessary. We used the Cauchy-Schwarz inequality instead. 

\begin{proof}[Theorem \ref{thm:main}]
	The first statement follows from \cite[Theorem 6.2]{AuchmutyHsSpaces}.
	
	For short we denote $g_i = g_i^{(\mathrm{D})}$ and $\tilde{g}_i = \tilde{g}_i^{(\mathrm{D})}$. Using the triangle inequality, we have that
	\begin{align*}
		\NORM{E_H[g]-\tilde{u}_{M,N}^{(\mathrm{D})}}_\del
			&\leq \NORM{E_H[g]-E_H^{(M)}[g]}_\del
		+\NORM{E_H^{(M)}[g]-\tilde{u}_{M,N}^{(\mathrm{D})}}_\del\\
			&\leq \NORM{\sum_{i=M+1}^\infty g_i s_i}_\del
		+\sum_{i=0}^{M} \NORM{g_i s_i-\tilde{g_i}\tilde{s_i}}_\del\,.
	\end{align*}
	Let us consider $\NORM{\sum_{i=M+1}^\infty g_i s_i}_\del$. Using the properties of the Steklov eigenfunctions, we have that 
	\begin{align*}
		\NORM{ \sum_{i=M+1}^\infty g_i s_i}_\del^2
			&= \sum_{i=M+1}^\infty g_i^2 (1+\lambda_i) \int_{\del\Om}(s_i)^2 
			= \sum_{i=M+1}^\infty  g_i^2\\ 
			&=  \sum_{i=M+1}^\infty (1+\lambda_i)^2 \Big(\int_{\del\Om} g\, s_i \Big)^2\\
			&\leq \sum_{i=M+1}^\infty (1+\lambda_i)^2 \Big(C_{\rho} \frac{1}{\lambda_i^{\min(p,q)}(1+\lambda_i)^{1/2}}\Big)^2\\
			&= \OO\Big(\frac{1}{M^{2(\min(p,q)-1))}}\Big)\,,
	\end{align*}
	for $M\rightarrow\infty$ and any $p\in\NN$, where we used Lemma \ref{lemma:DN estimate} and \ref{lemma:normDN<normg invf} and that $\sum_{i=M}^\infty 1/i^n=\OO(1/M^{n-1})$, for $n>1$, and that $\lambda_i$ increase linear with $i$ according to Equation (\ref{equ:lambdaAsymp-C1}).
%
%
	The term $\NORM{g_i s_i-\tilde{g_i}\tilde{s_i}}_\del$ can be reformulated to
	\begin{align}\label{equ:proof:1}
		\NORM{(g_i-\tilde{g_i}) (s_i-\tilde{s_i})+\tilde{s_i}(g_i-\tilde{g_i})+\tilde{g_i}(s_i-\tilde{s_i})}_\del\,.
	\end{align}	
	Using the triangle inequality, we can consider each term separately. 
	
	First we consider $\NORM{s_i-\tilde{s_i}}_\del$. Using that $\NORM{\cdot}_\del$ is equivalent to $\NORM{\cdot}_{H^1}$, \cite[Corollary 6.2]{AuchmutyFoundation}, let us consider its $H^1$ norm. We have from \cite[Theorem 7.8]{Mazya2010} that
	\begin{align*}
		\NORM{s_i-\tilde{s_i}}_{H^1(\Om)}
		 = \NORM{\mathcal{S}[\phi_i-\tilde{\phi_i}]}_{H^1(\Om)}
		 \leq C'' \, \NORM{\mathcal{S}[\phi_i-\tilde{\phi_i}]\mid_{\del\Om}}_{H^{1/2}(\del\Om)}\,,
	\end{align*} 
	where we use that $\mathcal{S}[\cdot]$ satisfies the Laplace equation with Dirichlet data $\mathcal{S}[\cdot]_{\del\Om}$. Then according to \cite[Theorem 3.3]{VerchotaSingleLayer} we further have that
	\begin{align*}
		\NORM{\mathcal{S}[\phi_i-\tilde{\phi_i}]\mid_{\del\Om}}_{H^{1}(\del\Om)}
		\leq C' \, \NORM{\phi_i-\tilde{\phi_i}}_{L^2(\del\Om)}\,.
	\end{align*} 
	Thus we can readily see that
	\begin{align}\label{equ:proof:2}
		\NORM{s_i-\tilde{s_i}}_{\del}
		 \leq C \, \NORM{\phi_i-\tilde{\phi_i}}_{L^2(\del\Om)}\,.
	\end{align} 
	Similarly $\NORM{\tilde{s_i}}_{\del}\leq C \, \NORM{\tilde{\phi_i}}_{L^2(\del\Om)}$.
	Using the theory on convergence of Fourier series, we have that
	\begin{align}\label{equ:proof:3}
		\NORM{\phi_i-\tilde{\phi_i}}_{L^2(\del\Om)}^2
			&\!\!\!=\!\!\! \sum_{|n|=N+1}^{\infty}|c_n^{(i)}|^2\int_{-\pi}^\pi|\exp(\imagi nt)|^2\intd t \nonumber\\
			&= 2\pi \sum_{|n|=N+1}^{\infty}|c_n^{(i)}|^2
			=  \sum_{n=N+1}^{\infty} \frac{C}{n^{2p}} \NORM{\phi^{(p)}}_{L^2(\del\Om)}^2 \nonumber\\
			&=  \NORM{\phi^{(p)}}_{L^2(\del\Om)}^2 \OO(1/N^{2p-1}) \,,
	\end{align} 
	where we used that $\sum_{i=N}^\infty \frac{1}{i^{p}}=\OO(1/N^{p-1})$, that $\phi\in H^p(\del\Om)$, due to Proposition \ref{prop:s_i in H^p} and results on the coefficients of Fourier series. 
	
	Next we consider the term $|g_i-\tilde{g_i}|$ in Equation (\ref{equ:proof:1}). Using the Cauchy-Schwarz inequality, we have that
	\begin{align*}
		|g_i-\tilde{g_i}|= \Big|\int_{\del\Om}g\,(s_i-\tilde{s}_i)\Big|\leq \NORM{g}_{L^2(\del\Om)}\NORM{s_i-\tilde{s}_i}_{L^2(\del\Om)}\,,
	\end{align*}
	Then with Equations (\ref{equ:proof:2})-(\ref{equ:proof:3}) we can infer that 
	\begin{align}\label{equ:proof:4}
		|g_i-\tilde{g_i}|\leq \OO(1/N^{p-1/2}) \NORM{g}_{L^2(\del\Om)}\, \NORM{\phi^{(\min(p,q))}}_{L^2(\del\Om)}\,,
	\end{align}
	and similarly for $|\tilde{g}_i|\leq C \,\NORM{g}_{L^2(\del\Om)} \NORM{\tilde{\phi_i}}_{L^2(\del\Om)}$.
	
	With Inequalities (\ref{equ:proof:2})-(\ref{equ:proof:4}) we can estimate the term $\NORM{g_i s_i-\tilde{g_i}\tilde{s_i}}_\del$ to be less or equal to
	\begin{align*}
		\OO(1/N^{p-1/2}) \NORM{g}_{L^2(\del\Om)}\, \NORM{\phi^{(p)}}_{L^2(\del\Om)}\, \NORM{\tilde{\phi_i}}_{L^2(\del\Om)}\,.
	\end{align*}
	We note here that the series $\NORM{\tilde{\phi_i}}_{L^2(\del\Om)}$ is bounded from above because so is $\NORM{s_i}_{L^2(\del\Om)}$ and $\mathcal{S}$ is a bounded isomorphism. With the first inequality in this proof, the theorem follows.
\end{proof}

\begin{proof}[Corollary \ref{coro:main}]
	For the first statement, the solution $u$ has a trace on the boundary $u\!\mid_{\del\Om}$ and from Theorem \ref{thm:main} we can infer that $E_H[u\!\mid_{\del\Om}]$ represents $u$. Applying the representation to the Robin boundary condition, then applying $s_i$ on both sides and integrating both sides over the boundary $\del\Om$ we obtain an expression for all the coefficients in the infinite sum $E_H[u\!\mid_{\del\Om}]$. This then leads us to Term (\ref{coro:term}).
	
	For short we denote $g_i = g_i^{(\mathrm{R})}$ and $\tilde{g}_i = \tilde{g}_i^{(\mathrm{R})}$. For the second statement, we proceed similarly to the proof of Theorem \ref{thm:main}. With the triangle inequality we obtain
	\begin{align*}
		\NORM{u-\tilde{u}_{M,N}^{(\mathrm{R})}}_\del
			&\leq \NORM{\sum_{i=M+1}^\infty \frac{ g_i\, s_i}{\lambda_i+b}}_\del
		+\sum_{i=0}^{M} \frac{ 1}{\lambda_i+b}\NORM{g_i s_i-\tilde{g_i}\tilde{s_i}}_\del\,.
	\end{align*}
	The first term on the right-hand side can be estimated similarly to the proof of Theorem \ref{thm:main}, then we obtain
	\begin{align*}
		\NORM{ \sum_{i=M+1}^\infty \frac{ g_i\, s_i}{\lambda_i+b}}_\del^2 
			&=  \sum_{i=M+1}^\infty \frac{(1+\lambda_i)^2}{(b+\lambda_j)^2} \Big(\int_{\del\Om} g\, s_i \Big)^2\\
			&\leq \sum_{i=M+1}^\infty \frac{(1+\lambda_i)^2}{(b+\lambda_j)^2} \Big(C_p \frac{1}{\lambda_i^p(1+\lambda_i)^{1/2}}\Big)^2\\
			&= \OO\Big(\frac{1}{M^{2p}}\Big)\,,
	\end{align*}
	For the second term we analogously get
	\begin{align*}
		\frac{\NORM{g_i s_i-\tilde{g_i}\tilde{s_i}}_\delta}{b+\lambda_i}
			\leq\OO(\frac{1}{(b+\lambda_i)\,N^{p-1/2}}) \NORM{g}_{L^2(\del\Om)}\, \NORM{\phi^{(p)}}_{L^2(\del\Om)}\, \NORM{\tilde{\phi_i}}_{L^2(\del\Om)}\,.
	\end{align*}
\end{proof}

\section{Numerical Implementation and Tests}\label{sec:Numerics}

\subsection{Numerical Implementation}
We give three methods to numerically compute accurate 2-dimensional solutions to the Steklov eigenvalue problem. 
\vspace{0.2cm}

\underline{First method (Conformal Method)}

The first method relies on the implementation in \cite{ConfNumImpl}, where they used a conformal map, which maps the unit disk to the domain, with which they were able to reformulate the Steklov boundary problem. Using Fourier series they achieved very accurate eigenvalues, but they did not provide a method to evaluate the corresponding eigenfunctions. Here we rewrite their method to compute the eigenfunctions and achieve high accuracy.  

Given a domain $\Om\in\RR^2$ and a conformal map $f: \CC\rightarrow\CC$, which maps $\{z\in\CC | |z|^2<1\}$ to the complex embedded domain $\Om\in\CC$. Let $(\lambda_i,s_i)$ be one of the solutions to the Steklov problem (\ref{PDE:Steklov}). Then according to \cite[Proposition 2]{ConfNumImpl} there exists an analytic function $\Psi_i:\Om\rightarrow\CC$, such that $s_i = \mathcal{R}(\Psi_i)$, the real part of $\Psi_i$, and the following equation is satisfied,
\begin{align}\label{equ:Re wPsi = lam f RePsi}
	\mathcal{R}(\omega\; \del_\omega\Psi_i) = \lambda_i \,|\del_\omega f|\,\mathcal{R}(\Psi_i)\,,
\end{align}
where $|\omega| = 1$ and $\del_\omega$ denotes the complex derivative. For the next step we decompose $\Psi_i$ in their Fourier series on the unit circle, that is
\begin{align*}
	\Psi_i(\omega) = \sum_{n=-\infty}^\infty b^{(i)}_n \omega^n\,,
\end{align*}
for $|\omega| = 1$. From Equation (\ref{equ:Re wPsi = lam f RePsi}) we can infer that $b^{(i)}_n = 0$ for all $n\leq -1$. Next we truncate the last three Fourier series up to coefficient $N\in\NN$ and we define the nodes $(\omega_l)_{l=1}^L$, $L\in\NN$, on the unit disk. Then we define the matrices
\begin{align*}
	\mathbf{N} &= \begin{bmatrix}
    0       & 0      & 0      & \dots   & 0 \\
    0       & 1      & 0      & \dots   & 0 \\
    0       & 0      & 2      & \dots   & 0 \\
    \vdots  & \vdots & \vdots & \ddots  & 0 \\
    0       & 0      & 0      & 0       & N
	\end{bmatrix}\,,\\
	\mathbf{D} &= \begin{bmatrix}
    |\del_\omega f(\omega_1)|       & 0      & 0      & \dots   & 0 \\
    0       & |\del_\omega f(\omega_2)|      & 0      & \dots   & 0 \\
    0       & 0      & |\del_\omega f(\omega_2)|      & \dots   & 0 \\
    \vdots  & \vdots & \vdots & \ddots  & 0 \\
    0       & 0      & 0      & 0       & |\del_\omega f(\omega_N)|
	\end{bmatrix}\,,
\end{align*}
as well as $\mathbf{E}_{n,l}=\omega_l^n$ and the vector $\mathbf{B}^{(i)}_{n}=b^{(i)}_n$. Then we can rewrite Equation (\ref{equ:Re wPsi = lam f RePsi}) to
\begin{align*}
	\frac{1}{2}((\mathbf{N}\mathbf{E})^\TransT \mathbf{B}^{(i)}+\overline{(\mathbf{N}\mathbf{E})^\TransT \mathbf{B}^{(i)}}) 
		= \lambda_i \,\mathbf{D}\,\frac{1}{2}(\mathbf{E}^\TransT \mathbf{B}^{(i)}+\overline{\mathbf{E}^\TransT \mathbf{B}^{(i)}})\,.
\end{align*}
Splitting $\mathbf{E}$ and $\mathbf{B}^{(i)}$ into their real parts $\mathbf{E}_{R}$ and $\mathbf{B}^{(i)}_{R}$ and their imaginary parts $\mathbf{E}_{I}$ and $\mathbf{B}^{(i)}_{I}$, we can reformulate the last equation to
\begin{align}\label{equ:matrixSteklovProblem}
	\begin{bmatrix}
    	(\mathbf{N}\mathbf{E}_R)^\TransT & -(\mathbf{N}\mathbf{E}_I)^\TransT      
	\end{bmatrix}
	\;
	\begin{bmatrix}
    	\mathbf{B}^{(i)}_{R} \\ \mathbf{B}^{(i)}_{I}     
	\end{bmatrix}
	=
	\lambda_i
	\underbrace{
	\begin{bmatrix}
    	\mathbf{D}(\mathbf{E}_R)^\TransT & -\mathbf{D}(\mathbf{E}_I)^\TransT      
	\end{bmatrix}}_{\FED \mathbf{R}}
	\;
	\begin{bmatrix}
    	\mathbf{B}^{(i)}_{R} \\ \mathbf{B}^{(i)}_{I}     
	\end{bmatrix}\,.
\end{align}
Given that $L=2N$, we have square matrices on both sides and then the equation represents an generalized eigenvalue problem. The in-build MATLAB \cite{MATLAB} program \verb+eigs+ can be used to solve the problem and obtain $\lambda_i$ and $\mathbf{B}^{(i)}$. For non-square matrices we use the singular value decomposition for $\mathbf{R} = \mathbf{U}\mathbf{S}\mathbf{V}^\mathrm{H}$, then remove the non-compatible singular values, such that we can reformulate Equation (\ref{equ:matrixSteklovProblem}) into a standard eigenvalue problem. Subsequently we solve the problem with \verb+eigs+. 
When we obtain the coefficients $\mathbf{B}^{(i)}$ for $i=1,\ldots,M$, we evaluate the Steklov eigenfunction $s_i= \mathcal{R}(\mathbf{E}^\TransT \mathbf{B}^{(i)})$, at the nodes $(z_l)_{l=1}^L$. Afterwards we use the Gram-Schmidt orthogonalisation on all $(s_i)_{i=1,\ldots,M}$ with the inner-product $\langle \cdot\,,\cdot \rangle_\del$.

The disadvantage to this method is to find a conformal map. Especially, the map requires that the conformal map has no singularity in its domain. This is an issue because even the conformal map, which maps the unit circle to an ellipse, cannot not be described in a simple form. The Joukowski map is such a map, which maps the unit circle to an ellipse, but it has a singularity at the origin and at infinity. The Riemann mapping theorem guarantees a desired conformal map, but it is not easily obtainable. One way to obtain a conformal mapping to the ellipse is through the Bergman kernel, we refer to \cite[pp. 529 - 552]{HencriciPeterVol3} for further instructions.
For polygonal domains, we have to find the Fourier coefficients by evaluating $|\del_\omega f|$ by hand. In general this is done using the Schwarz-Christoffel conformal map. This map maps the upper half sphere onto the interior of a simple polygon. For the special case of the square $[-1,1]\times[-1,1]$, the function $f$, it maps the unit disk to the square, can be formulated as
\begin{align*}
	f(w) = 2\,c \int_0^{\frac{\imagi(1-\omega)}{1+\omega}}\!\!\!\!\frac{\intd z}{\sqrt{z(1-z^2)}}\;-\;(1-\imagi)\,,
\end{align*}
where $c=\frac{\Gamma(\tfrac{3}{4})}{2\sqrt{\pi}\Gamma(\tfrac{5}{4})}$. Given $\omega = \exp(\imagi \theta)$, $\theta\in (-\pi,\pi]$, then $\frac{\imagi(1-\omega)}{1+\omega}= \tan(\theta/2)$. Then we can infer that
\begin{align*}
	\del_\omega f(w) = 2\,c \frac{-2\imagi}{(1+\omega)^2\sqrt{\tan(\theta/2)(1-\tan(\theta/2)^2)}}\,.
\end{align*}
Then we can conclude that
\begin{align*}
	|\del_\omega f(w)|^2 =  \frac{c^2}{|\cos(\theta)|\cos(\theta/2)|\sin(\theta/2)|}\,.
\end{align*} 
\vspace{0.2cm}

\underline{Second method (Weak Form Method)}

The second method does not rely on a conformal map, but instead uses Green's identity to obtain that
\begin{align*}
	\int_{\del\Om} s_i\,\del v\,\intd\sigma = \int_{\del\Om} \lambda_i s_i\, v\,\intd\sigma\,,
\end{align*}
for all $i\in\NN_0$ and all $v\in C^{\infty}(\overline{\Om})$, with $\Laplace v = 0$ in $\Om$. 
We then use the trapezoidal rule to approximate the integral through $L\in\NN$ nodes on the boundary. Then the equation can be rewritten to
\begin{align*}
	\mathbf{F_{\del}}^\TransT\,{\Sigma}\,\mathbf{s}_i = \lambda_i\,\mathbf{F}^\TransT\,{\Sigma}\,\mathbf{s}_i\,,
\end{align*}
where $\mathbf{F}$ is a vector of the evaluation of the testfunction at the nodes on the boundary, and $\mathbf{F_{\del}}$ is the evaluation of the normal derivative of the testfunction at the nodes, analogous for the vector $\mathbf{s}_i$, and where ${\Sigma}$ is a diagonal matrix whose entries represent the additional information from the trapezoidal rule and the curve parametrization. We then choose the following testfunctions
\begin{align*}
	v_{n,\,\text{cos}} &= r^n\cos(n\, \theta)\,,\\ 
	v_{n,\,\text{sin}} &= r^n\sin(n\, \theta)\,.
\end{align*}
where $n\in\NN$, where $r\in [0,\infty)$ denotes the radius from the origin, and $\theta\in [0,2\pi)$ the angle. Then we write the evaluation of all testfunctions into the columns of $\mathbf{F}$ and analogous $\mathbf{F_{\del}}$, where
\begin{align*}
	\del_x v_{n,\,\text{cos}} &= n\, r^{n-1}\big(
		\cos(\theta)\cos(n\, \theta) + \sin(\theta)\sin(n\, \theta)\big)\,,\\ 
	\del_y v_{n,\,\text{cos}} &= n\, r^{n-1}\big(
		\sin(\theta)\cos(n\, \theta) - \cos(\theta)\sin(n\, \theta)\big)\,,\\
	\del_x v_{n,\,\text{sin}} &= n\, r^{n-1}\big(
		\cos(\theta)\sin(n\, \theta) - \sin(\theta)\cos(n\, \theta)\big)\,,\\ 
	\del_y v_{n,\,\text{sin}} &= n\, r^{n-1}\big(
		\sin(\theta)\sin(n\, \theta) + \cos(\theta)\cos(n\, \theta)\big)\,.
\end{align*}
Using the QR decomposition we obtain an matrix $\mathbf{B}$ and an upper triangular matrix $\mathbf{R}$ such that $(\sqrt{\Sigma}\,\mathbf{B}) \, \mathbf{R} = (\sqrt{\Sigma}\,\mathbf{F})$, and $(\sqrt{\Sigma}\,\mathbf{B})^\TransT (\sqrt{\Sigma}\,\mathbf{B})$ is an identity matrix. Then we define $\mathbf{B_{\del}} = \mathbf{F_{\del}} \,\mathbf{R}^{-1}$. Then we have that 
\begin{align*}
	\mathbf{B_{\del}}^\TransT\,{\Sigma}\,\mathbf{s}_i = \lambda_i\,\mathbf{B}^\TransT\,{\Sigma}\,\mathbf{s}_i\,.
\end{align*}
We can further use that $(\sqrt{\Sigma}\,\mathbf{B}) (\sqrt{\Sigma}\,\mathbf{B})^\TransT$ is also the identity matrix and obtain
\begin{align*}
	(\mathbf{B}\,\mathbf{B_{\del}}^\TransT\,{\Sigma} )\,\mathbf{s}_i = \lambda_i\,\mathbf{s}_i\,.
\end{align*}
We solve this eigenvalue problem with the MATLAB \cite{MATLAB} in-build function \verb+eigs+. 

The disadvantage to method 2 is that the term $r^n$ in $v_{n,\,\text{cos}},	v_{n,\,\text{sin}}$ can reach extreme values for $n>50$, which might lead to numerical error. This is especially the case, when the radius varies strongly on the boundary.
\vspace{0.2cm}

\underline{Third method (Collocation Method)}

The third method relies on the collocation method presented in \cite[Chapter 13]{kress2013linear}. Let $(\mathcal{L}_i)_{i=1}^{\infty}$ be the Lagrange basis for the trigonometric interpolation on the boundary to the nodes $(z_l)_{l=1}^L = \big(z(\theta_l)\big)_{l=1}^L\subset\del\Om$, that is, we can represent the density on the boundary $\phi(t)\DEF\phi(z(t))\in L^2(\del\Om)$, for $t\in(-\pi,\pi]$ through
\begin{align*}
	\phi(t) = \sum_{i=1}^{L} \gamma_i\, \mathcal{L}_i(t)\,,
\end{align*}
such that $\phi$ is $2\pi$-periodic and $\phi(\theta_l)=\gamma_l$, for all $l=1,\ldots,L$. This resembles the formulation given through Equation (\ref{equ:phiExpSum}). Then we apply the Steklov boundary equation with boundary layer potentials, that is $(-\tfrac{1}{2}\mathrm{I}+\mathcal{K}^\ast)[\phi_i] = \lambda_i\,\mathcal{S}[\phi_i]$. This requires the evaluation of integrals of the form
\begin{align*}
	\int_{\del\Om} K(\tau,t)\phi(t)\intd t\,.
\end{align*}
We use the Lagrange basis for the trigonometric interpolation on $K(\tau,\cdot)$. This leads to the evaluation of the integrals
\begin{align*}
	\int_{-\pi}^\pi \mathcal{L}_i(t)\, \mathcal{L}_j(t) \intd t\,.
\end{align*}
In case the kernel function $K(\tau,\cdot)$ has a logarithmic singularity, we extract the singularity and this leads to the evaluation of the integrals 
\begin{align*}
	\int_{-\pi}^\pi \log\big(\sin\big(\tfrac{t}{2}\big)^2\big)\mathcal{L}_i(t)\, \mathcal{L}_j(t) \intd t\,.
\end{align*}
The generalized eigenvalue problem is then again solved with \verb+eigs+. To obtain then the Steklov functions $s_i$, we apply the single layer potential on the density, that is we compute $\mathcal{S}[\phi_i]$, using again the Lagrange basis for the trigonometric interpolation.

Method three, that is the collocation method, usually performs worse than method 1, the collocation method, but method three does not require a conformal map. Furthermore, method three widely outperforms method 2. One downside to method three is that it does not work on polygonal domains, without mayors modification on the evaluation of the integrals involved. 

\subsection{Numerical Tests}
We test Theorem \ref{thm:main} on our first domain, which is expressed through the image of the conformal map $f:\omega \mapsto \sin(\omega)$, on the function domain $\{\omega\in\CC | \,|\omega|=1\}$. The image is depicted in Fig. \ref{fig:sinDomain}.
\begin{figure}
    \center
    \includegraphics[width=0.6\textwidth]{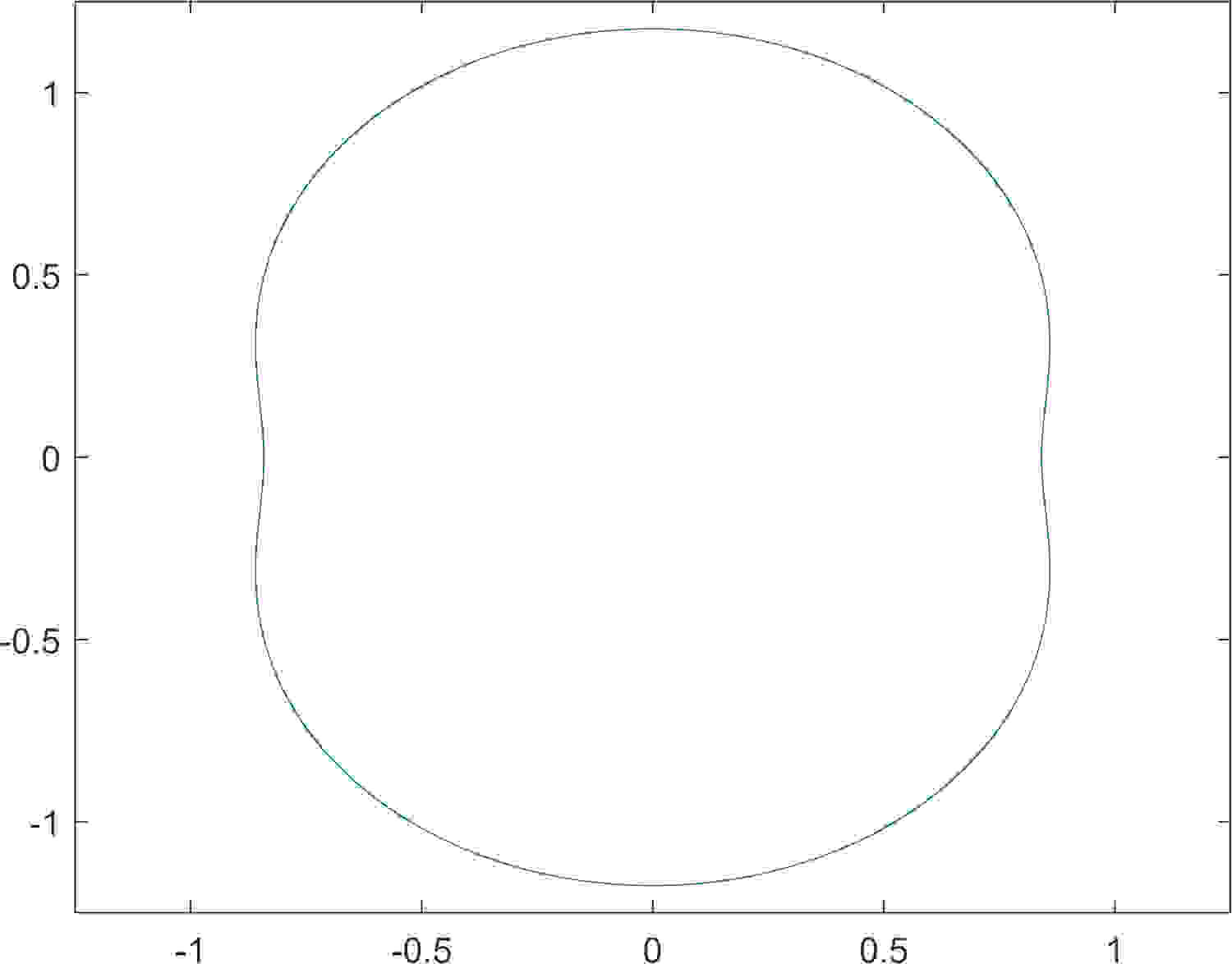}
    \caption{The image of the conformal map $f:\omega \mapsto \sin(\omega)$, for $|\omega|=1$, on the complex map.}\label{fig:sinDomain}
\end{figure}
Method 1 in the previously discussed numerical implementation with $L=2^{11}$, yields the Steklov eigenvalues given in Table \ref{tab:sinDomainEVals}. They agree with the eigenvalue obtained with method 2 up to an error $10^{-4}$ with 81 testfunctions and with method 3 up to an error $10^{-12}$.
\begin{table}[]
\begin{tabular}{|l l |l l | l l|}
0. & 0 &	 		     10.& 4.75126661249149 & 20. & 9.41724459940020\\
1. & 0.75224246625831 &  11.& 5.62207611654998 & 21. & 10.3475633749242\\
2. & 1.16409224301375 &  12.& 5.67424837547389 & 22. & 10.3565781261533\\
3. & 1.78027007332070 &  13.& 6.57012429497931 & 23. & 11.2899109864248\\ 
4. & 1.98491915625123 &  14.& 6.60890022153768 & 24. & 11.2960071742002\\
5. & 2.74939389918074 &  15.& 7.51651018262098 & 25. & 12.2318451677593\\
6. & 2.92299487296403 &  16.& 7.54208571185139 & 26. & 12.2362133827355\\
7. & 3.71499584614978 &  17.& 8.46107755518764 & 27. & 13.1735418905728\\
8. & 3.82029302327986 &  18.& 8.47973069883999 & 28. & 13.1765144903792\\
9. & 4.66997330559431 &  19.& 9.40475003802314 & 29. & 14.1150383379212
\end{tabular}
\caption{First 30 Steklov eigenvalues of the domain given in Fig. \ref{fig:sinDomain}.}\label{tab:sinDomainEVals}
\end{table}

We test Theorem \ref{thm:main} on our second domain, which is a kite shape form which is expressed through the formula 
\begin{align}\label{equ:kiteDomain}
	\begin{bmatrix}
    	\cos(\theta)+0.65\cos(2\,\theta)-0.65\\ 1.5\sin(\theta)   
	\end{bmatrix}\,,
\end{align}
for $\theta\in [0,2\pi)$. 
\begin{figure}
    \center
    \includegraphics[width=0.6\textwidth]{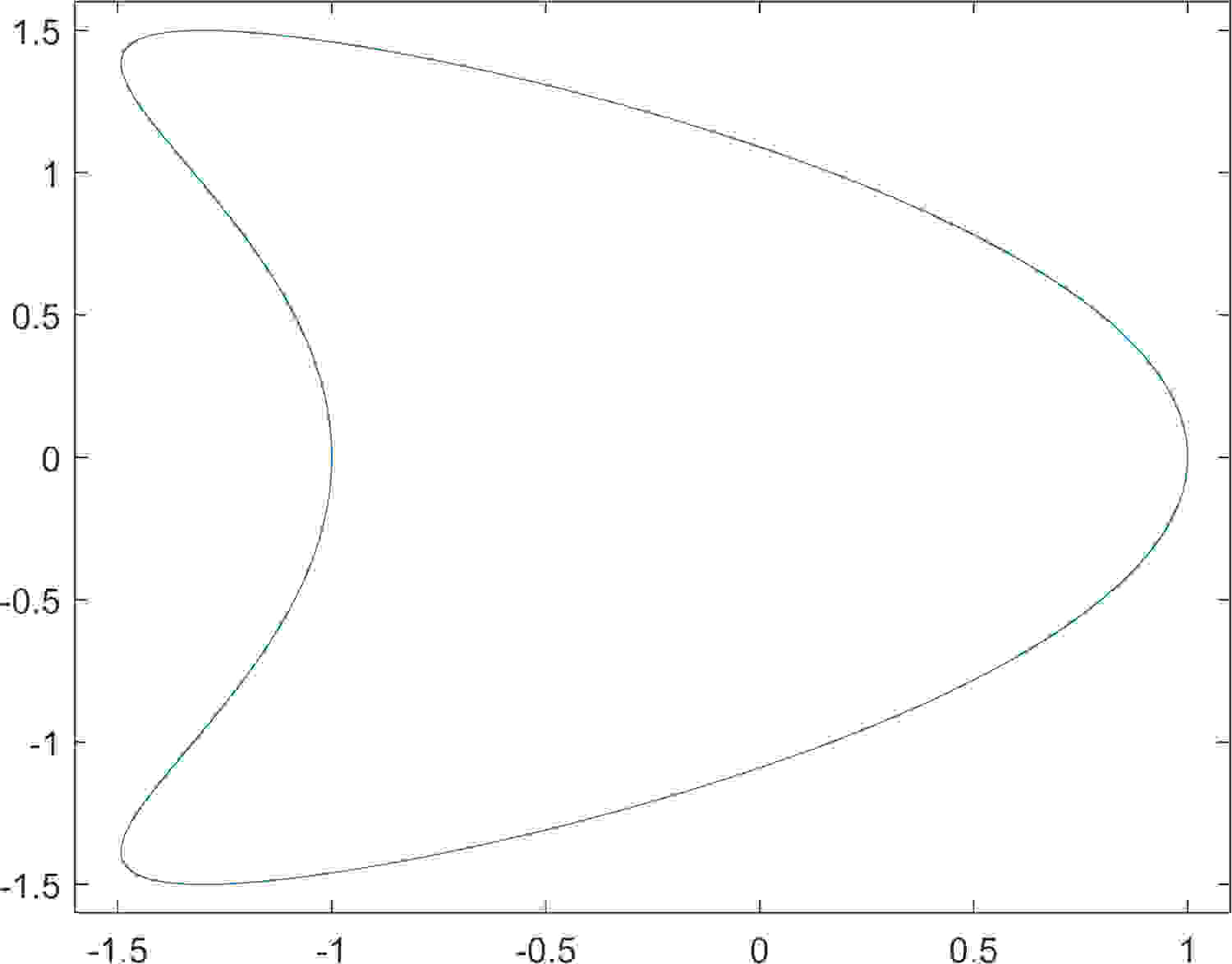}
    \caption{The kite shaped domain given by Equation (\ref{equ:kiteDomain}).}\label{fig:kiteDomain}
\end{figure}
Method 3 yields the Steklov eigenvalues given in Table \ref{tab:kiteEVals}.
\begin{table}[]
\begin{tabular}{|l l |l l | l l|}
0. & 0     			  &	 10.& 3.68146658262641 & 20. & 6.85521397993749\\
1. & 0.35414802795542 &  11.& 3.93065445612738 & 21. & 7.36643424104065\\
2. & 0.61788332452304 &  12.& 4.35126956149157 & 22. & 7.64210790405525\\
3. & 1.40104403386069 &  13.& 4.73445012960648 & 23. & 7.92286021006587\\ 
4. & 1.50919410730056 &  14.& 4.85934470703683 & 24. & 8.31088119268709\\
5. & 2.08851865665672 &  15.& 5.36316304266106 & 25. & 8.73881939891210\\
6. & 2.27367975961909 &  16.& 5.70611024861326 & 26. & 8.84738018807641\\
7. & 2.80793641242654 &  17.& 5.86271680048617 & 27. & 9.38887482455907\\
8. & 2.86305939312361 &  18.& 6.35291749840296 & 28. & 9.62242746116134\\
9. & 3.39265638962772 &  19.& 6.70357248994343 & 29. & 10.0138463423938
\end{tabular}
\caption{First 30 Steklov eigenvalues of the kite shaped domain.}\label{tab:kiteEVals}
\end{table}


We test Theorem \ref{thm:main} on our third domain, which is a perturbed circle which is expressed through the formula 
\begin{align}\label{equ:PerturbedDom}
	\begin{bmatrix}
    	|\theta-\pi|^3\mathrm{e}^{-2\,(\theta-\pi)^2}\cos(\theta)\\ \sin(\theta)   
	\end{bmatrix}\,,
\end{align}
for $\theta\in [0,2\pi)$. The boundary of this domain is in $C^2$ but not in $C^3$. 
\begin{figure}
	\center
    \includegraphics[width=0.6\textwidth]{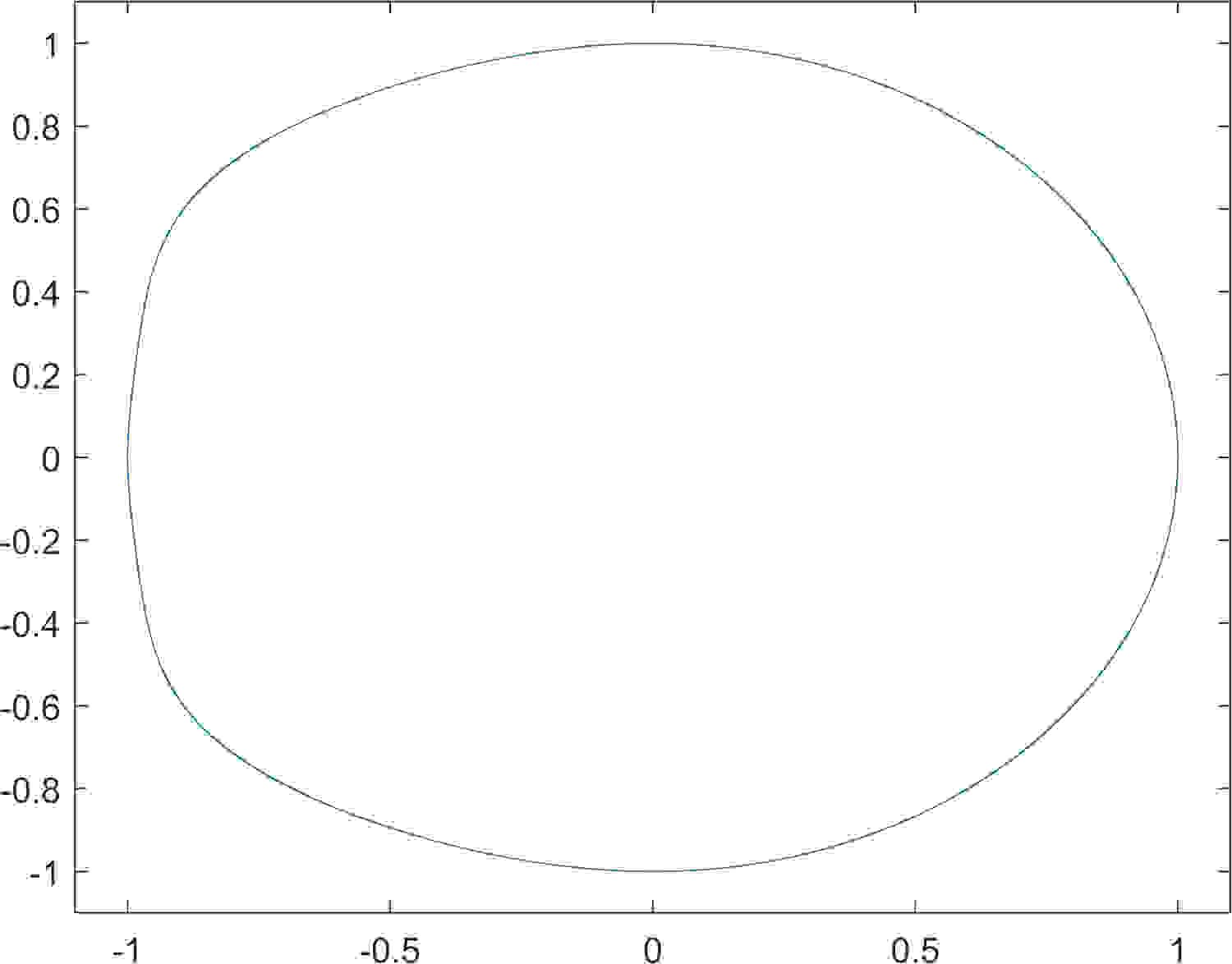}
    \caption{The perturbed circle shaped domain given by Equation (\ref{equ:PerturbedDom}).}\label{fig:PerturbedDom}
\end{figure}
Method 2 yields the Steklov eigenvalues given in Table \ref{tab:pertCircleEVals}.
\begin{table}[]
\begin{tabular}{|l l |l l | l l|}
0. & 0     			  &	 10.& 4.96027339646894 & 20. & 9.82373355092222\\
1. & 0.97166458952976 &  11.& 5.86819830291099 & 21. & 10.7966481422483\\
2. & 0.98562205634898 &  12.& 5.92246463679941 & 22. & 10.8052793010619\\
3. & 1.87253912164287 &  13.& 6.87277981374903 & 23. & 11.7803681618299\\ 
4. & 2.04290340847734 &  14.& 6.87802102564589 & 24. & 11.7852662151779\\
5. & 2.89726174456174 &  15.& 7.85289271831483 & 25. & 12.7634424475282\\
6. & 3.02554535119315 &  16.& 7.85942895389139 & 26. & 12.7659283729440\\
7. & 3.90257998079881 &  17.& 8.83653915707903 & 27. & 13.7455269458373\\
8. & 3.95874854028104 &  18.& 8.83849677763925 & 28. & 13.7476110533295\\
9. & 4.86273725468972 &  19.& 9.81456284300081 & 29. & 14.7272796398592
\end{tabular}
\caption{First 30 Steklov eigenvalues of the perturbed circle shaped domain.}\label{tab:pertCircleEVals}
\end{table}


We test Theorem \ref{thm:main} on our forth domain, which is a square $[-1,1]\times[-1,1]$. It is depicted in Fig. \ref{fig:squareDomain}.
\begin{figure}
    \center
    \includegraphics[width=0.6\textwidth]{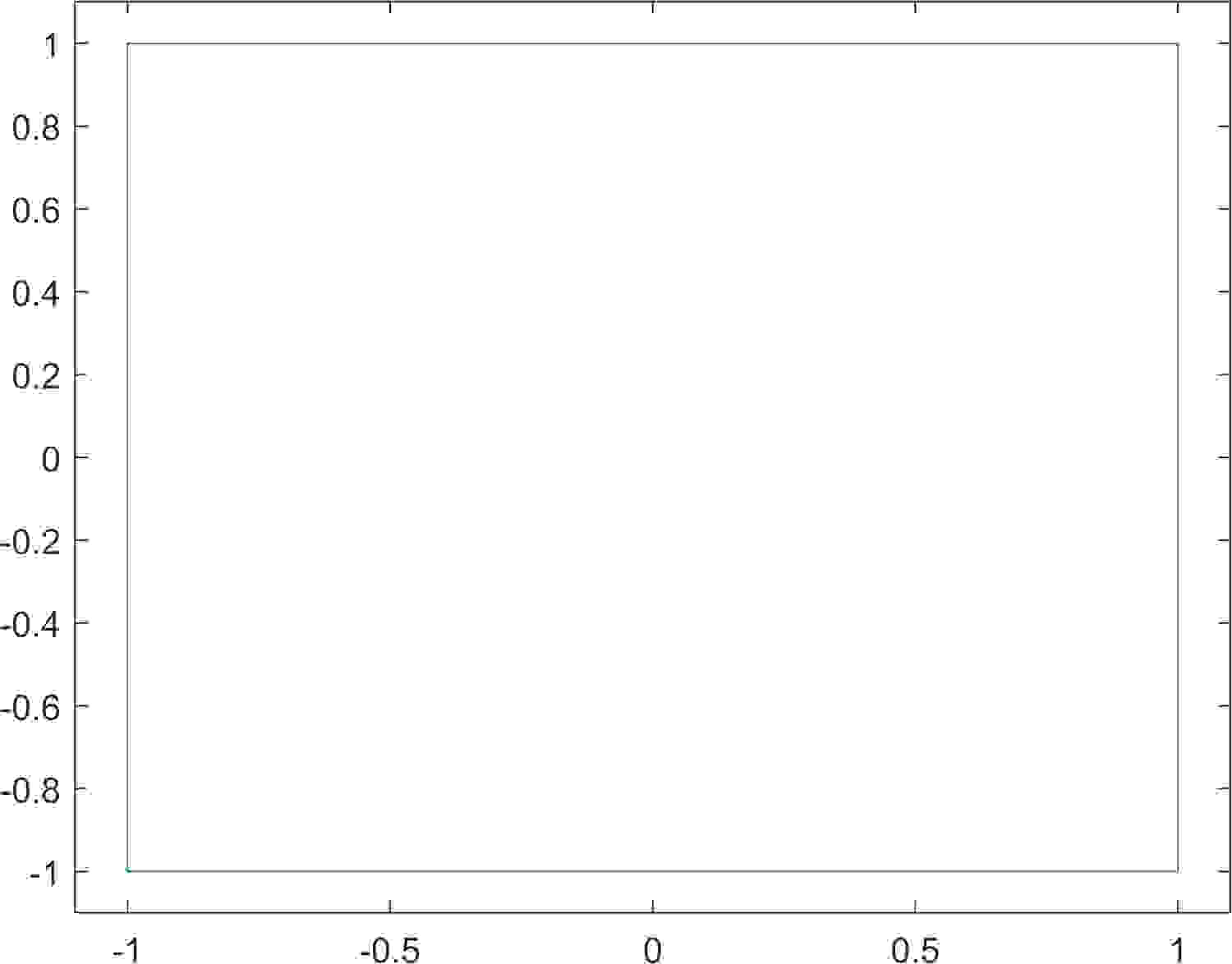}
    \caption{The forth domain, the square $[-1,1]\times[-1,1]$.}\label{fig:squareDomain}
\end{figure}
Here we are using neither of the three methods, instead we use the exact expressions for the eigenvalues and eigenvectors, which can be found in \cite{AuchmutyCho}. The first 30 Steklov eigenvalues are given in Table \ref{tab:squareEVals}.
\begin{table}[]
\begin{tabular}{|l l |l l | l l|}
0. & 0     			  &	 10.& 3.92965450678018 & 20. & 8.63937928739407\\
1. & 0.68825274233626 &  11.& 3.92965450678018 & 21. & 8.63937928739407\\
2. & 0.68825274233626 &  12.& 5.49761946836883 & 22. & 8.63938030734989\\
3. & 1                &  13.& 5.49761946836883 & 23. & 8.63938030734989\\ 
4. & 2.32363775343172 &  14.& 5.49795483551074 & 24. & 10.2101760978756\\
5. & 2.32363775343172 &  15.& 5.49795483551074 & 25. & 10.2101760978756\\
6. & 2.39038920510582 &  16.& 7.06857394684416 & 26. & 10.2101761504581\\
7. & 2.39038920510582 &  17.& 7.06857394684416 & 27. & 10.2101761504581\\
8. & 3.92433302324475 &  18.& 7.06859299435029 & 28. & 11.7809724496418\\
9. & 3.92433302324475 &  19.& 7.06859299435029 & 29. & 11.7809724496418
\end{tabular}
\caption{First 30 Steklov eigenvalues of the square domain $[-1,1]\times[-1,1]$.}\label{tab:squareEVals}
\end{table}
\vspace{0.2cm}

\underline{The Dirichlet Problem}\\
Considering the Dirichlet problem, we use four boundary function to test Theorem \ref{thm:main}, those are $g(\theta) = \exp(\cos(\theta))$, $g(\theta) = |\theta-\pi|$, $g(\theta) = \sqrt{|\theta-\pi|}$ and $g(\theta) = \mathbbm{1}_{\theta\geq \pi}-\mathbbm{1}_{\theta < \pi}$, for $\theta\in [0,2\pi)$ denoting the angle. The number of nodes $L$ is $2^{11}$ and the nodes are equidistant distributed $[0,2\pi)$.

For the sine mapped domain, Fig. \ref{fig:kiteDomain}, the corresponding numerical asymptotics of the $L^2(\del\Om)$ norm in the amount of eigenvalues $M$ used are shown in Figure \ref{fig:DiriErrorFigSine}. 
\begin{itemize}
\item For $g(\theta) = \exp(\cos(\theta))$ we have an exponential decay,
\item for $g(\theta) = |\theta-\pi|$ we have the asymptotic $\OO(M^{-3/2})$,
\item for $g(\theta) = \sqrt{|\theta-\pi|}$ we have the asymptotic $\OO(M^{-1})$, and
\item for $g(\theta) =  \mathbbm{1}_{\theta\geq \pi}-\mathbbm{1}_{\theta < \pi}$ we have the asymptotic $\OO(M^{-0.5})$.
\end{itemize}


For the kite shaped domain, Fig. \ref{fig:kiteDomain}, the corresponding numerical asymptotics of the $L^2(\del\Om)$ norm in the amount of eigenvalues $M$ used are shown in Figure \ref{fig:DiriErrorFigKite}. 
\begin{itemize}
\item For $g(\theta) = \exp(\cos(\theta))$ we have an exponential decay,
\item for $g(\theta) = |\theta-\pi|$ we have the asymptotic $\OO(M^{-3/2})$,
\item for $g(\theta) = \sqrt{|\theta-\pi|}$ we have the asymptotic $\OO(M^{-1})$, and
\item for $g(\theta) =  \mathbbm{1}_{\theta\geq \pi}-\mathbbm{1}_{\theta < \pi}$ we have the asymptotic $\OO(M^{-0.5})$.
\end{itemize}


For the perturbed circle domain, Fig. \ref{fig:PerturbedDom}, the corresponding numerical asymptotics of the $L^2(\del\Om)$ norm in the amount of eigenvalues $M$ used are shown in Figure \ref{fig:DiriErrorFigPertCircle}.
\begin{itemize}
\item For $g(\theta) = \exp(\cos(\theta))$ we have the asymptotic $\OO(M^{-3.5})$,
\item for $g(\theta) = |\theta-\pi|$ we have the asymptotic $\OO(M^{-3/2})$,
\item for $g(\theta) = \sqrt{|\theta-\pi|}$ we have the asymptotic $\OO(M^{-1})$, and
\item for $g(\theta) =  \mathbbm{1}_{\theta\geq \pi}-\mathbbm{1}_{\theta < \pi}$ we have the asymptotic $\OO(M^{-0.5})$.
\end{itemize}


For the square domain, Fig. \ref{fig:squareDomain}, the corresponding numerical asymptotics of the $L^2(\del\Om)$ norm in the amount of eigenvalues $M$ used are shown in Figure \ref{fig:DiriErrorFigSquare}. 
\begin{itemize}
\item For $g(\theta) = \exp(\cos(\theta))$ we have the asymptotic $\OO(M^{-3/2})$,
\item for $g(\theta) = |\theta-\pi|$ we have the asymptotic $\OO(M^{-3/2})$ too,
\item for $g(\theta) = \sqrt{|\theta-\pi|}$ we have the asymptotic $\OO(M^{-1})$, and
\item for $g(\theta) =  \mathbbm{1}_{\theta\geq \pi}-\mathbbm{1}_{\theta < \pi}$, where the discontinuities do not happen at the corners, we have the asymptotic $\OO(M^{-0.5})$.
\end{itemize}
\vspace{0.2cm}

\underline{The Robin Problem}\\
Considering the Robin problem, we use the same four domains to test Theorem \ref{thm:main}, but we use the same function. That function is $g(x)=\tfrac{1}{2}\log(|x-y|^2)$ for $x\in\del\Om$ and $y\not\in\overline{\Om}$, where we specifically have chosen $y=(0,2.5)^\TransT$. We then have that
\begin{align*}
	\del_{\nu_x}g(x) = \frac{\nu_x\cdot(x-y)}{|x-y|^2}\,.
\end{align*} 
The Robin constant $b$ is set to $1.5$, The number of nodes $L$ is $2^{11}$ and the nodes are equidistant distributed on $[0,2\pi)$. The corresponding numerical asymptotics of the $L^2(\del\Om)$ norm in the amount of eigenvalues and eigenfunctions used are shown in Figure \ref{fig:RobinErrorFig}. 
\begin{itemize}
\item For the sine mapped domain (Fig. \ref{fig:sinDomain}) we have an exponential decay,
\item for the kite shaped domain (Fig. \ref{fig:kiteDomain}) we have an exponential decay,
\item for the perturbed circle domain (Fig. \ref{fig:PerturbedDom}) we have the asymptotic $\OO(M^{-3.5})$, and
\item for the square (Fig. \ref{fig:squareDomain}) we have the asymptotic $\OO(M^{-1.5})$.
\end{itemize}

Next we used the following non-smooth function 
\begin{align*}
	g_0(x_1,x_2) =&\pi x_2 
        +((x_1+1)^2 + x_2^2)\log(x_1^2+x_2^2)\\
        &+2\, x_2 \tan^{-1}\Big(\frac{x_1}{x_2}\Big)
        -\frac{(x_1+1)\log(x_1^2+x_2^2)}{(x_1+1)^2 + x_2^2}\,.
\end{align*}
$g_0(x)$ is harmonic for $|x_2|>0$. It is not smooth around $x_1=0, x_2=0$.
Then we define the function $g(x) = g_0(x-y)$ where $y$ is a point on the boundary of the applied domain. 
The Robin constant $b$ is set to $1.5$, The number of nodes $L$ is $2^{11}$ and the nodes are equidistant distributed on $[0,2\pi)$. The corresponding numerical asymptotics of the $L^2(\del\Om)$ norm in the amount of eigenvalues and eigenfunctions used are shown in Figure \ref{fig:RobinErrorFigNonSmoothData}.
\begin{itemize}
\item For the sine mapped domain (Fig. \ref{fig:sinDomain}) we have the asymptotic $\OO(M^{-1.5})$,
\item for the kite shaped domain (Fig. \ref{fig:kiteDomain}) we suppose we have the asymptotic $\OO(M^{-1.5})$,
\item for the perturbed circle domain (Fig. \ref{fig:PerturbedDom}) we have the asymptotic $\OO(M^{-1.5})$, and
\item for the square (Fig. \ref{fig:squareDomain}) we did not get good results.
\end{itemize}
\vspace{0.2cm}

\underline{Considering the $L^2(\Omega)$ norm}\\
We evaluate the $L^2(\Omega)$ norm of the error between the exact function and the series expansion in terms of Steklov functions. To this end we need the Steklov functions evaluated on the boundary. We achieve this using Green's identity, we then obtain that
\begin{align*}
	s_i(x) = \int_{\del\Om} \del_{\nu}\Gamma(x,y)\, s_i(y)\intd \sigma_y
			  -\lambda_i\int_{\del\Om} \Gamma(x,y)\, s_i(y)\intd \sigma_y\,,
\end{align*} 
where $x\in\Om$, $y\in\del\Om$, and $\Gamma(x,y)= \frac{1}{2\pi}\log(|x-y|)$ is the fundamental solution to the Laplace problem. The numerical evaluation of $s_i(x)$ is not very accurate when $x$ is close to the boundary. Thus we have a threshold, which is 8 times the distance between two neighbouring nodes on the boundary.
For the numerical evaluation of the $L^2(\Omega)$ norm we apply the polar transformation for the first three domains and we evaluate $s_i(x)$ exactly inside the square domain using an equidistant grid. The transformation has the form
\begin{align*}
	\int_\Om F(x,y)\intd (x,y) = \int_0^1\int_0^{2\pi} F(\phi(r,\theta))|\det\mathrm{D}\phi(r,\theta)|\intd \theta\intd r\,,
\end{align*} 
where $F: \Om \rightarrow [0,\infty)$, $\phi: [0,1]\times[0,2\pi]\rightarrow \Om$, and $\mathrm{D}$ denotes the Jacobian matrix. We define $\phi$ through $\phi(r,\theta) = r\,f(\theta)$, where $f$ is given through $\sin(\exp(\imagi \theta))$ for the sine mapped domain, through Equation (\ref{equ:kiteDomain}) for the kite shaped domain and through Equation \ref{equ:PerturbedDom} for the perturbed domain. Then we use an grid, which is equidistant in each dimension, on the $(r,\theta)$ domain. 

We use $L=2^{13}$ number of nodes on the boundary and $2^6\times 2^9$ points on the $(r,\theta)$ domain. As exact comparing function we use the same as in the Robin problem, that is $g(x)=\tfrac{1}{2}\log(|x-(0,2.5)^\TransT|)$.
\\ \\
\noindent For the Dirichlet problem we have following results:
\begin{itemize}
\item For the sine mapped domain (Fig. \ref{fig:sinDomain}) we have an exponential decay,
\item for the kite shaped domain (Fig. \ref{fig:kiteDomain}) we have an exponential decay,
\item for the perturbed circle domain (Fig. \ref{fig:PerturbedDom}) we have the asymptotic $\OO(M^{-3.5})$, and
\item for the square (Fig. \ref{fig:squareDomain}) we have the asymptotic $\OO(M^{-2})$.
\end{itemize}
The corresponding numerical asymptotics of the $L^2(\del\Om)$ norm in the amount of eigenvalues and eigenfunctions used are shown in Figure \ref{fig:DomDiriErrorFig}.
\\ \\
For the Robin problem we have following results:
\begin{itemize}
\item For the sine mapped domain (Fig. \ref{fig:sinDomain}) we have an exponential decay,
\item for the kite shaped domain (Fig. \ref{fig:kiteDomain}) we have an exponential decay,
\item for the perturbed circle domain (Fig.\ref{fig:PerturbedDom}) we have the asymptotic $\OO(M^{-3.5})$, and
\item for the square (Fig. \ref{fig:squareDomain}) we have the asymptotic $\OO(M^{-2})$.
\end{itemize}
The corresponding numerical asymptotics of the $L^2(\del\Om)$ norm in the amount of eigenvalues and eigenfunctions used are shown in Figure \ref{fig:DomRobinErrorFig}.
\vspace{0.5cm}

\underline{Testing the Asymptotics on the Number $N$}
\vspace{0.2cm}\\
We test Theorem \ref{thm:main} in its statement about the number $N$ of terms in the series expansion for the numerical approximation. To this we apply the test functions $g(\theta) = \exp(\cos(\theta))$, $g(\theta) = |\theta-\pi|$ on the two domain given in Fig \ref{fig:sinDomain}, the sine-mapped domain, and Fig. \ref{fig:PerturbedDom} the perturbed circle shaped domain. We picked $M=20$, that is we numerically evaluate the error between the correct, theoretical, series expansion in terms of the Steklov eigenpair truncated to the first 20 Steklov eigenfunctions and between the numerical approximation with respect to $N$. Here we use throughout the Collocation method (Method 3), this implies $L=N$. We approximate the correct, theoretical, solution through an evaluation with $N=2400$.

We consider first the error on the boundary for the Dirichlet problem.
\begin{itemize}
\item For $g(\theta) = \exp(\cos(\theta))$ on the sine-mapped domain (Fig. \ref{fig:sinDomain}) we observe a very fast convergence. We assume we have here an exponential convergence.
\item For $g(\theta) = \exp(\cos(\theta))$ on the perturbed circle shaped domain domain (Fig \ref{fig:PerturbedDom}) we observe the asymptotic $\OO(N^{-4})$.
\item For $g(\theta) = |\theta-\pi|$ on the sine-mapped domain (Fig. \ref{fig:sinDomain}) we observe the asymptotic $\OO(N^{-2})$.
\item For $g(\theta) = |\theta-\pi|$ on the perturbed circle shaped domain domain (Fig. \ref{fig:PerturbedDom}) we observe the asymptotic $\OO(N^{-2})$.
\end{itemize}
The corresponding numerical asymptotics are depicted in Fig. \ref{fig:NErrorDiriBdry}.  We see that we have asymptotics in $N$ which depend on both, the regularity of the boundary and the regularity of the Robin boundary data. Theorem \ref{thm:main} states that the error is independent of the regularity of the Robin boundary data, however. We suppose this is due to the trapezoidal rule which we used for the term $\int_{\del\Om}g^{(D)}\tilde{s}_i$.
\section{Concluding Remarks}\label{sec:ConcludingRemarks}

In this paper we showed how we can use the Steklov eigenfunctions to approximate the solution to the Laplace equation with either Dirichlet or Robin boundary conditions. Using higher regularity results for layer potentials we proved that $s_i\in H^p(\del\Om)$ whenever $\del\Om\in H^{p+1/2}_{\text{loc}}$, this allowed us to conclude that the series expansion involving the Steklov eigenpairs approximates the solution to the homogeneous Laplace equation with an exponentially decreasing error given a smooth boundary. Furthermore, we provided a polynomial error approximation in terms of the smoothness of the boundary and in terms of the regularity of the boundary data. We then elaborated on three different methods to numerically compute the Steklov eigenfunctions, the conformal method, the weak form method and the collocation method. With those methods we were able to verify Theorem \ref{thm:main} on various domains and with various boundary data.

An interesting question arises here, namely can we smooth corners out and utilize the exponential fast approximation? Some preliminary tests on a square with rounded-off corners give a positive answer. Here we have to further consider how we can approximate the boundary data, when we smooth out corners. 

Another question is about higher dimensional application of the Steklov eigenfunctions. In this paper we only considered the two dimensional problem. Applying the results in \cite{AuchmutyFoundation} and \cite{HighRegLayerPotMazya}, we can show that Proposition \ref{prop:s_i in H^p} is also valid. For the theorem we further have to show that the regularity results for the Dirichlet to Neumann operator still hold, then we firmly believe that all the results in Theorem \ref{thm:main} are valid in higher dimensions. 

An other thought is about the applications of the Steklov eigenfunctions on different PDEs. We think that we have similar results on the homogeneous Helmholtz equation that is $(\Laplace + k^2)u=0$, for $k>0$. We are not sure about other PDE's like the Schroedinger-type operator \cite{AuchmutySchrodinger}.

\begin{figure}
    \includegraphics[width=.49\textwidth]{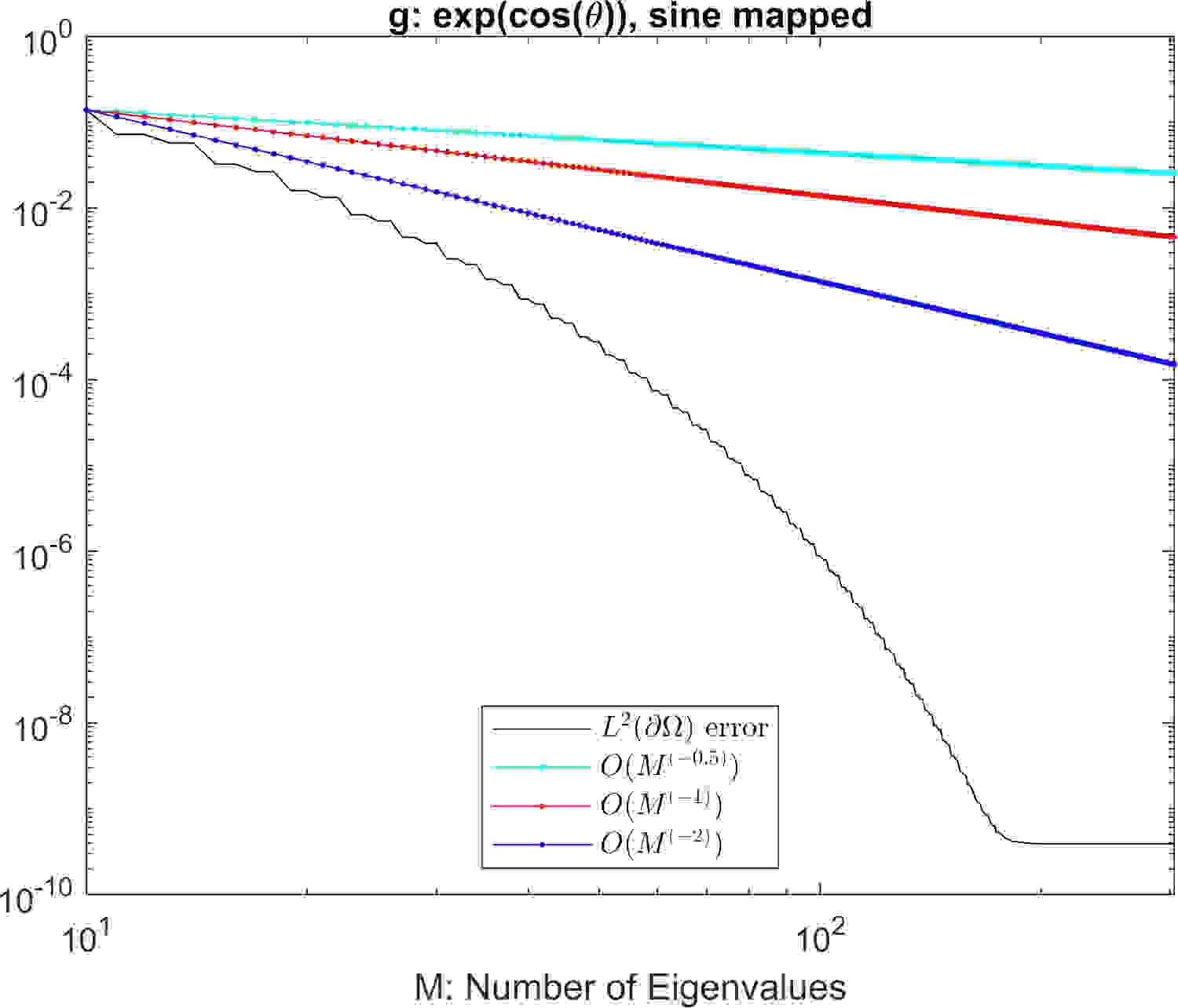}\hfill
    \includegraphics[width=.49\textwidth]{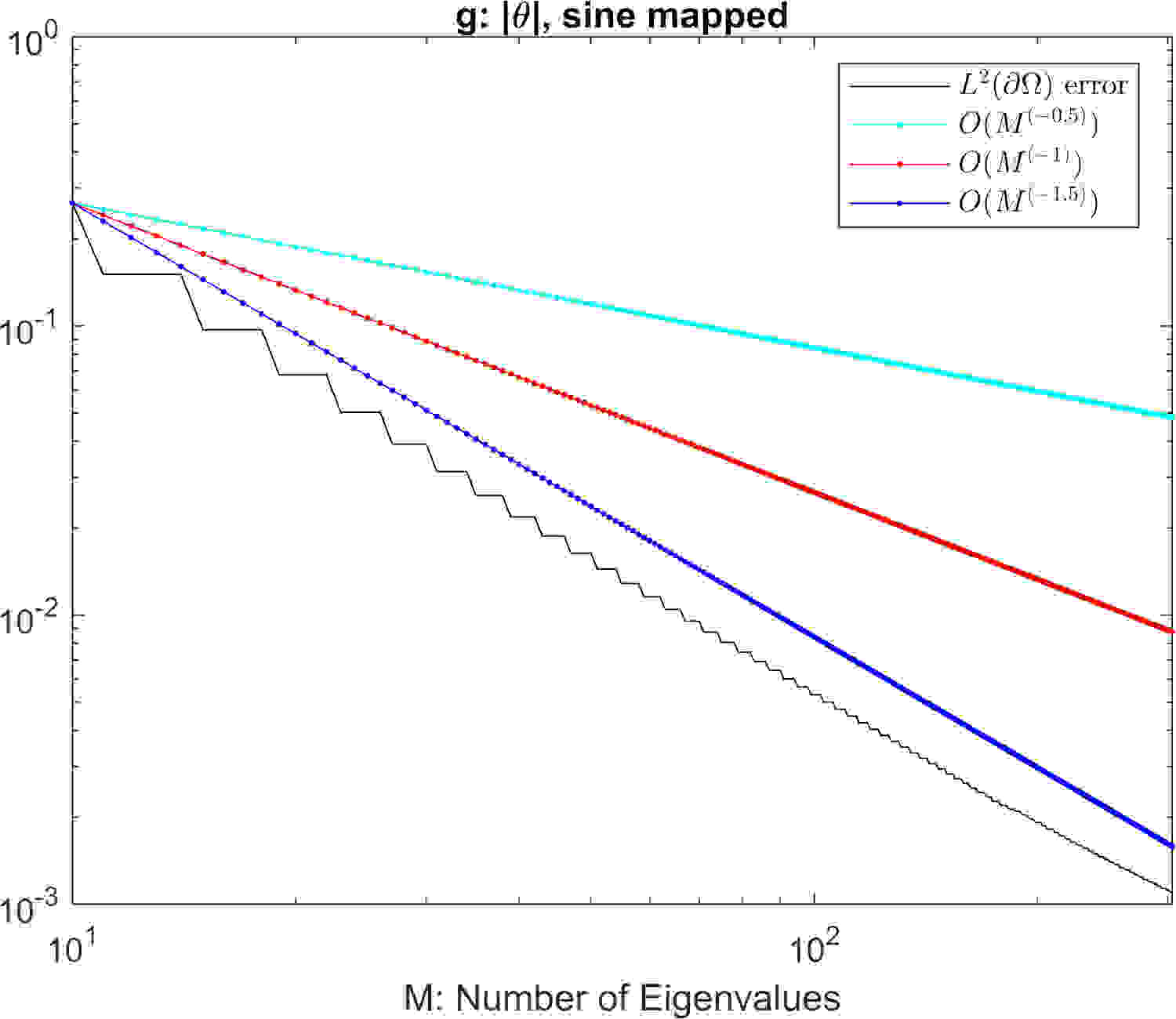}
    \\[\smallskipamount]
    \includegraphics[width=.49\textwidth]{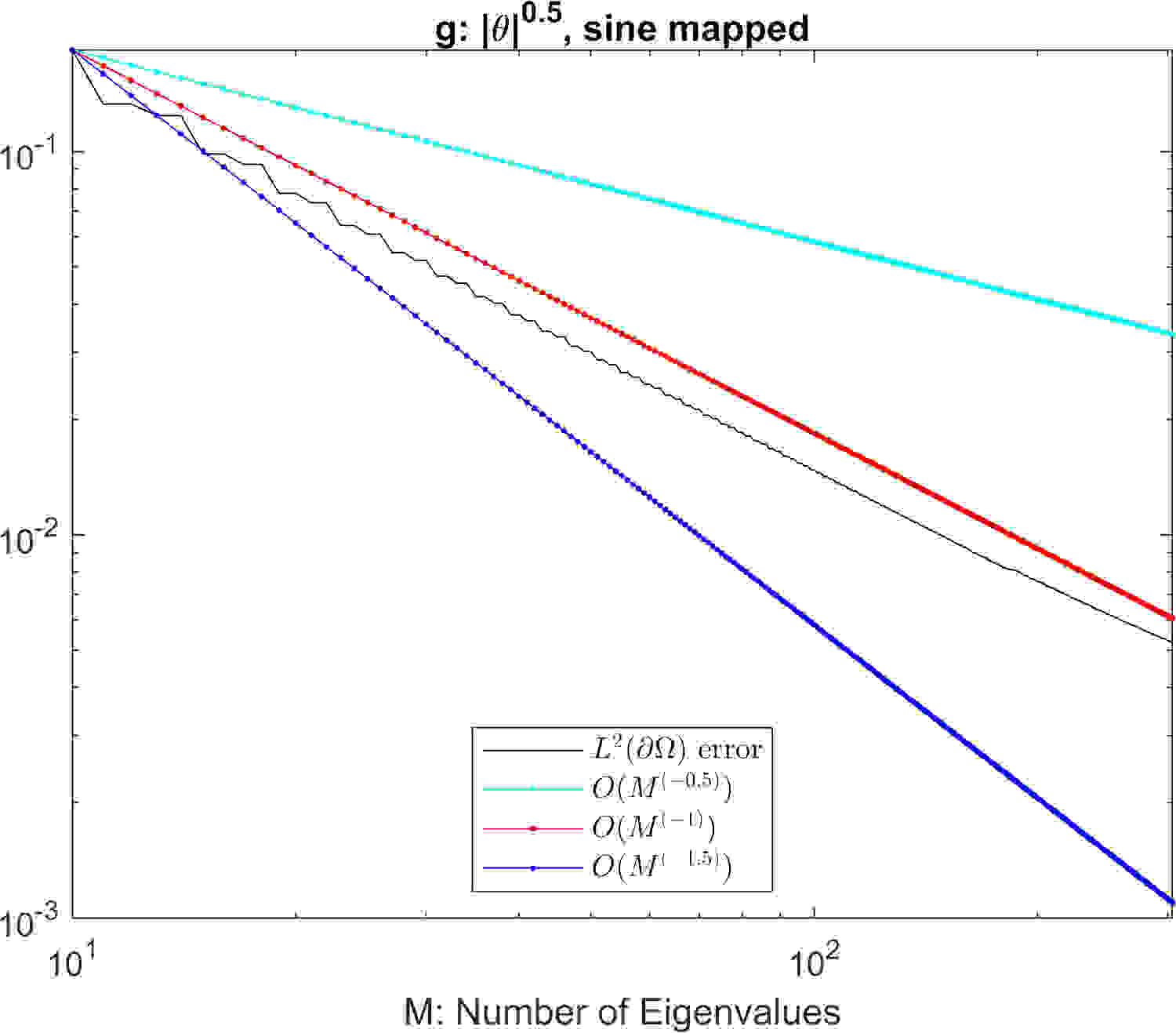}\hfill
    \includegraphics[width=.49\textwidth]{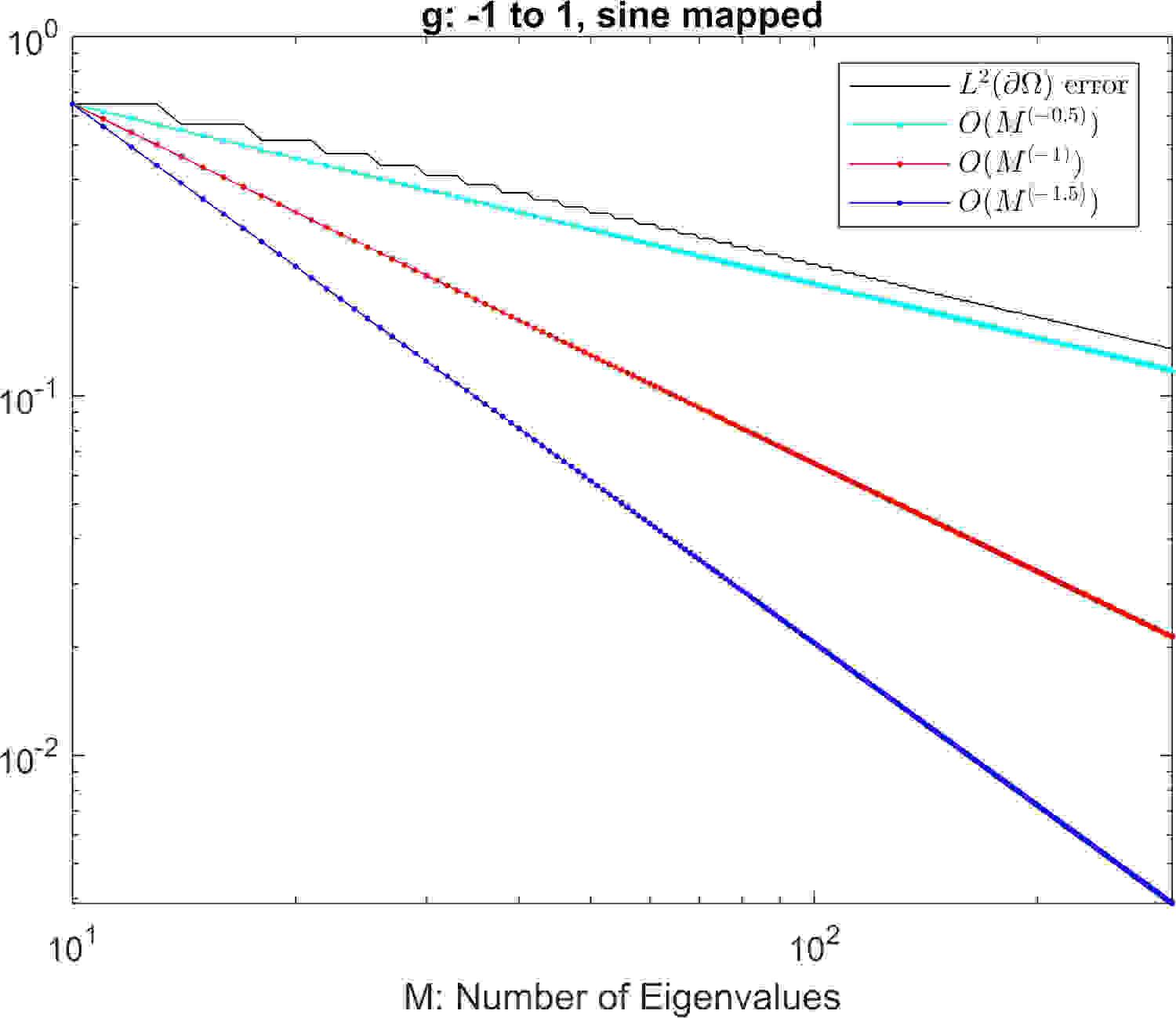}
    \caption{Error plots for the numerical approximation of the solution to the Dirichlet problem on the boundary through method one to four different boundary functions on the sine mapped domain, Fig \ref{fig:sinDomain}.}\label{fig:DiriErrorFigSine}
\end{figure}

\begin{figure}
    \includegraphics[width=.49\textwidth]{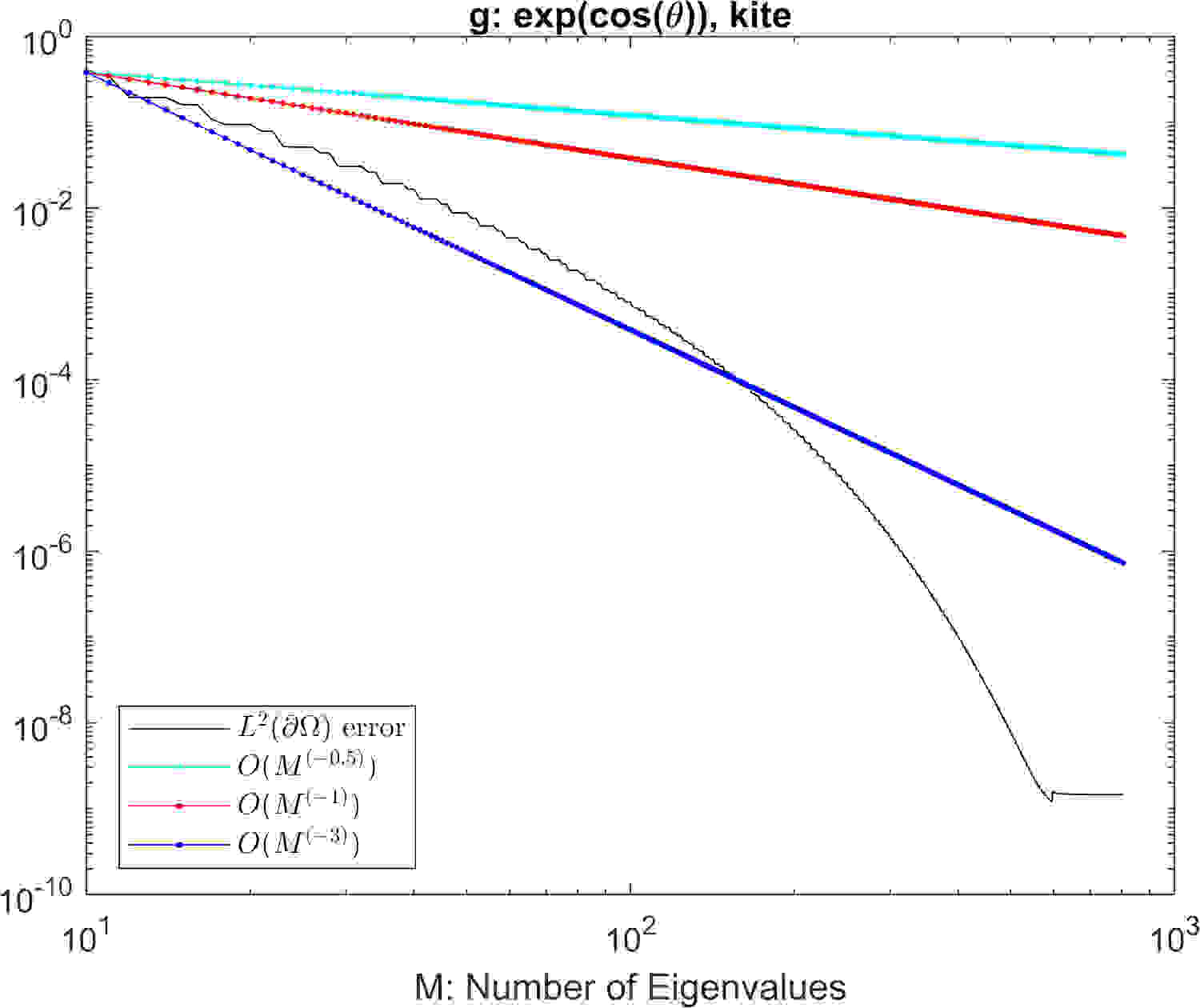}\hfill
    \includegraphics[width=.49\textwidth]{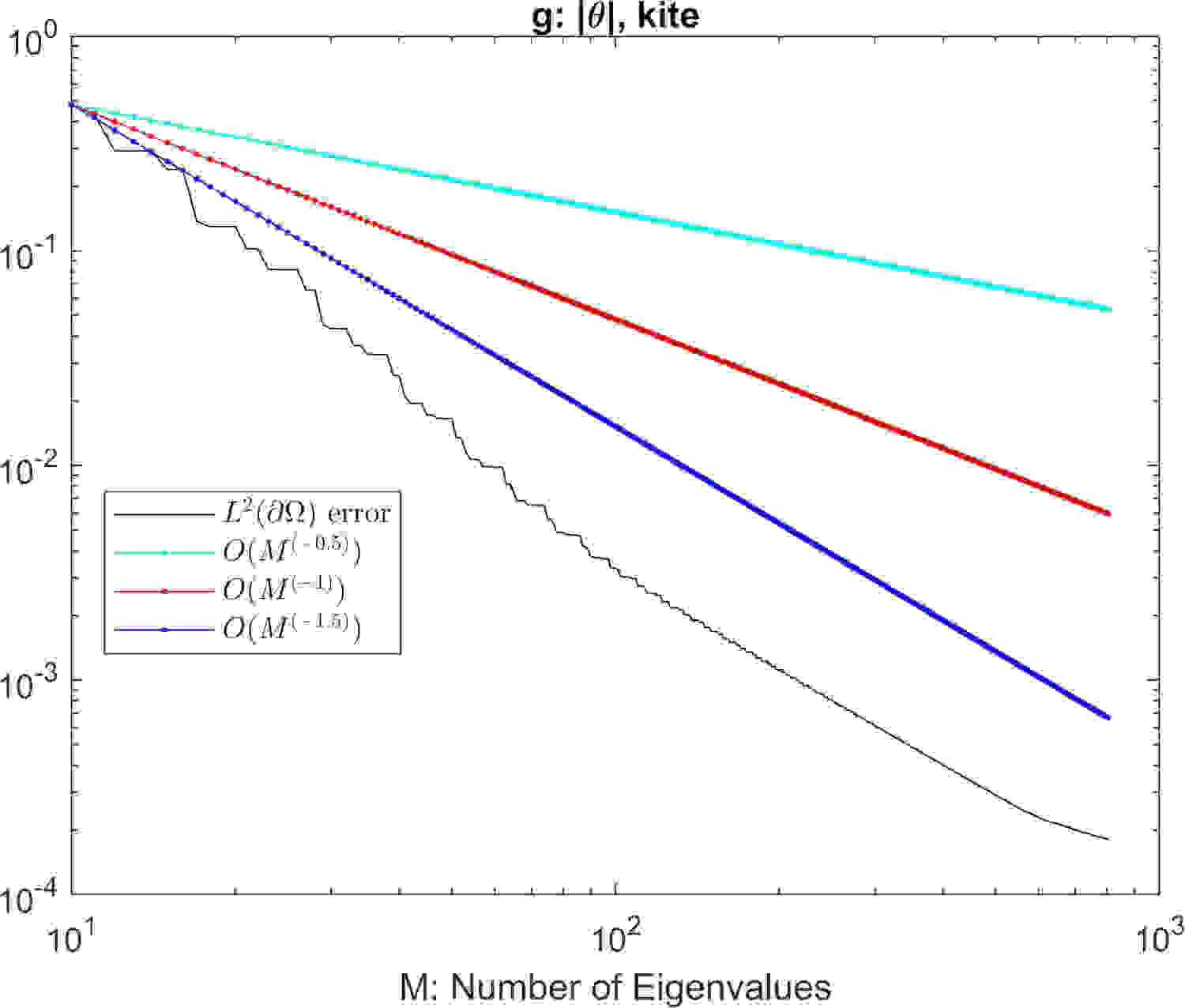}
    \\[\smallskipamount]
    \includegraphics[width=.49\textwidth]{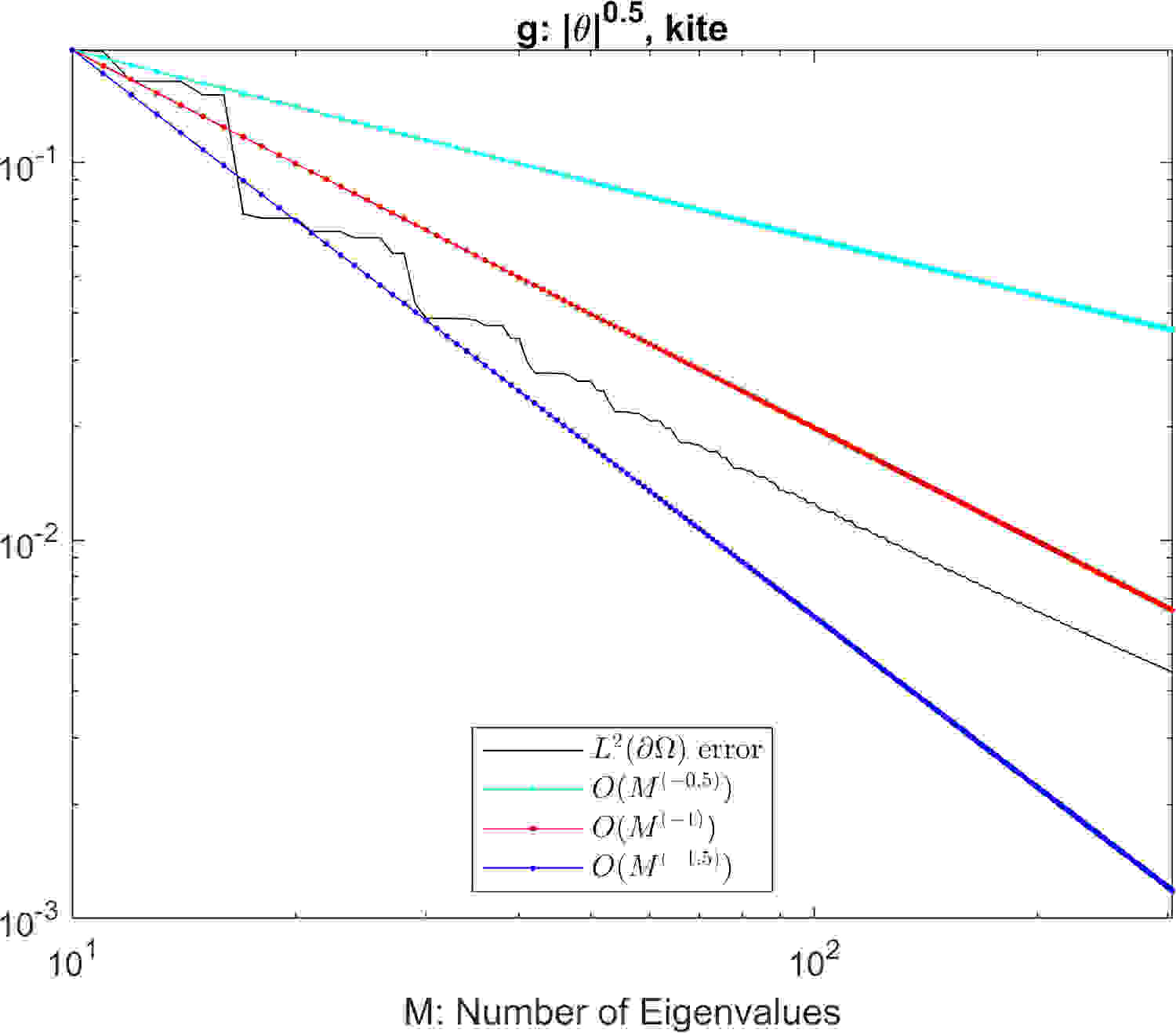}\hfill
    \includegraphics[width=.49\textwidth]{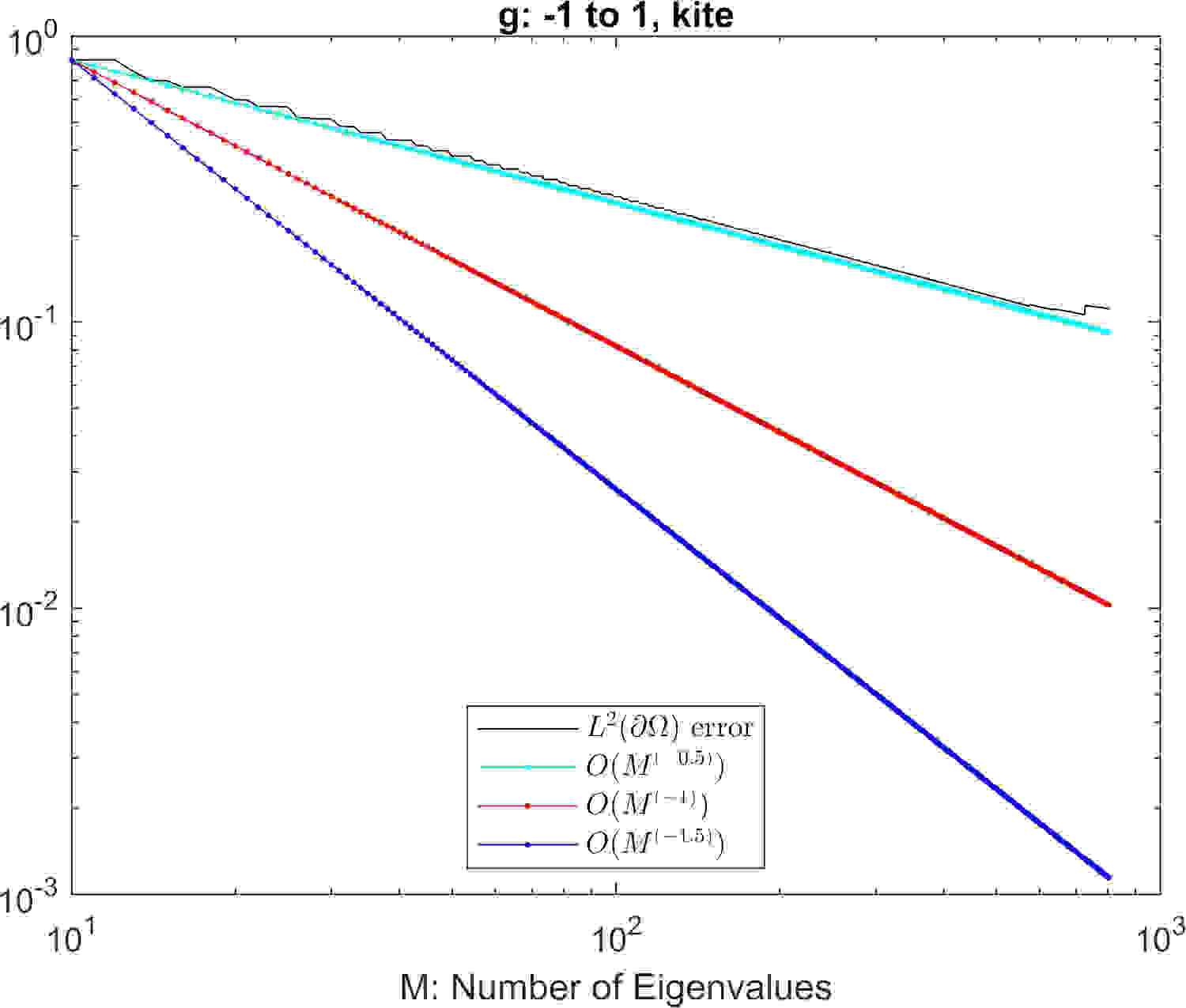}
    \caption{Error plots for the numerical approximation of the solution to the Dirichlet problem on the boundary to four different boundary functions on the kite shaped domain, Fig \ref{fig:kiteDomain}.}\label{fig:DiriErrorFigKite}
\end{figure}

\begin{figure}
    \includegraphics[width=.49\textwidth]{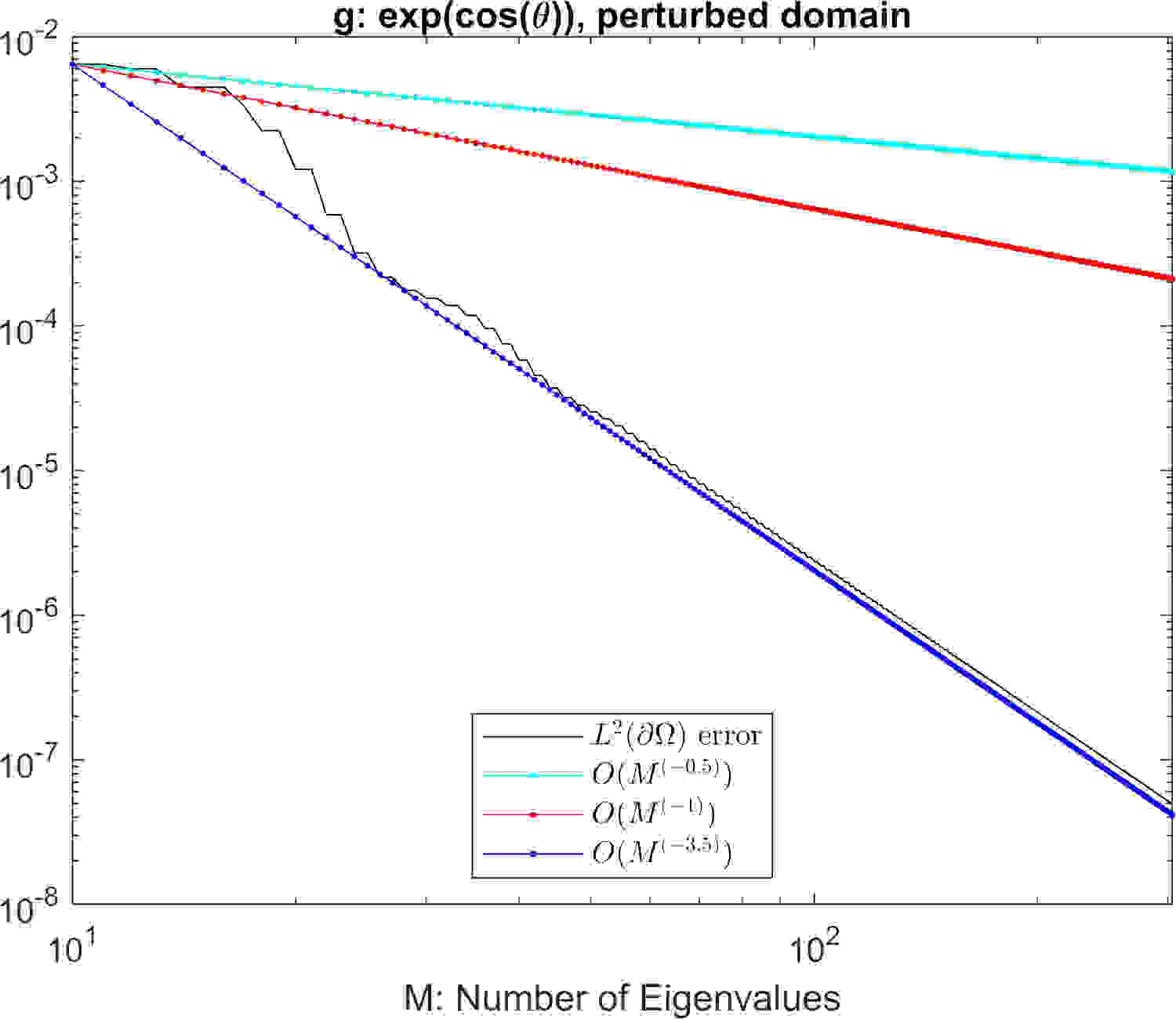}\hfill
    \includegraphics[width=.49\textwidth]{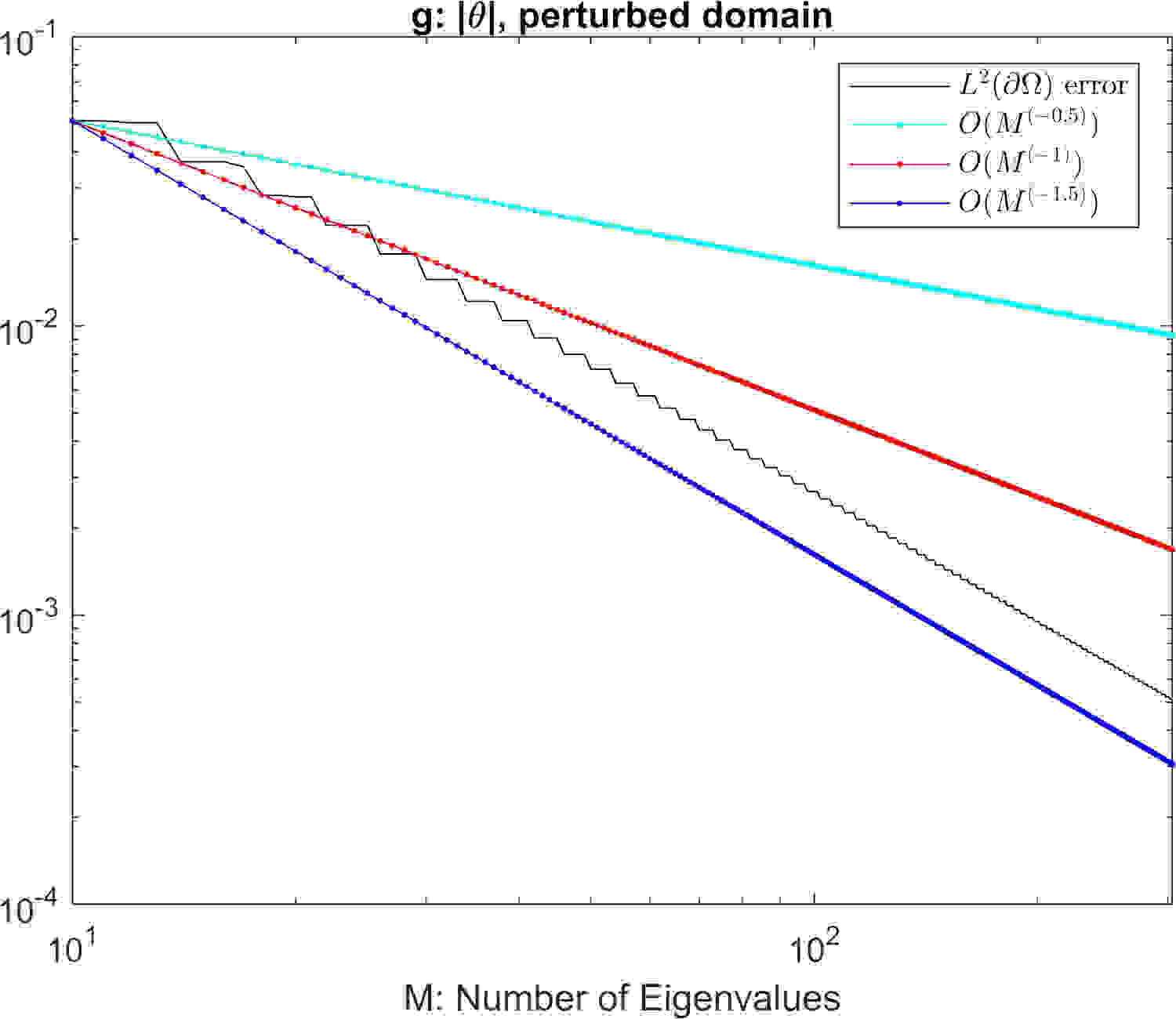}
    \\[\smallskipamount]
    \includegraphics[width=.49\textwidth]{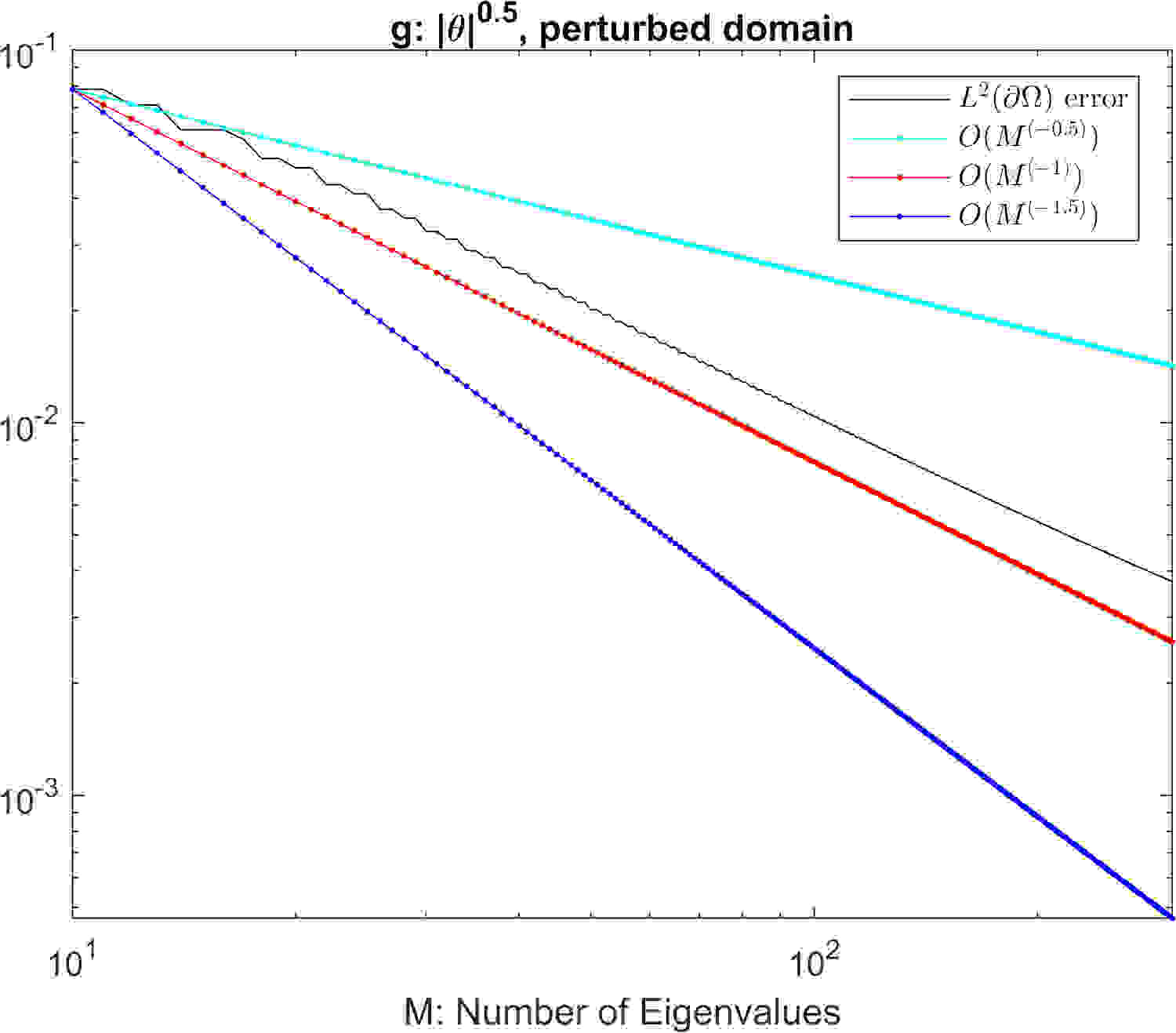}\hfill
    \includegraphics[width=.49\textwidth]{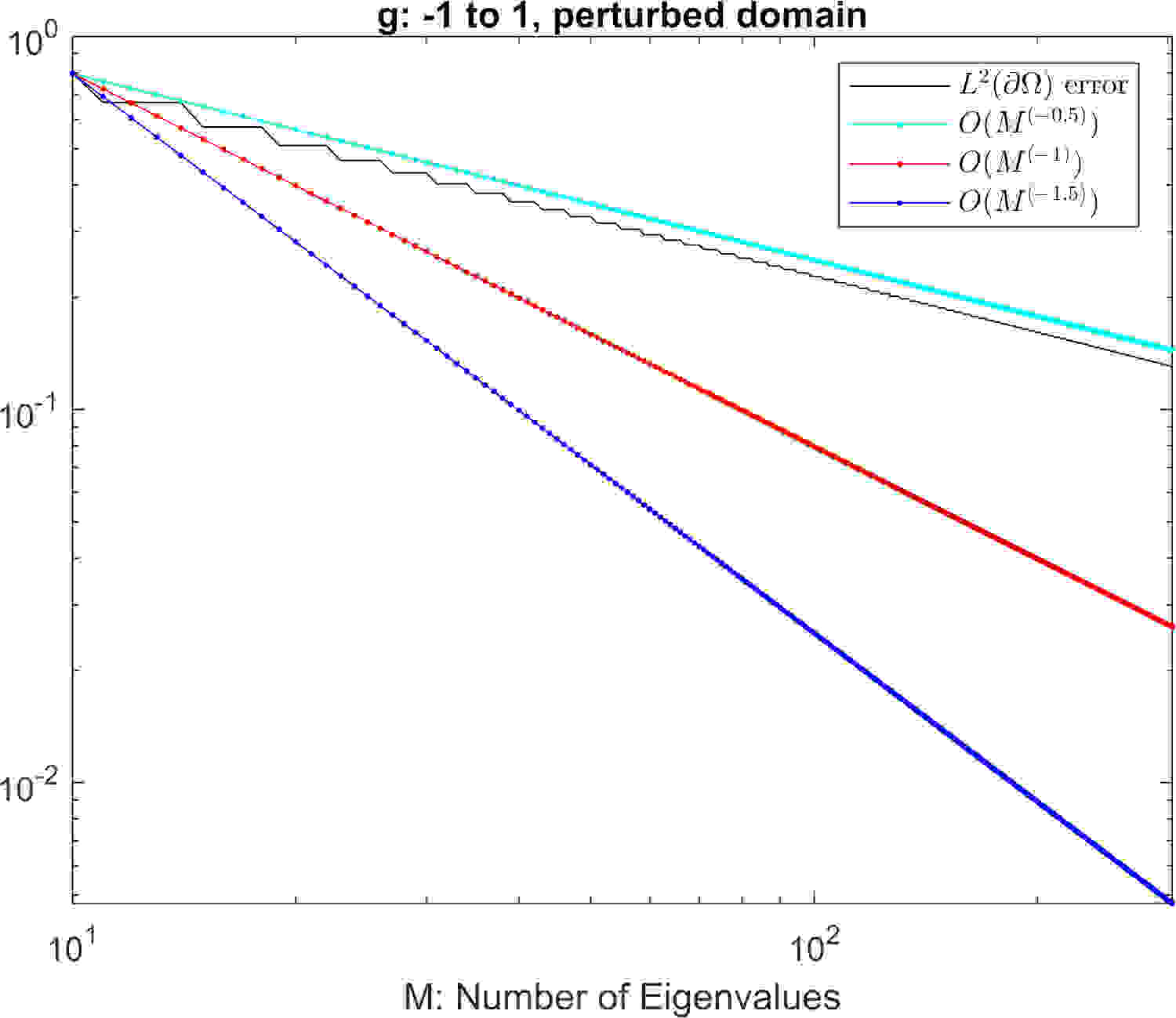}
    \caption{Error plots for the numerical approximation of the solution to the Dirichlet problem on the boundary to four different boundary functions on the perturbed circle, Fig \ref{fig:PerturbedDom}.}\label{fig:DiriErrorFigPertCircle}
\end{figure}

\begin{figure}
    \includegraphics[width=.49\textwidth]{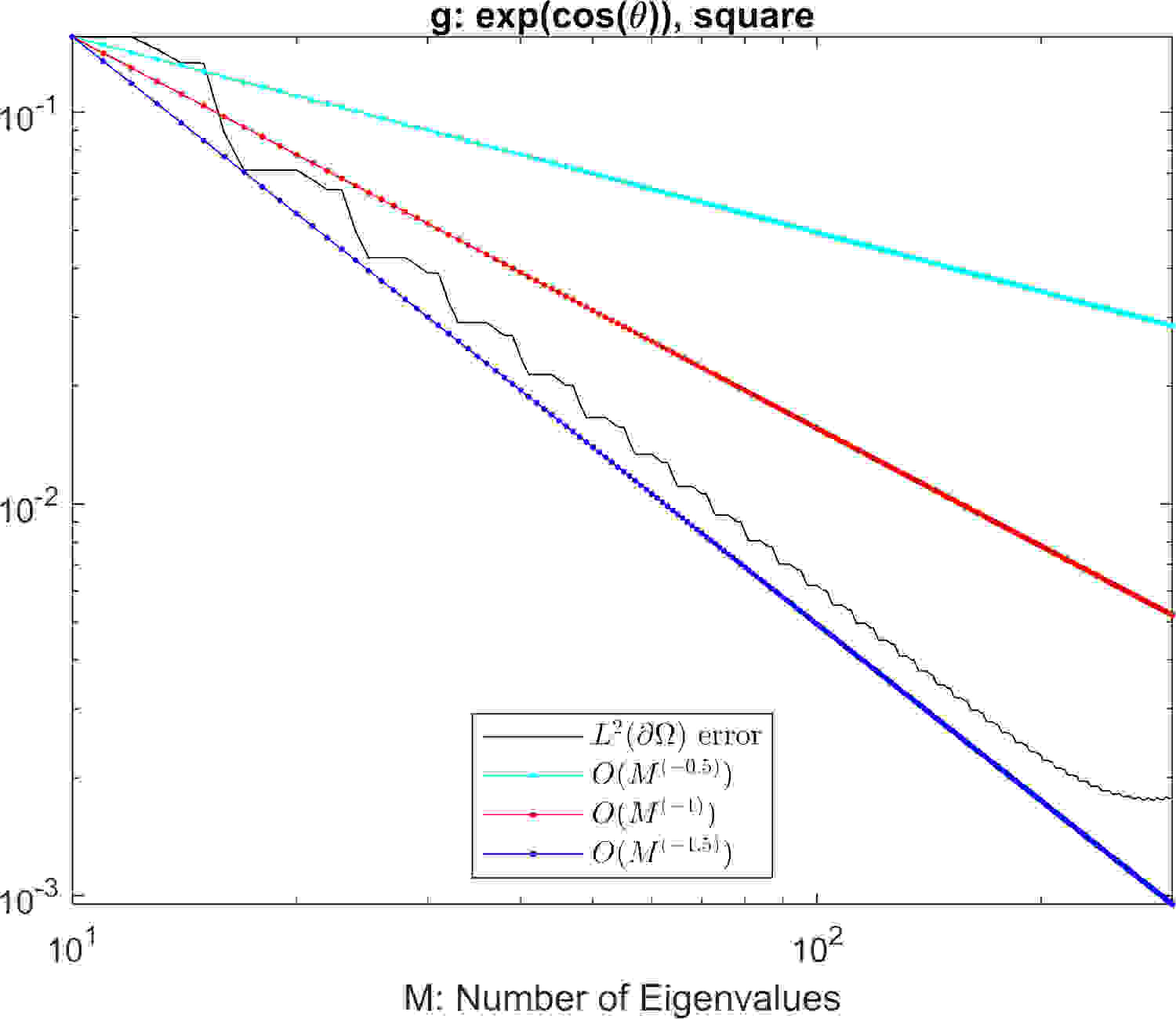}\hfill
    \includegraphics[width=.49\textwidth]{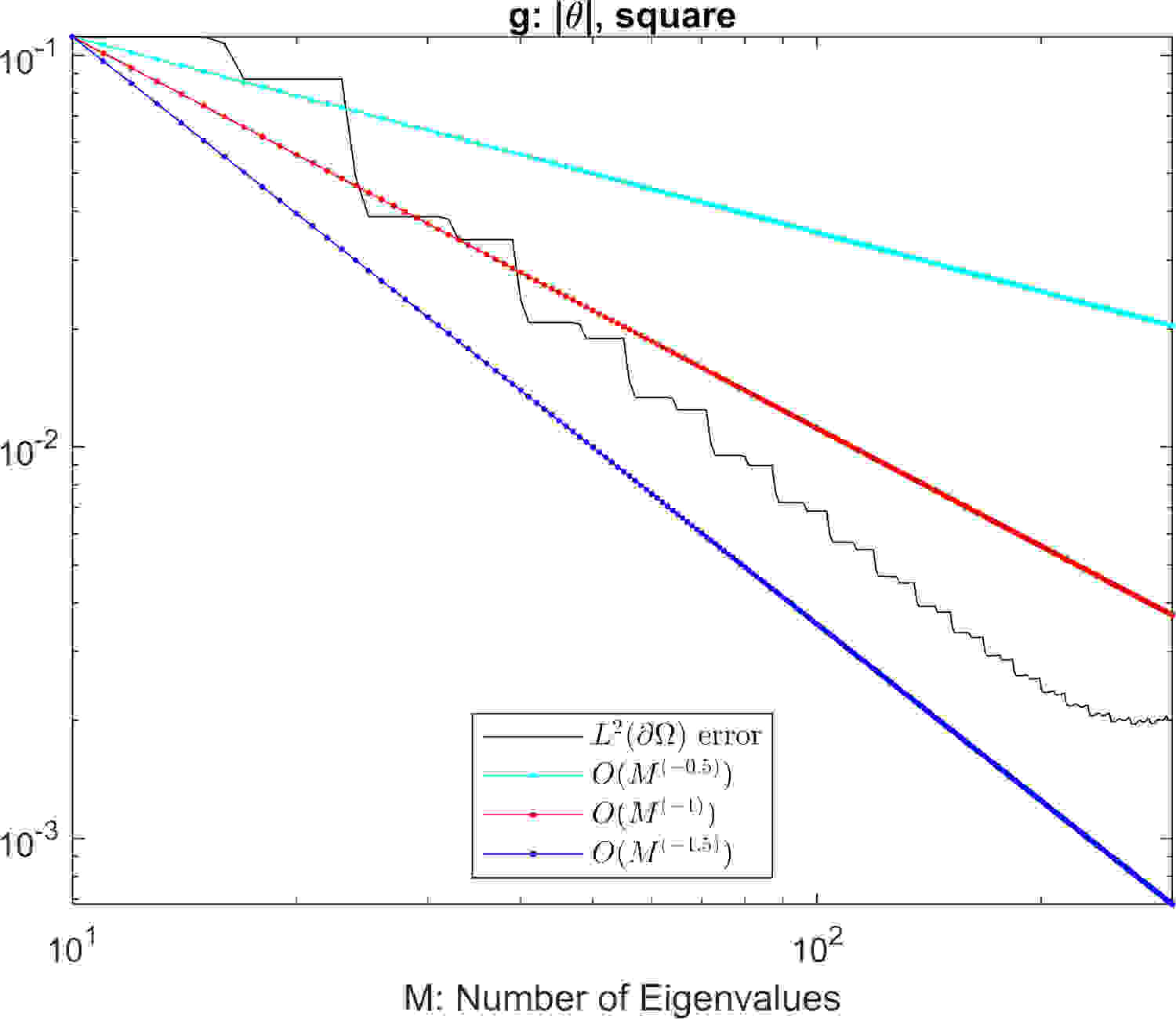}
    \\[\smallskipamount]
    \includegraphics[width=.49\textwidth]{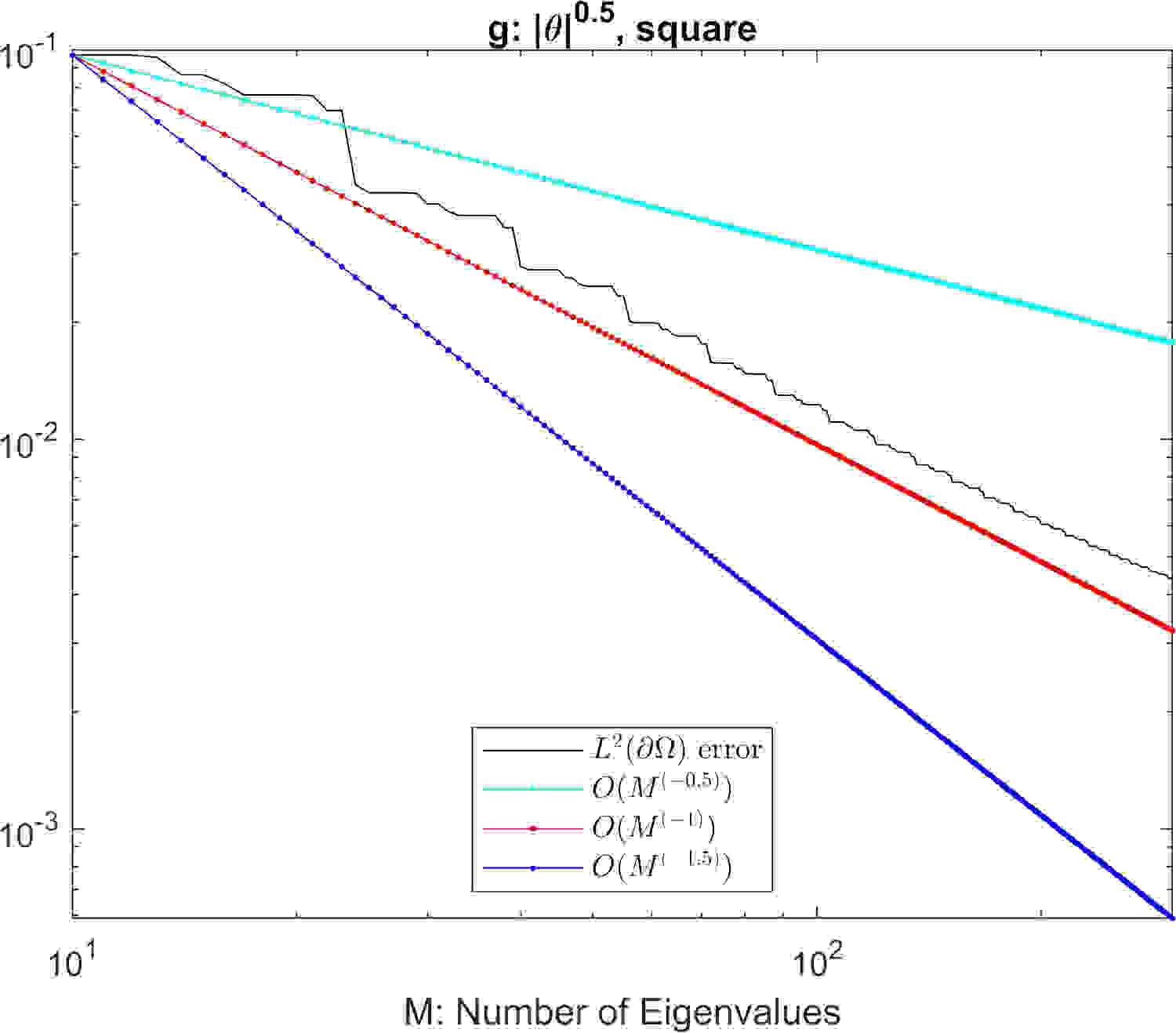}\hfill
    \includegraphics[width=.49\textwidth]{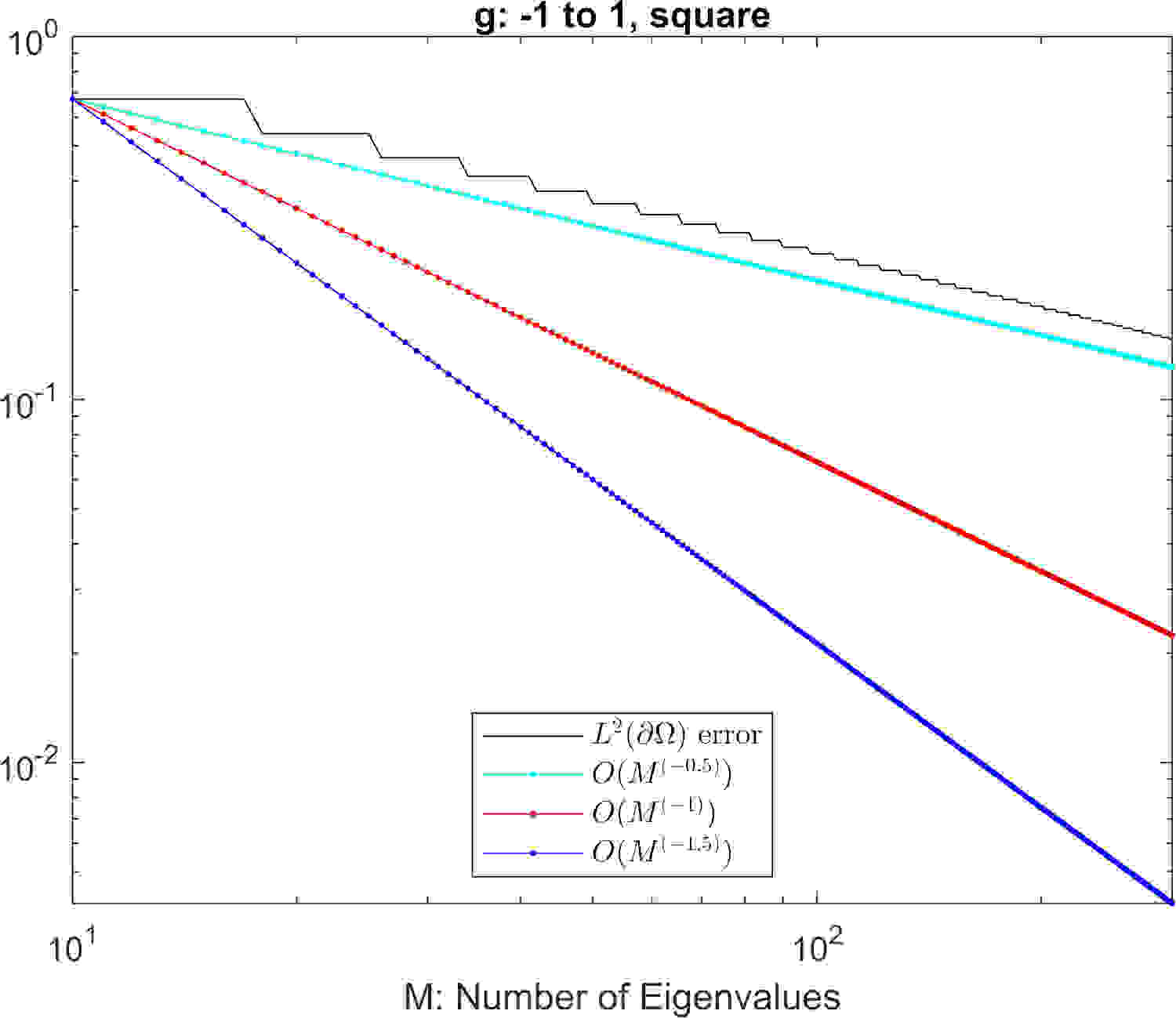}
    \caption{Error plots for the numerical approximation of the solution to the Dirichlet problem on the boundary to four different boundary functions on the square domain, Fig \ref{fig:squareDomain}.}\label{fig:DiriErrorFigSquare}
\end{figure}

\begin{figure}
    \includegraphics[width=.49\textwidth]{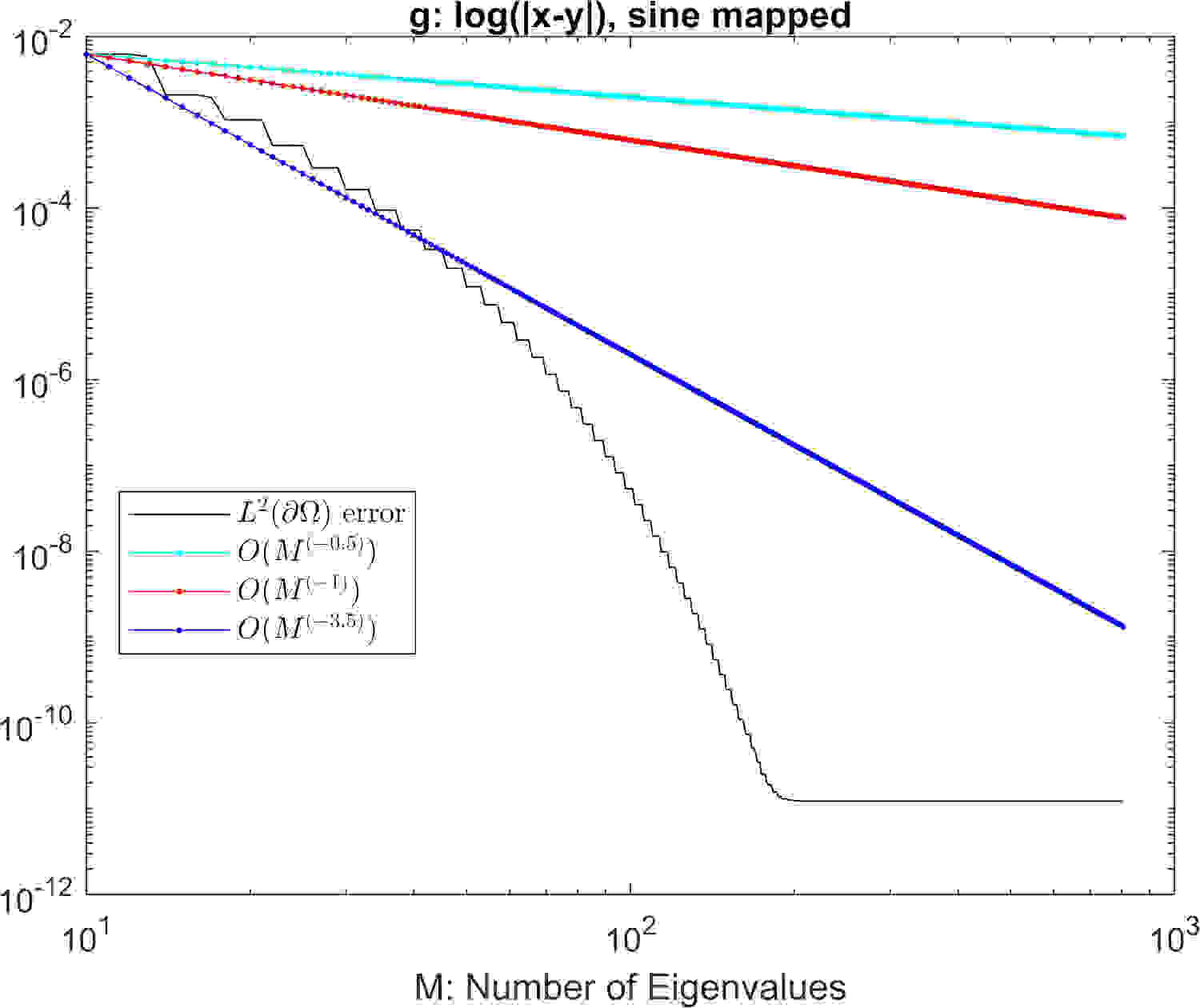}\hfill
    \includegraphics[width=.49\textwidth]{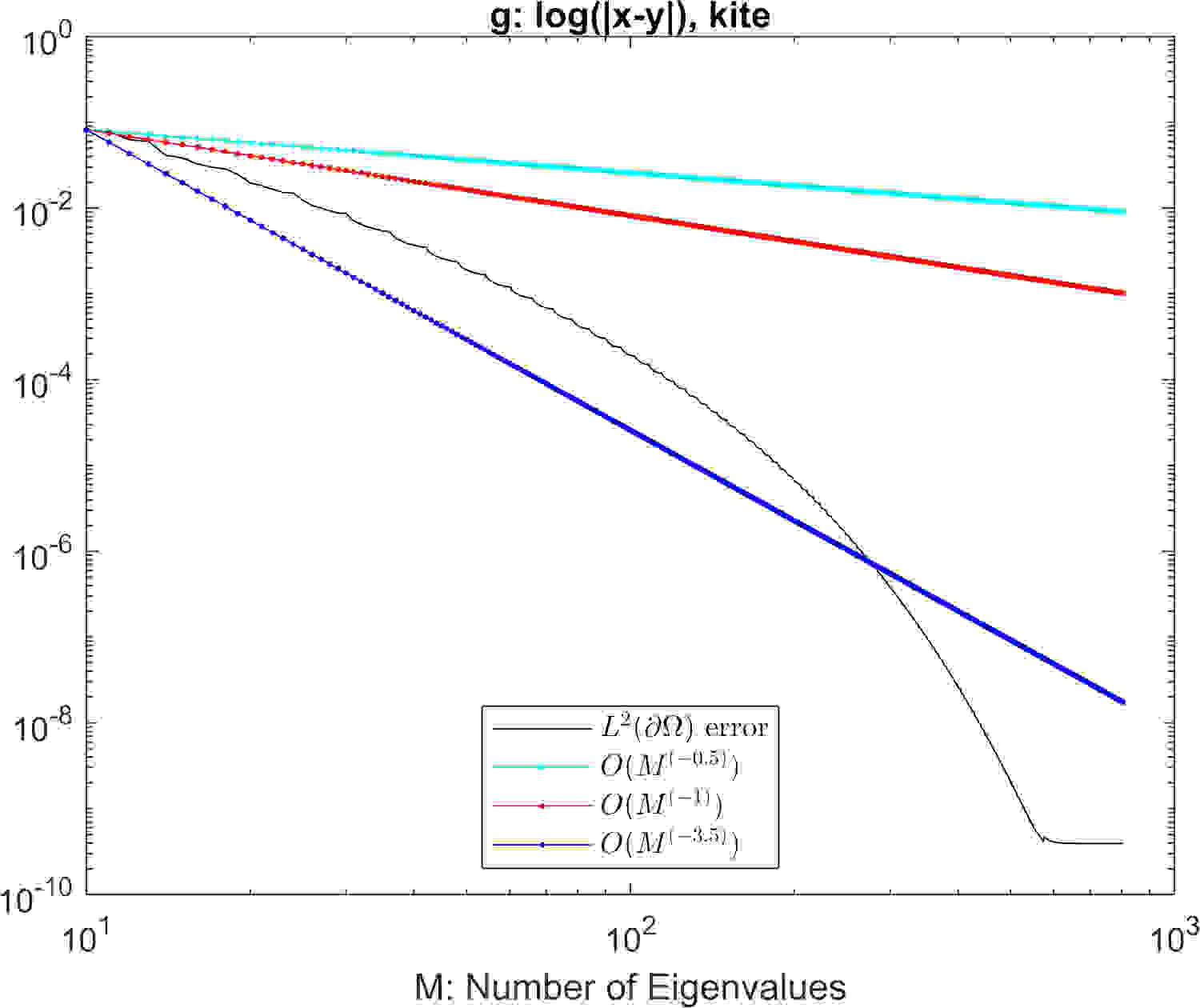}
    \\[\smallskipamount]
    \includegraphics[width=.49\textwidth]{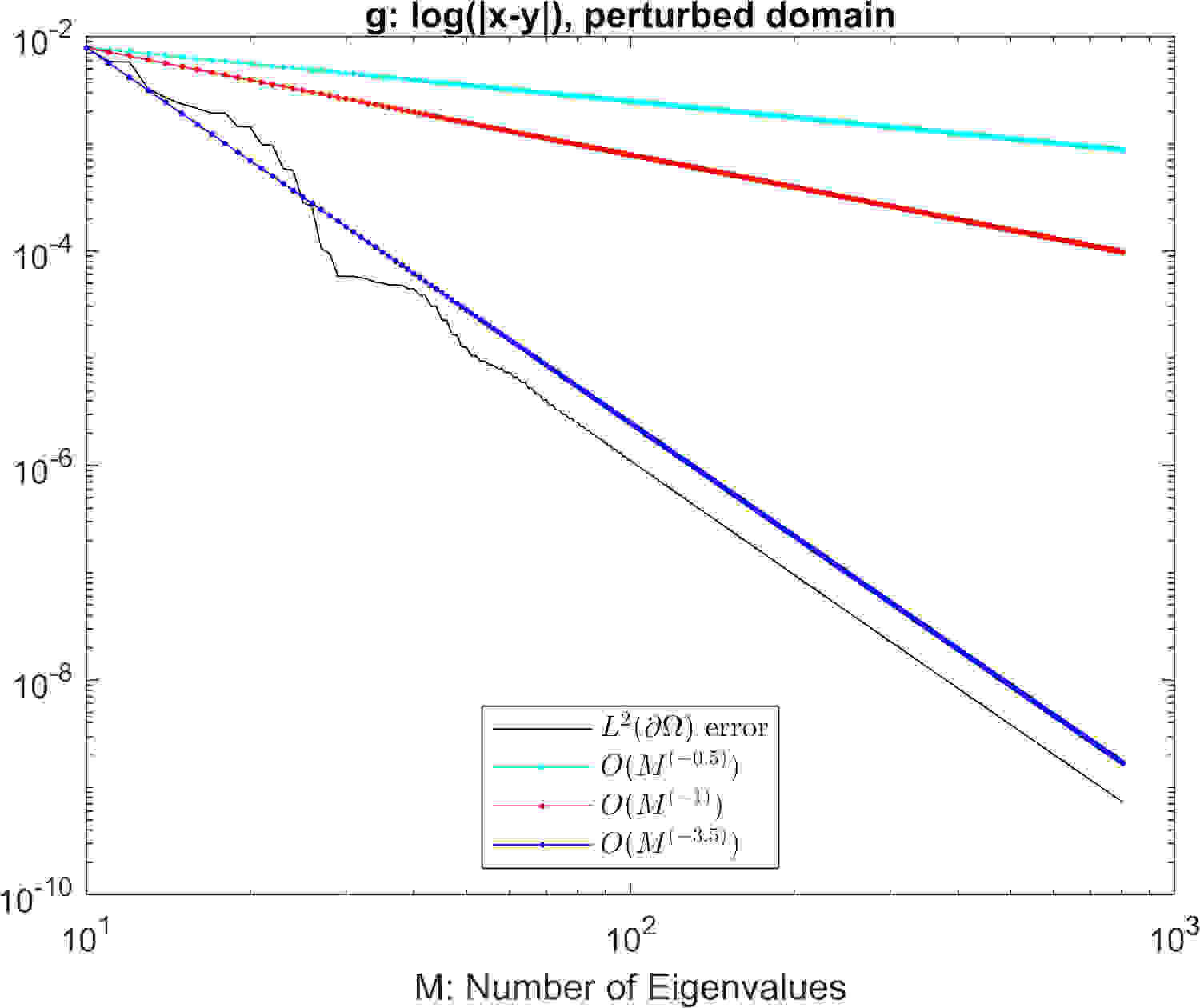}\hfill
    \includegraphics[width=.49\textwidth]{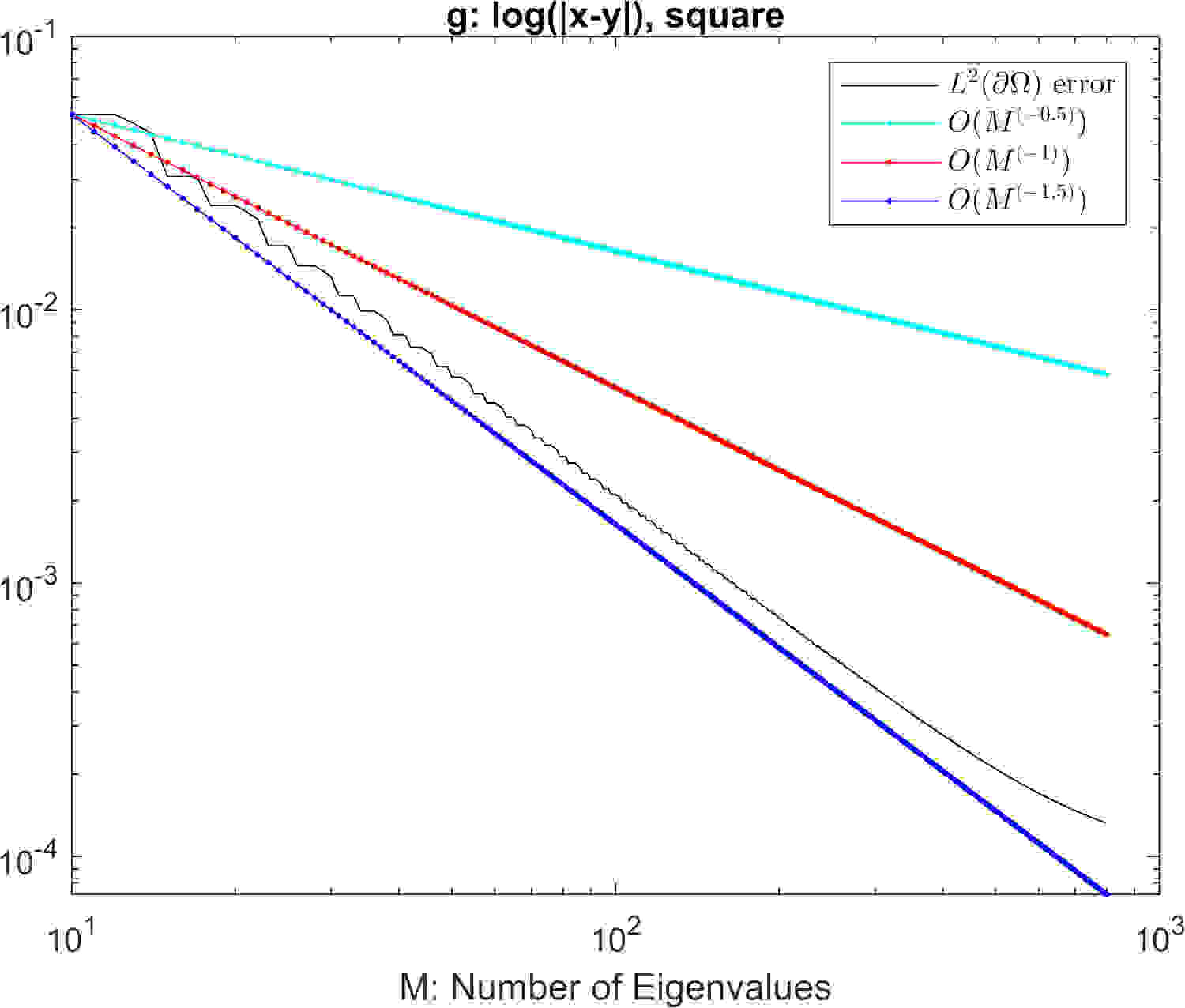}
    \caption{Error plots for the numerical approximation of the solution to the Robin boundary problem on the boundary to four different domains, the sine mapped domain (Fig. \ref{fig:sinDomain}), the kite shaped domain (Fig. \ref{fig:kiteDomain}), the perturbed circle domain (Fig. \ref{fig:PerturbedDom}) and the square (Fig. \ref{fig:squareDomain}).}\label{fig:RobinErrorFig}
\end{figure}

\begin{figure}
    \includegraphics[width=.49\textwidth]{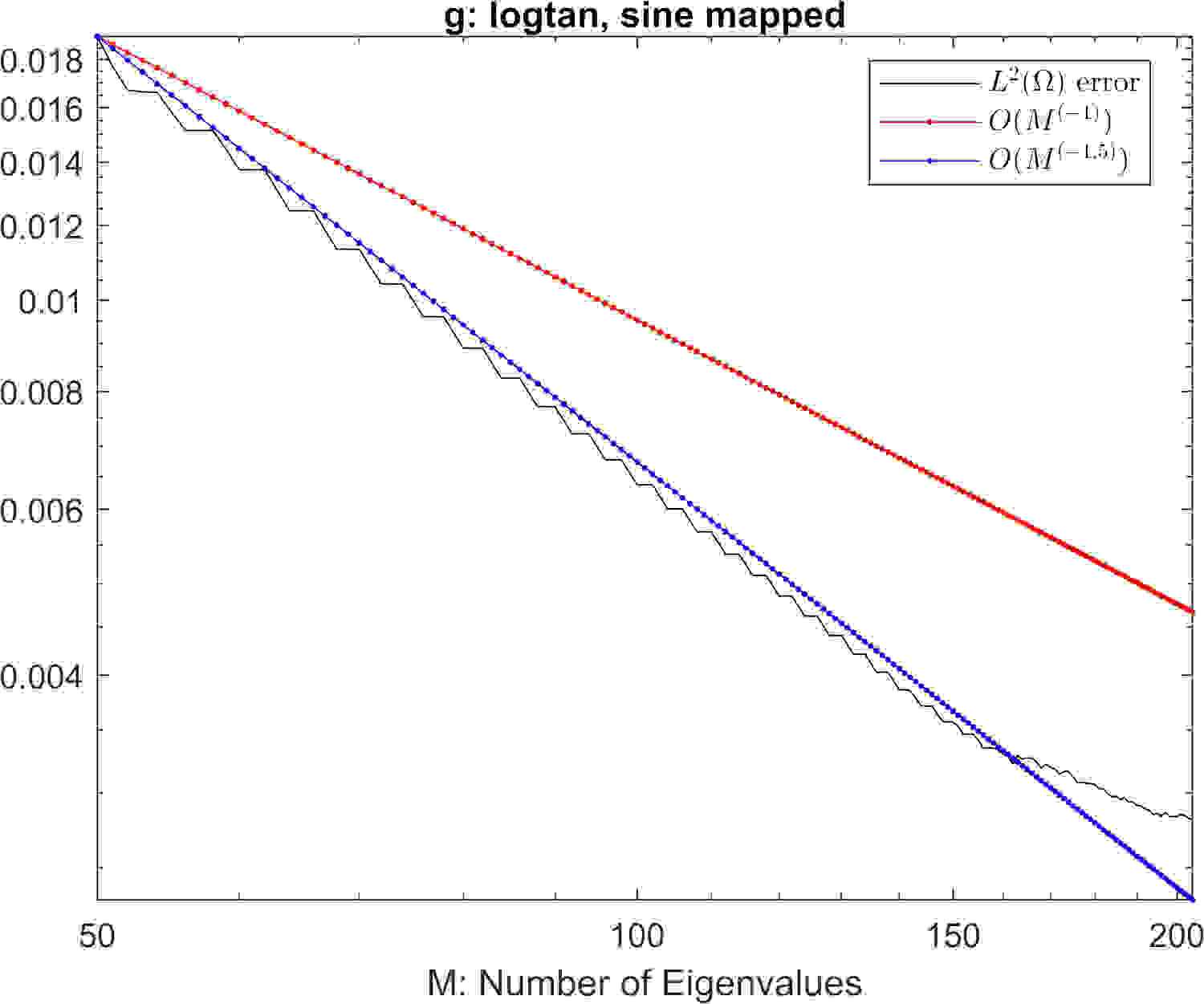}\hfill
    \includegraphics[width=.49\textwidth]{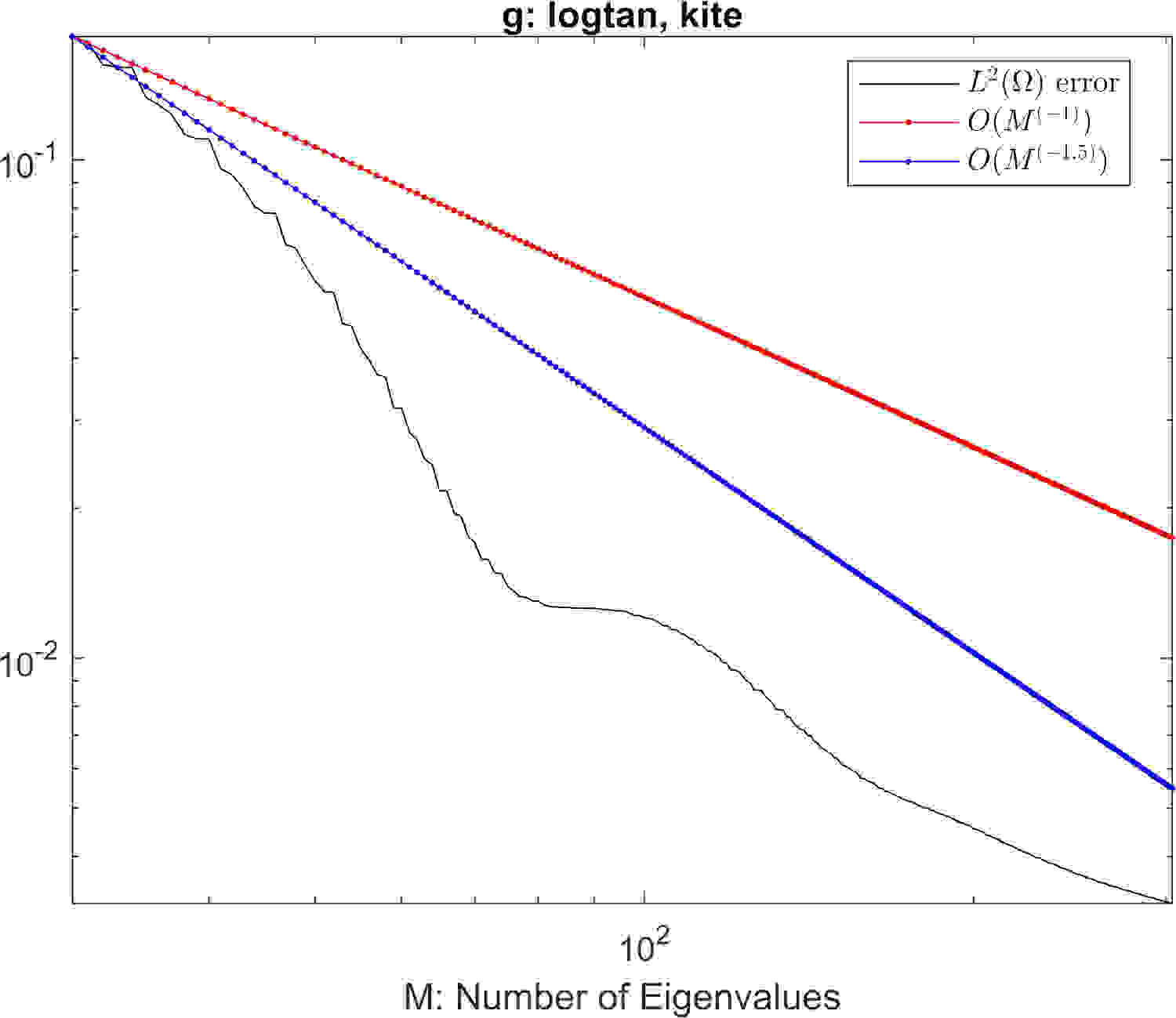}
    \\[\smallskipamount]
    \includegraphics[width=.49\textwidth]{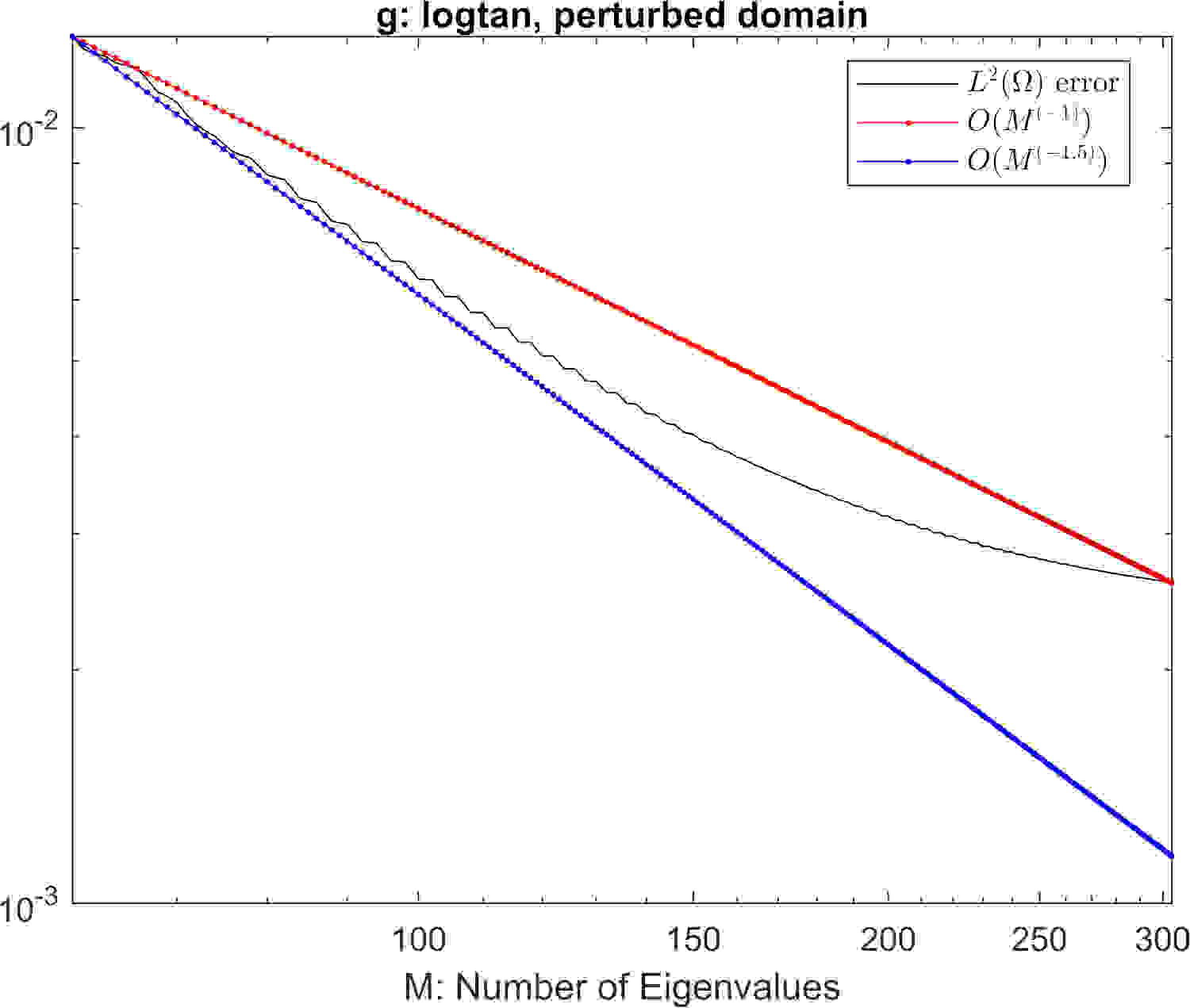}\hfill
    \caption{Error plots for the numerical approximation of the solution to the Robin boundary problem on the boundary to four different domains, the sine mapped domain (Fig. \ref{fig:sinDomain}), the kite shaped domain (Fig. \ref{fig:kiteDomain}), the perturbed circle domain (Fig. \ref{fig:PerturbedDom}) and the square (Fig. \ref{fig:squareDomain}) with a non-smooth boundary.}\label{fig:RobinErrorFigNonSmoothData}
\end{figure}

\begin{figure}
    \includegraphics[width=.49\textwidth]{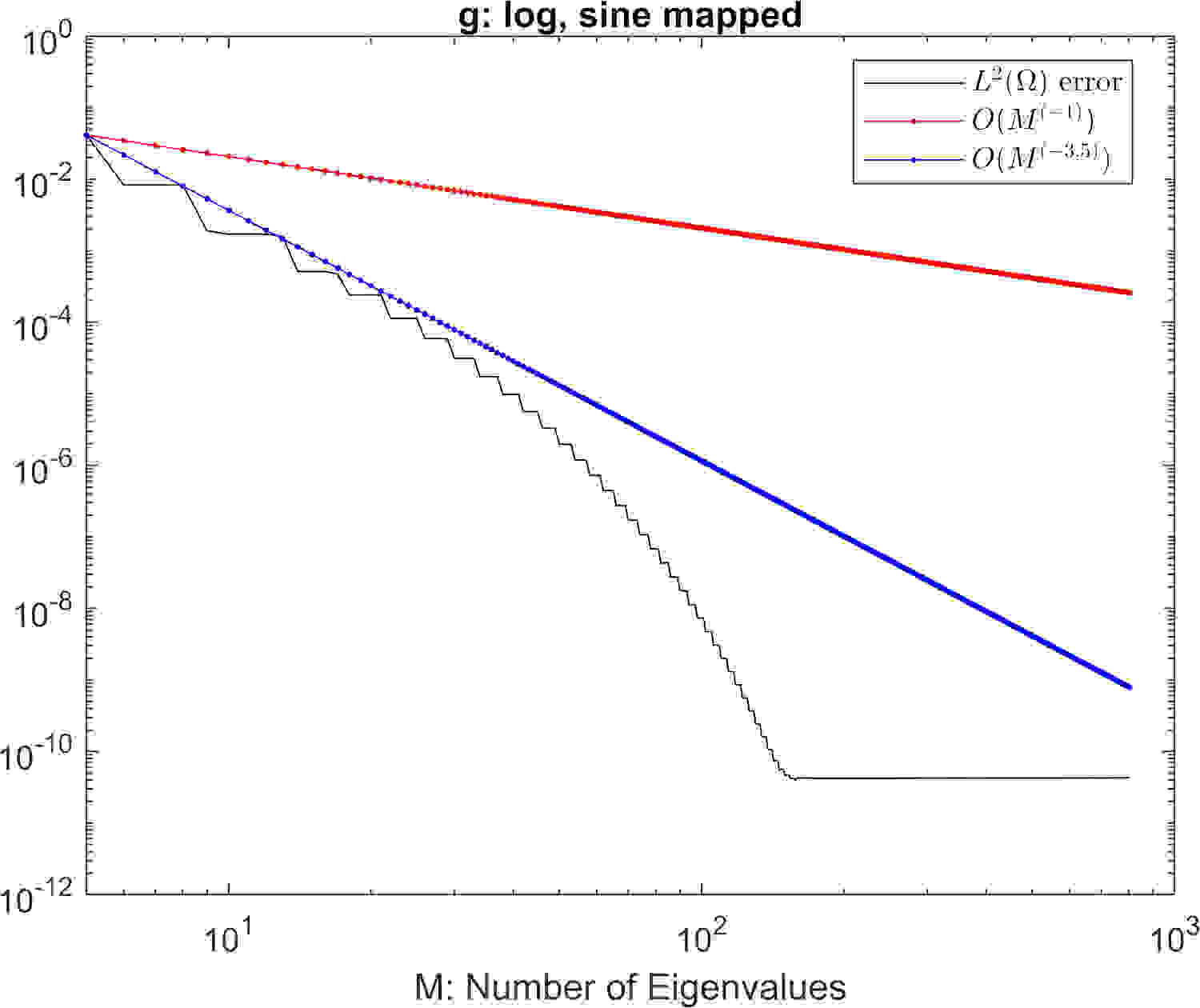}\hfill
    \includegraphics[width=.49\textwidth]{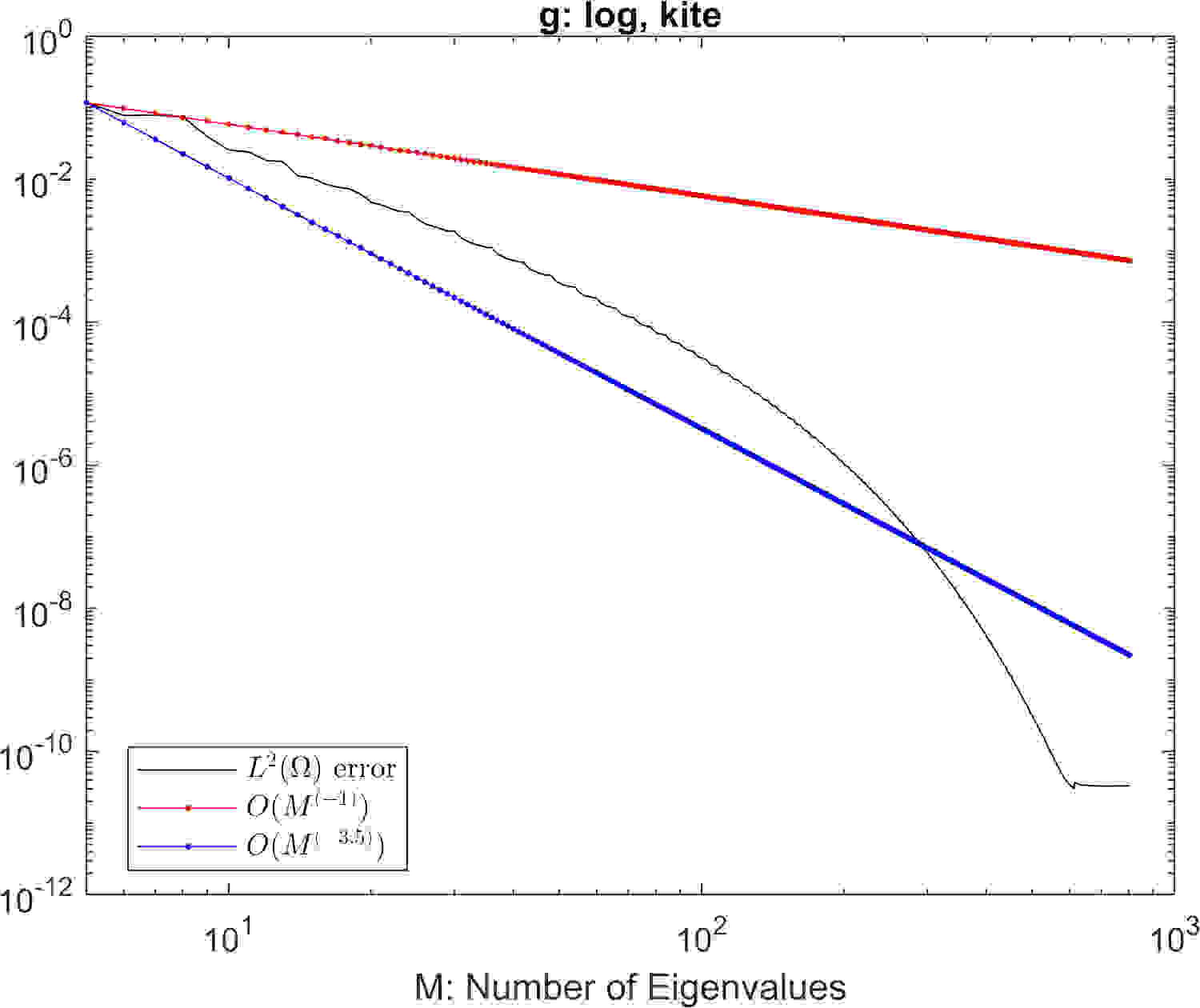}
    \\[\smallskipamount]
    \includegraphics[width=.49\textwidth]{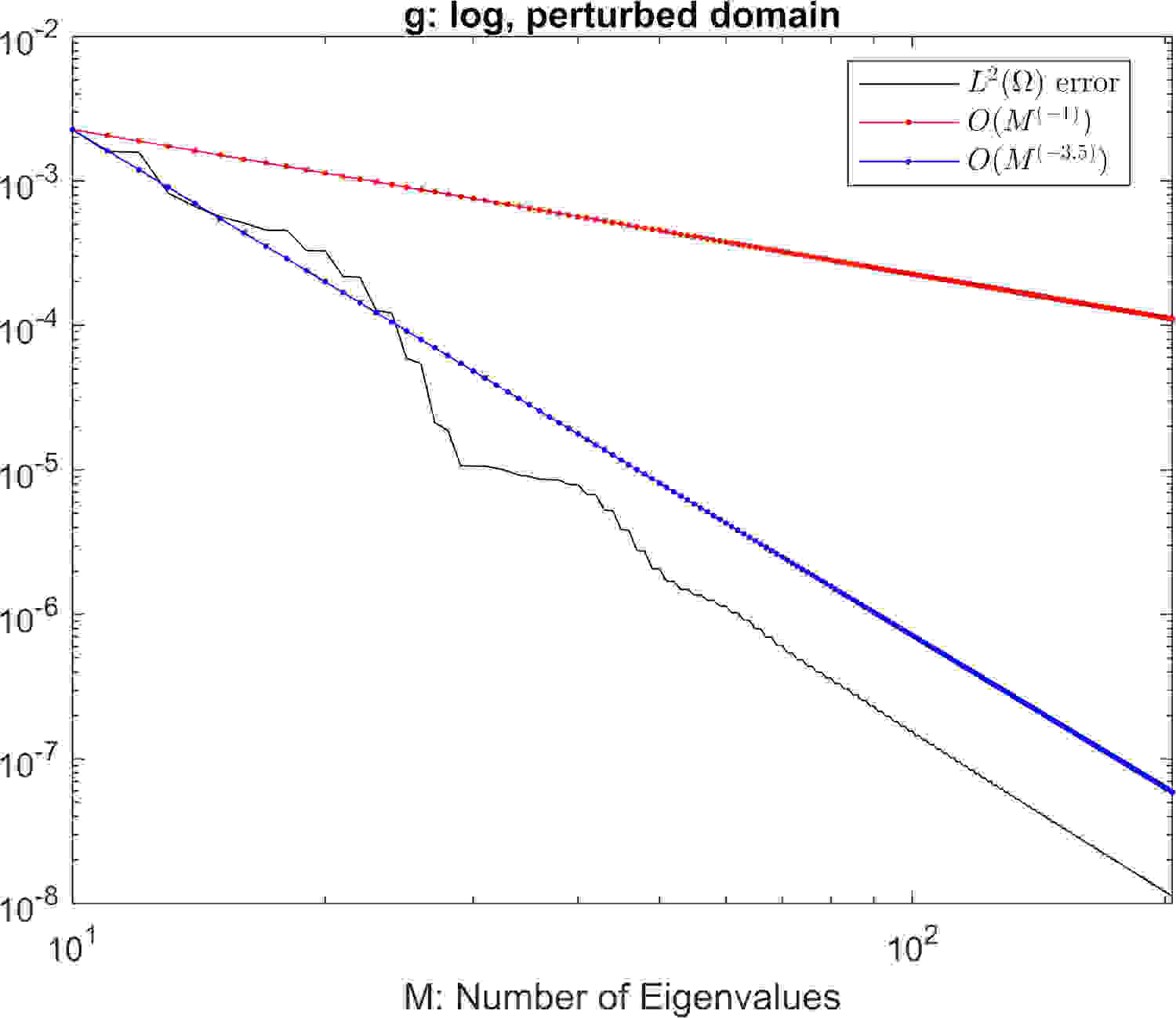}\hfill
    \includegraphics[width=.49\textwidth]{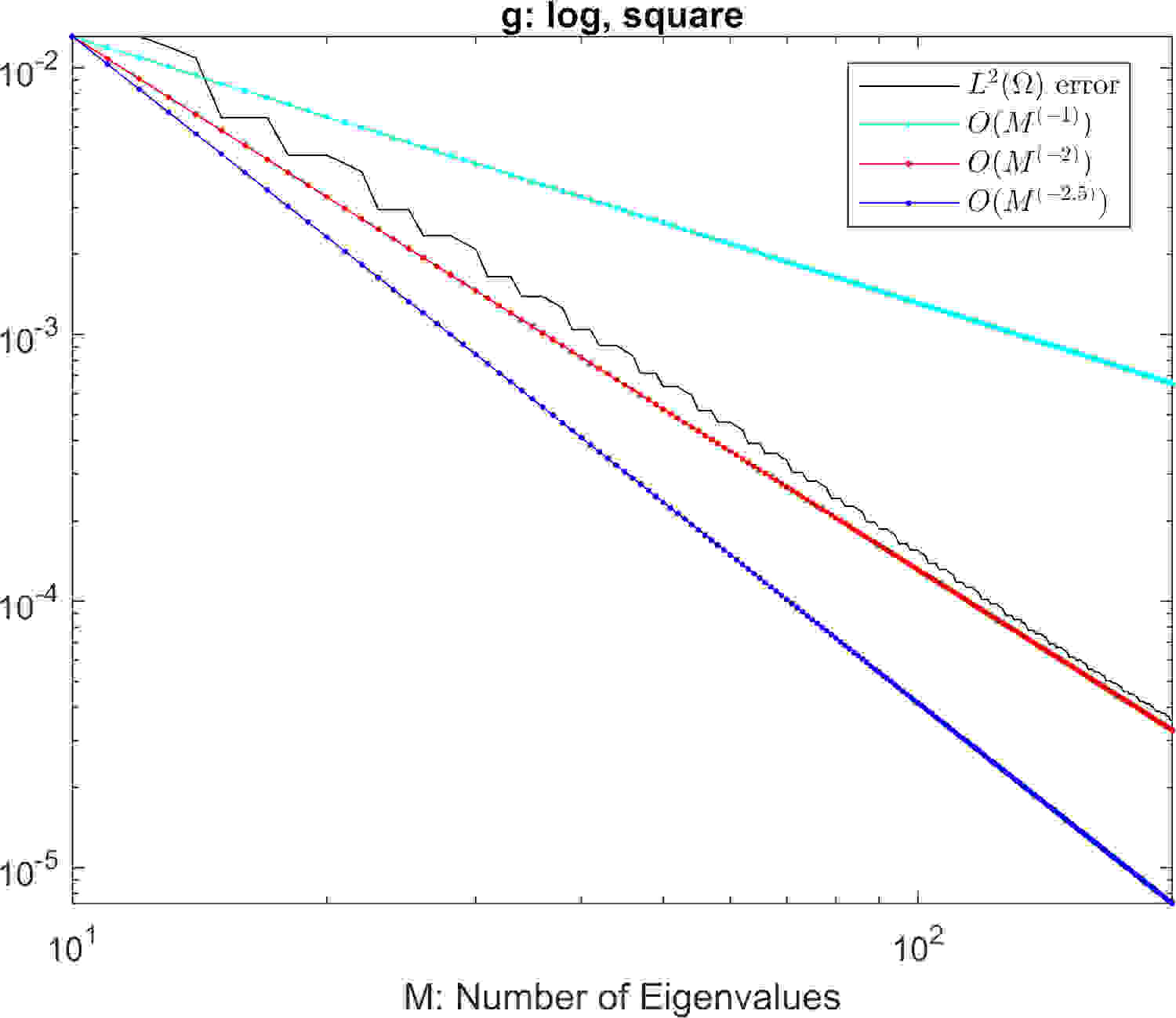}
    \caption{Error plots for the numerical approximation of the solution to the Dirichlet boundary problem on the domain to four different domains, the sine mapped domain (Fig. \ref{fig:sinDomain}), the kite shaped domain (Fig. \ref{fig:kiteDomain}), the perturbed circle domain (Fig. \ref{fig:PerturbedDom}) and the square (Fig. \ref{fig:squareDomain}).}\label{fig:DomDiriErrorFig}
\end{figure}

\begin{figure}
    \includegraphics[width=.49\textwidth]{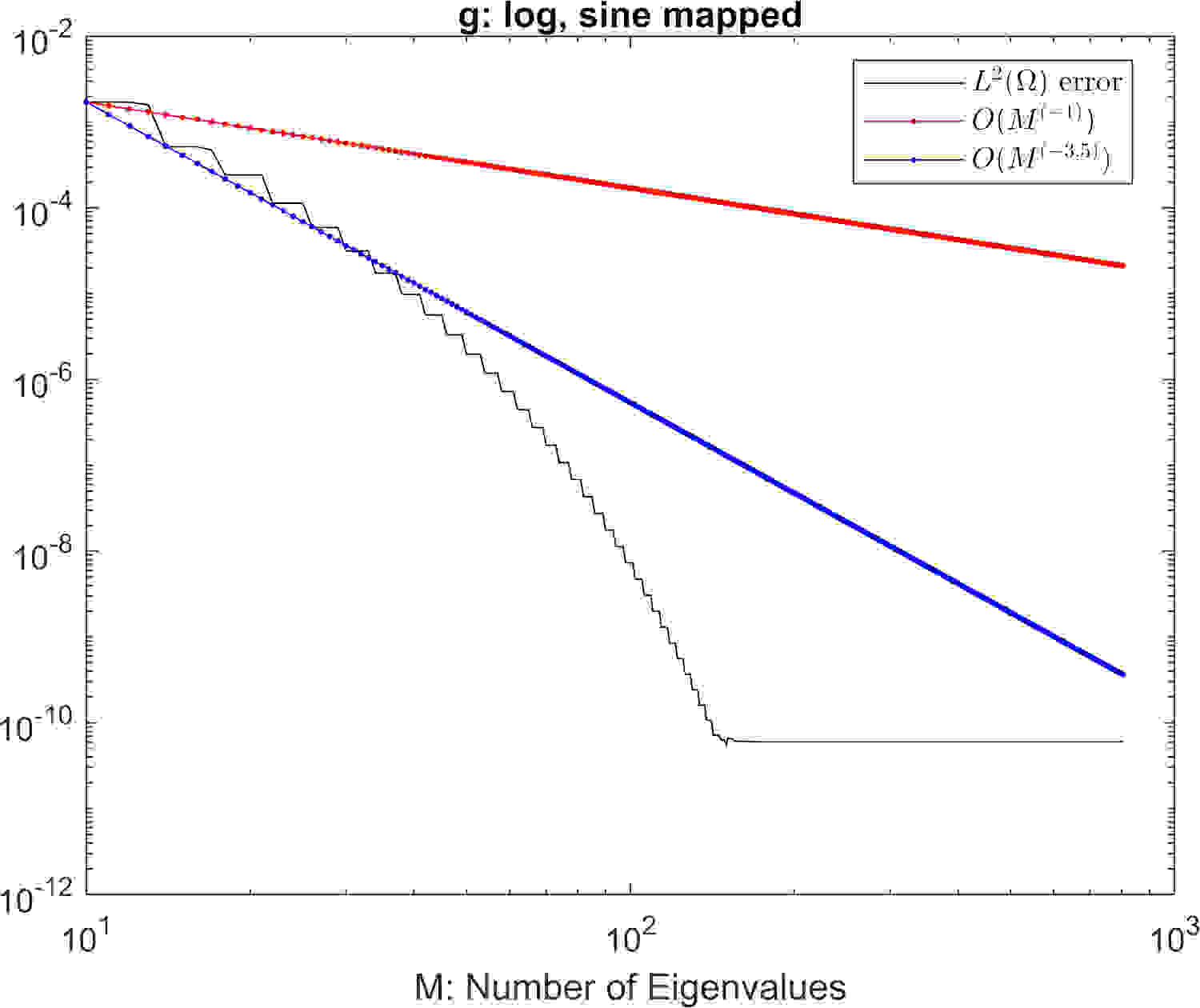}\hfill
    \includegraphics[width=.49\textwidth]{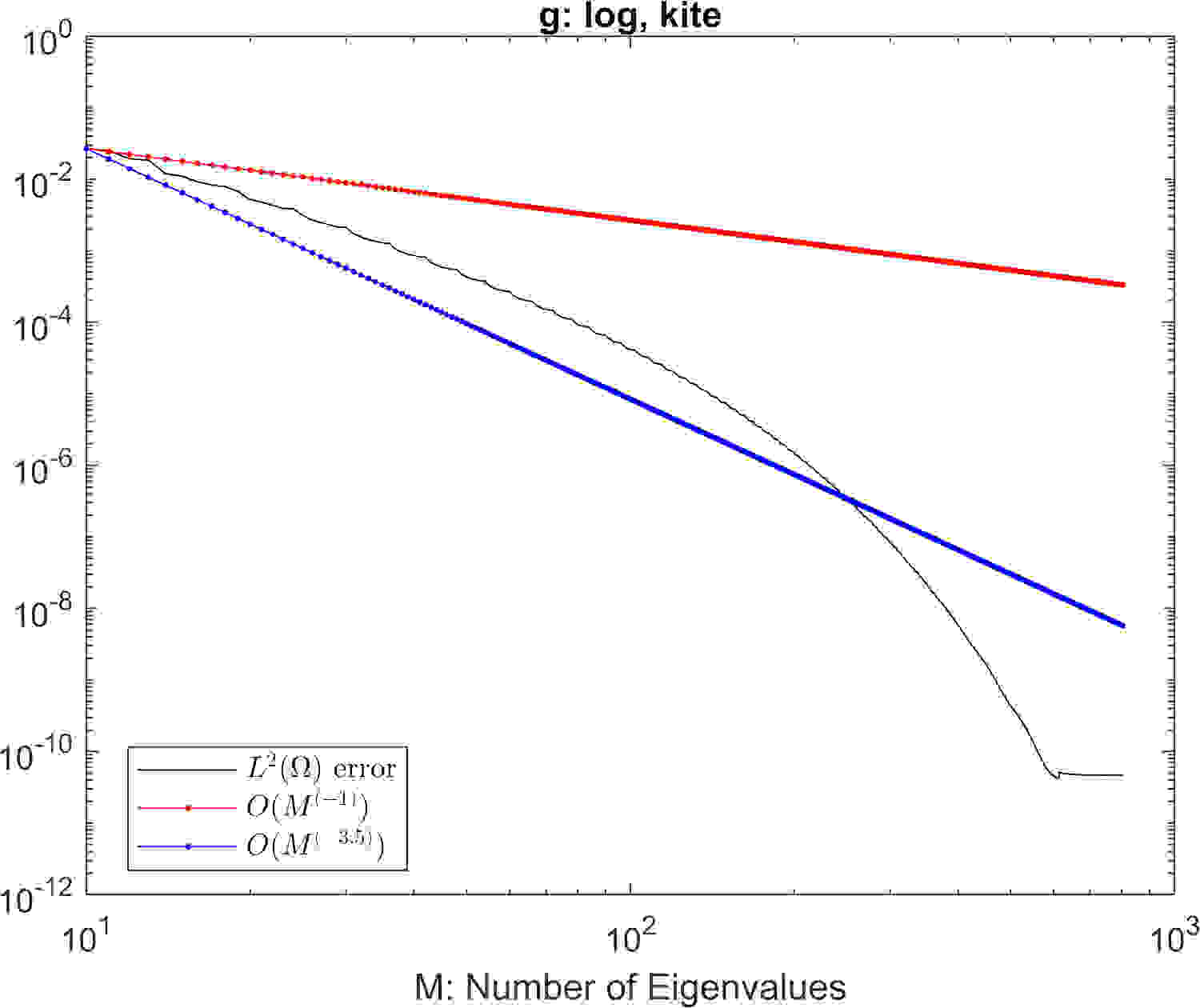}
    \\[\smallskipamount]
    \includegraphics[width=.49\textwidth]{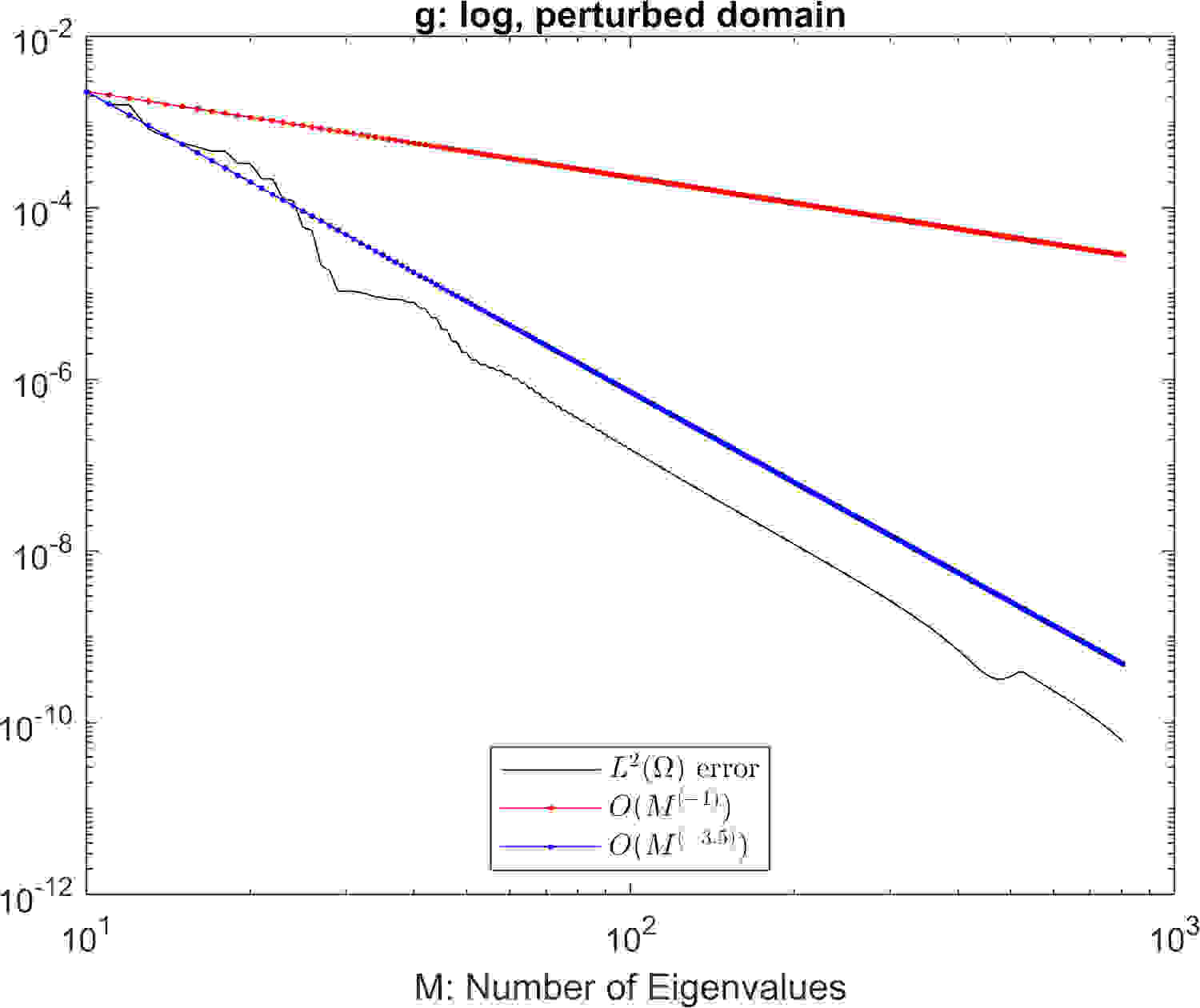}\hfill
    \includegraphics[width=.49\textwidth]{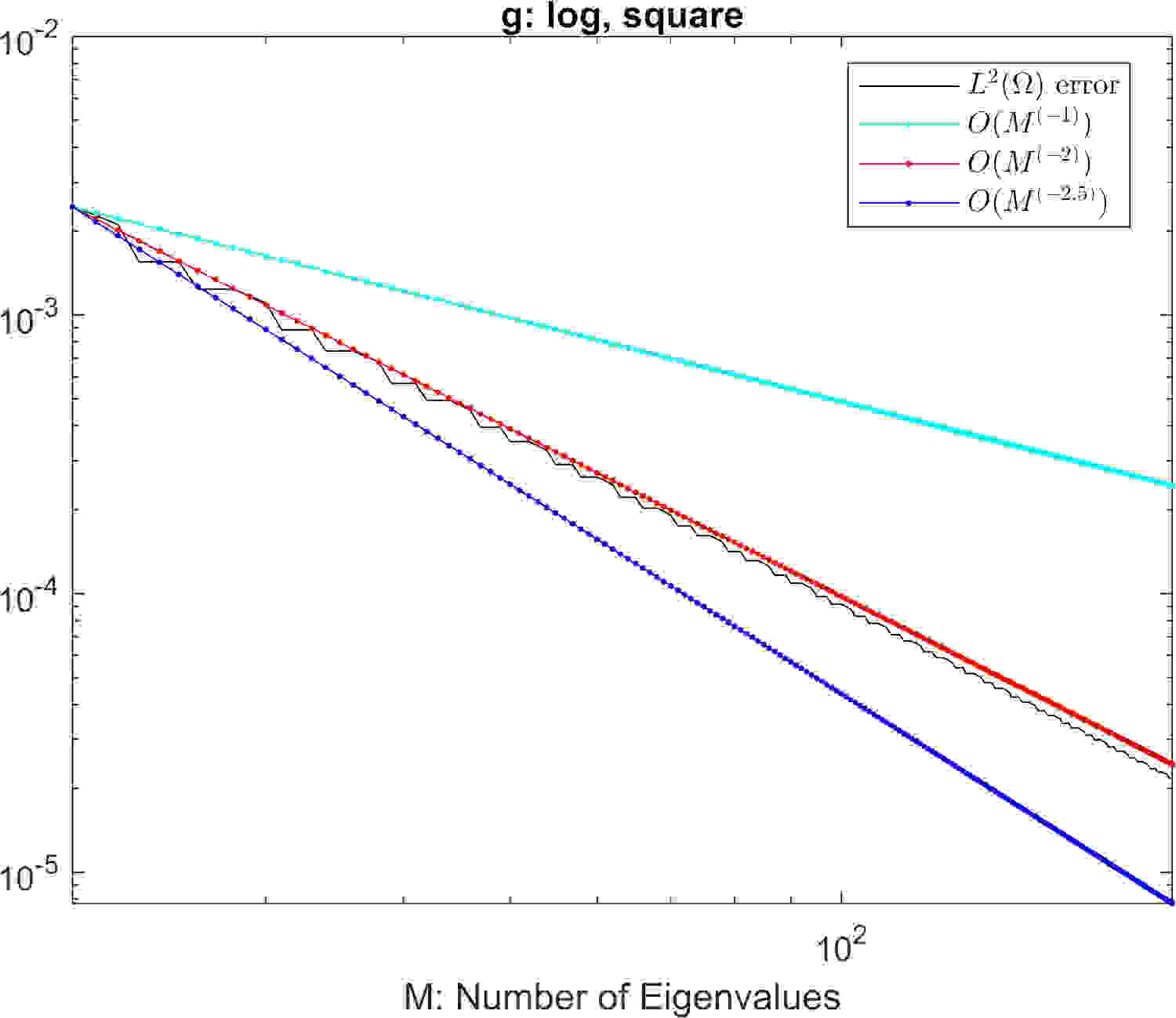}
    \caption{Error plots for the numerical approximation of the solution to the Robin boundary problem on the domain to four different domains, the sine mapped domain (Fig. \ref{fig:sinDomain}), the kite shaped domain (Fig. \ref{fig:kiteDomain}), the perturbed circle domain (Fig. \ref{fig:PerturbedDom}) and the square (Fig. \ref{fig:squareDomain}).}\label{fig:DomRobinErrorFig}
\end{figure}
\begin{figure}
    \includegraphics[width=.49\textwidth]{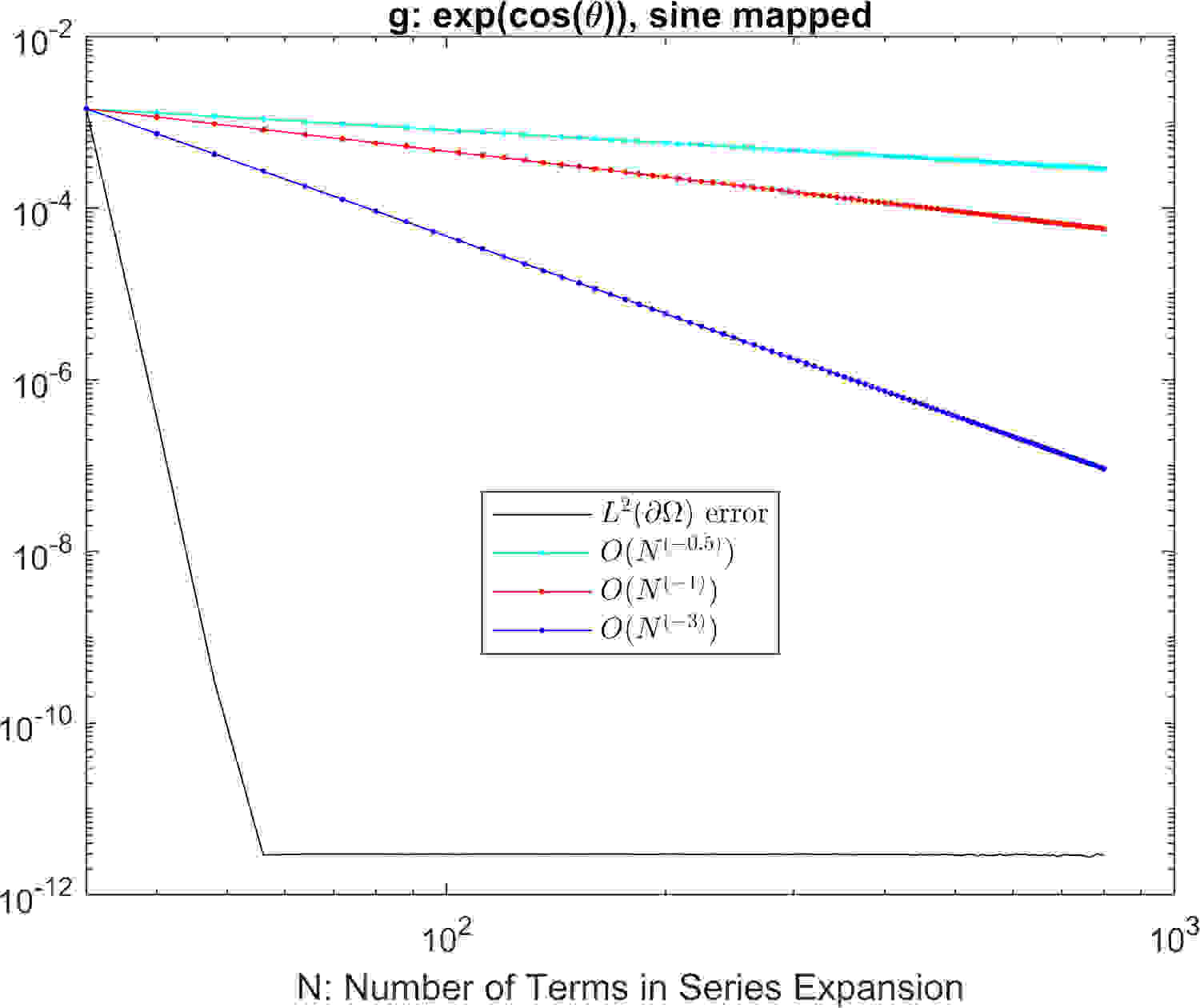}\hfill
    \includegraphics[width=.49\textwidth]{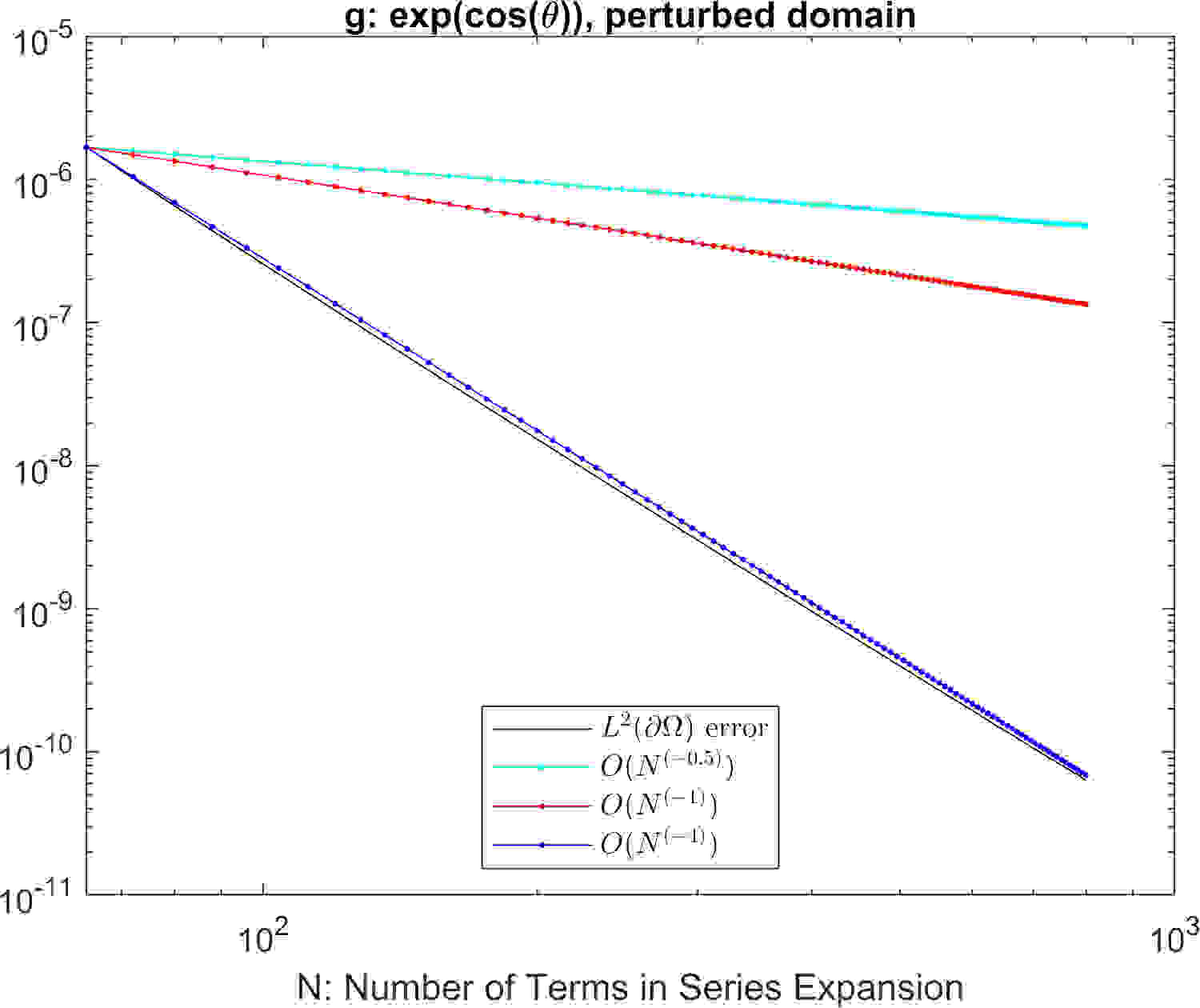}
    \\[\smallskipamount]
    \includegraphics[width=.49\textwidth]{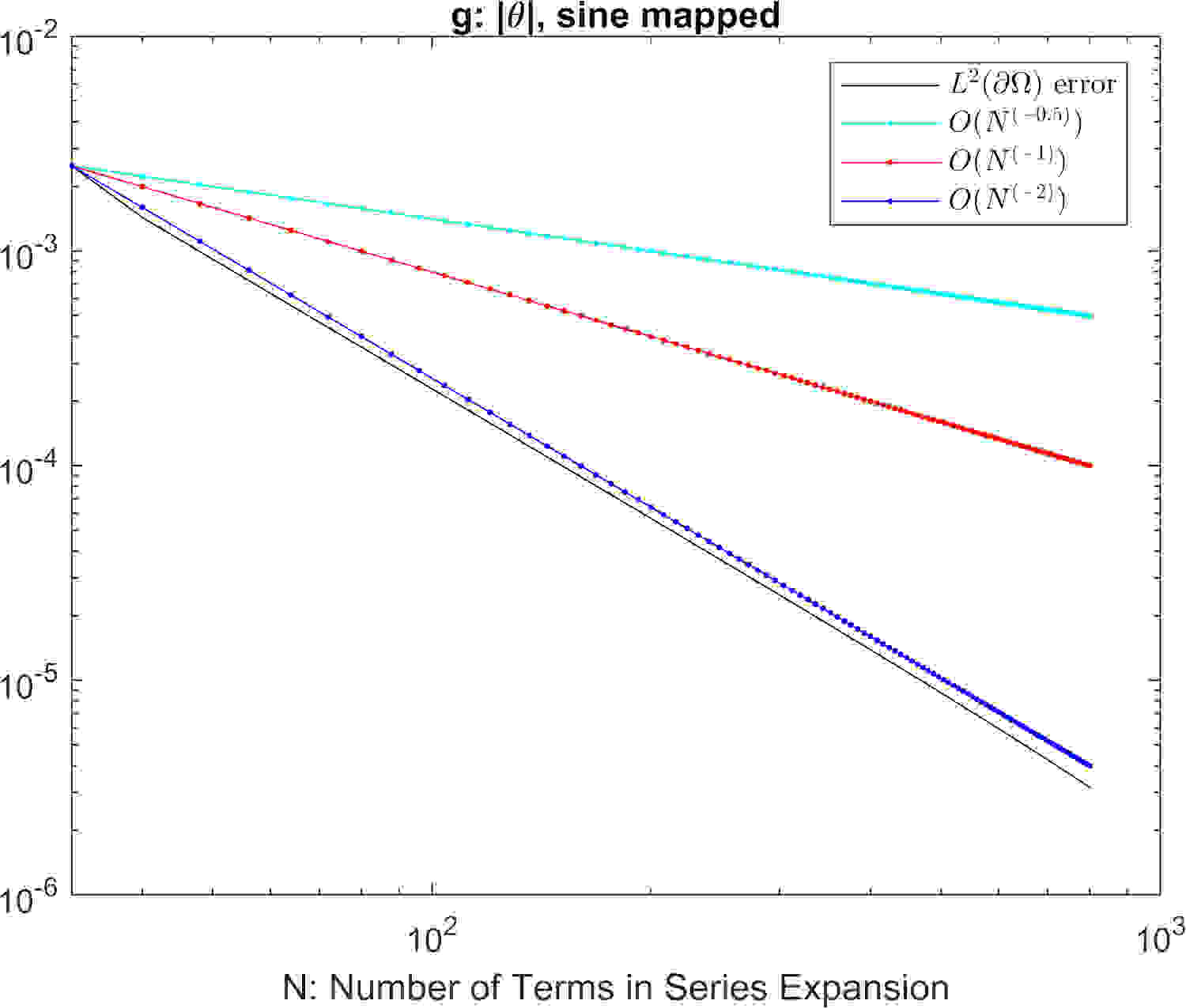}\hfill
    \includegraphics[width=.49\textwidth]{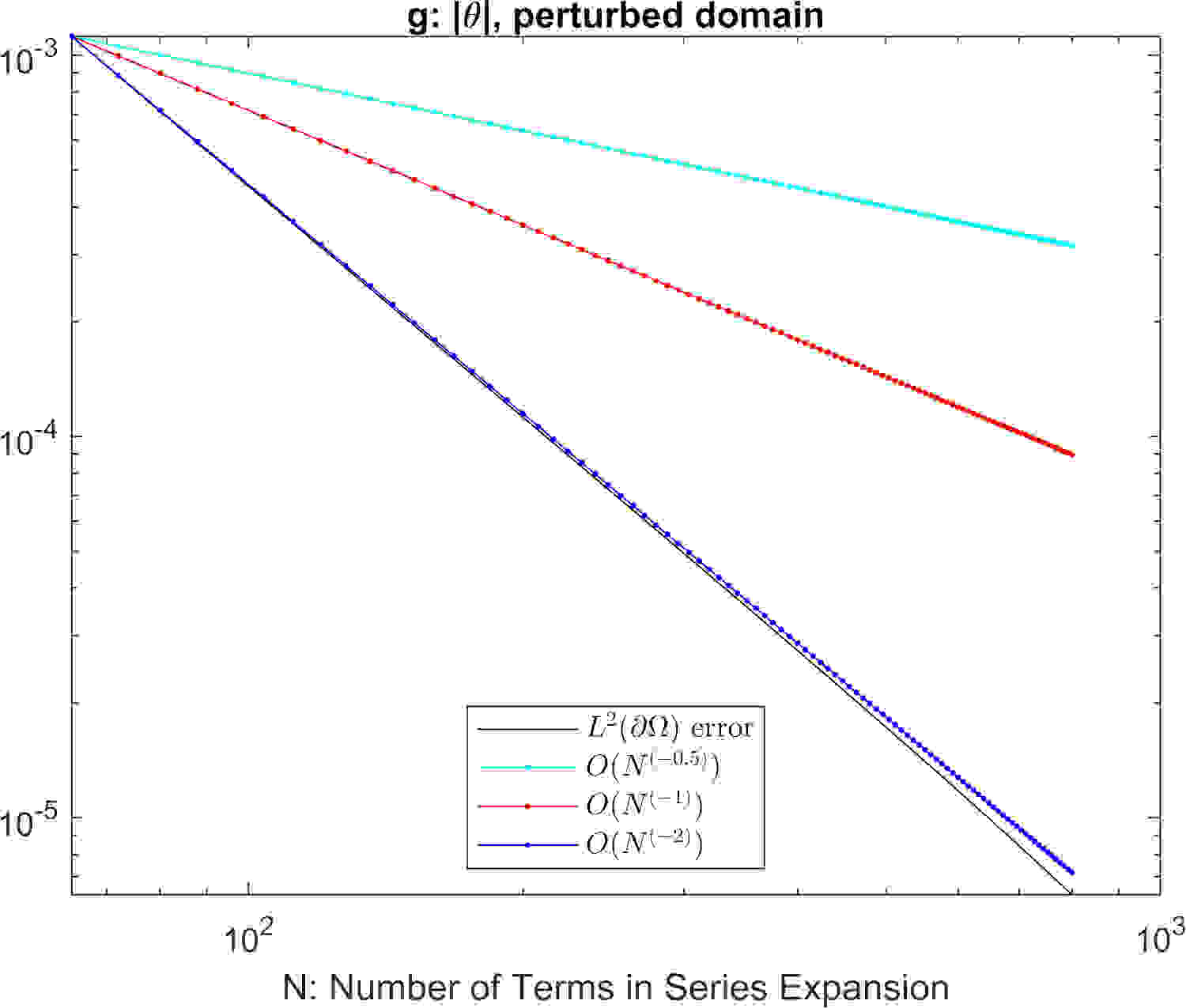}
    \caption{Error plots for the numerical approximation of the solution to the Dirichlet problem on the boundary in $N$, the number of terms in the series expansion. The top left figure describes the function $g(\theta) = \exp(\cos(\theta))$ on the sine-mapped domain, Fig. (\ref{fig:sinDomain}). The top right figure describes the function $g(\theta) = \exp(\cos(\theta))$ on the perturbed circl domain, Fig. (\ref{fig:PerturbedDom}). The bottom left figure describes the function $g(\theta) = |\theta-\pi|$ on the sine-mapped domain, Fig. (\ref{fig:sinDomain}). The bottom right figure describes the function $g(\theta) = |\theta-\pi|$ on the perturbed circle domain, Fig. (\ref{fig:PerturbedDom}).}\label{fig:NErrorDiriBdry}
\end{figure}

\pagebreak


\bibliographystyle{plain}
\bibliography{refs}

\end{document}